\documentclass[a4paper,reqno,11pt]{amsart}

\usepackage{amsmath, amsfonts, amssymb, amsthm, amscd}
\usepackage{graphicx}
\usepackage{psfrag}
\usepackage{perpage}
\usepackage{url}
\usepackage[cmyk,dvipsnames]{xcolor}
\usepackage{mathrsfs}
\usepackage{changes}
\setremarkmarkup{(#2)}

\definecolor{blue}{cmyk}{1,0.6,0,0.06}

\newcommand\red{\textcolor{red}}

\usepackage[utf8]{inputenc}
\usepackage[T1]{fontenc}
\usepackage{microtype}

\usepackage[a4paper,scale={0.72,0.74},marginratio={1:1},footskip=7mm,headsep=10mm]{geometry}

\usepackage{etoolbox}
\makeatletter
\let\ams@starttoc\@starttoc
\makeatother
\edef\restoreparindent{\parindent=\the\parindent\relax}
\usepackage{parskip}
\restoreparindent
\makeatletter
\let\@starttoc\ams@starttoc
\patchcmd{\@starttoc}{\makeatletter}{\makeatletter\parskip\z@}{}{}
\makeatother

\usepackage{hyperref}

\numberwithin{equation}{section}

\newtheorem{theorem}{Theorem}[section]
\newtheorem{lemma}[theorem]{Lemma}
\newtheorem{proposition}[theorem]{Proposition}

\newtheorem{remark}[theorem]{Remark}
\newtheorem{definition}[theorem]{Definition}

\newcommand{\bbE}{{\ensuremath{\mathbb E}} }

\newcommand{\bbN}{{\ensuremath{\mathbb N}} }

\newcommand{\bbP}{{\ensuremath{\mathbb P}} }

\newcommand{\bbR}{{\ensuremath{\mathbb R}} }

\newcommand{\cI}{{\ensuremath{\mathcal I}} }

\newcommand{\ga}{\alpha}
\newcommand{\gb}{\beta}

\newcommand{\gl}{\lambda}

\newcommand\sfK{\mathsf K}

\newcommand\sfR{\mathsf R}
\newcommand\sfS{\mathsf S}

\newcommand\sfg{\mathsf g}

\newcommand\bm{\boldsymbol{m}}

\newcommand\bx{\boldsymbol{x}}

\newcommand\bz{\boldsymbol{z}}

\renewcommand{\tilde}{\widetilde}          
\DeclareMathSymbol{\leqslant}{\mathalpha}{AMSa}{"36} 
\DeclareMathSymbol{\geqslant}{\mathalpha}{AMSa}{"3E} 
\DeclareMathSymbol{\eset}{\mathalpha}{AMSb}{"3F}     
\newcommand{\dd}{\text{\rm d}}             

\newcommand{\R}{\mathbb{R}}
\newcommand{\C}{\mathbb{C}}
\newcommand{\Z}{\mathbb{Z}}
\newcommand{\N}{\mathbb{N}}

\DeclareMathOperator{\sign}{sgn}

\renewcommand{\epsilon}{\varepsilon}
\renewcommand{\theta}{\vartheta}
\renewcommand{\rho}{\varrho}
\renewcommand{\phi}{\varphi}

\usepackage{tikz,tikz-cd}\usepackage{caption,subcaption}
\usepackage{bm,bbm}
\usepackage{mathtools}
\usepackage{subdepth}

\renewcommand{\i}{\mathrm{i}} 
\newcommand{\diff}{\mathop{}\!\mathrm{d}} 
\renewcommand{\P}{\mathbb{P}} 
\newcommand{\E}{\mathbb{E}} 
\newcommand{\1}{\mathbbm{1}} 
\renewcommand{\emptyset}{\varnothing} 
\renewcommand{\hat}{\widehat} 
\newcommand{\arrays}{\mathcal{A}_{\mathcal{I}}}

\DeclareMathOperator{\Flat}{flat}
\DeclareMathOperator{\hFlat}{h-flat}
\DeclareMathOperator{\rFlat}{r-flat}

\newcommand{\fZ}{Z^{\Flat}} 
\newcommand{\hZ}{Z^{\hFlat}} 
\newcommand{\rZ}{Z^{\rFlat}} 

\newcommand{\fPi}{\Pi^{\Flat}} 
\newcommand{\hPi}{\Pi^{\hFlat}} 
\newcommand{\rPi}{\Pi^{\rFlat}} 

\newcommand{\fG}{\Gamma^{\Flat}} 
\newcommand{\hG}{\Gamma^{\hFlat}} 
\newcommand{\rG}{\Gamma^{\rFlat}} 

\newcommand{\fPhi}{\Phi^{\Flat}} 
\newcommand{\hPhi}{\Phi^{\hFlat}} 
\newcommand{\rPhi}{\Phi^{\rFlat}} 

\newcommand{\fT}{\mathcal{T}^{\Flat}} 
\newcommand{\hT}{\mathcal{T}^{\hFlat}} 
\newcommand{\rT}{\mathcal{T}^{\rFlat}} 

\newcommand{\fTau}{\tau^{\Flat}} 
\newcommand{\hTau}{\tau^{\hFlat}} 
\newcommand{\rTau}{\tau^{\rFlat}} 

\newcommand\rsk{\sfR\sfS\sfK} 
\newcommand\grsk{\sfg\sfR\sfS\sfK} 

\DeclareMathOperator{\type}{type}
\DeclareMathOperator{\sym}{sym}
\DeclareMathOperator{\GL}{GL}

\DeclareMathOperator{\GT}{GT}

\let\sp\relax
\DeclareMathOperator{\sp}{sp}
\DeclareMathOperator{\schur}{s}
\DeclareMathOperator{\Pf}{Pf} 

\makeatletter
\DeclarePairedDelimiter\abs{\lvert}{\rvert} 
\DeclarePairedDelimiterX{\norm}[1]{\lVert}{\rVert}{#1} 
\let\oldnorm\norm
\def\norm{\@ifstar{\oldnorm}{\oldnorm*}}
\DeclarePairedDelimiterX{\ceil}[1]{\lceil}{\rceil}{#1} 
\let\oldceil\ceil
\def\ceil{\@ifstar{\oldceil}{\oldceil*}}
\DeclarePairedDelimiterX{\floor}[1]{\lfloor}{\rfloor}{#1}
\let\oldfloor\floor
\def\floor{\@ifstar{\oldfloor}{\oldfloor*}}
\makeatother

\newcommand\llrighttriangle[1][1]{%
\begin{tikzpicture}[scale=#1]
\draw (0,0) -- (0,1) -- (1,0) -- (0,0);
\end{tikzpicture}
}

\newenvironment{myenumerate}{%
\renewcommand{\theenumi}{(\roman{enumi})}%
\renewcommand{\labelenumi}{\theenumi}%
\begin{list}{\labelenumi}
	{%
	\setlength{\itemsep}{0.4em}%
	\setlength{\topsep}{0.5em}%
	\setlength\leftmargin{2.45em}%
	\setlength\labelwidth{2.05em}%
	\setlength{\labelsep}{0.4em}%
	\usecounter{enumi}%
	}%
	}%
{\end{list}
}
\renewenvironment{enumerate}{
\begin{myenumerate}}%
{\end{myenumerate}}

\makeatletter
\newsavebox{\mybox}\newsavebox{\mysim}
\newcommand{\asymptotic}[1]{%
  \savebox{\mybox}{\hbox{\kern3pt$\scriptstyle#1$\kern3pt}}%
  \savebox{\mysim}{\hbox{$\sim$}}%
  \mathbin{\overset{#1}{\kern\z@\resizebox{\wd\mybox}{\ht\mysim}{$\sim$}}}%
}
\makeatother

\makeatletter
\renewcommand{\@secnumfont}{\bfseries}
 
\renewcommand\section{\@startsection{section}{1}%
\z@{.7\linespacing\@plus\linespacing}{.5\linespacing}%
{\large\scshape\bfseries\centering}}

\renewcommand\subsection{\@startsection{subsection}{2}%
  \z@{.5\linespacing\@plus.7\linespacing}{-.5em}%
  {\bfseries\scshape}}
\makeatother

\newenvironment{myitemize}{%
\begin{list}{$\bullet$}%
 	{%
	\setlength{\itemsep}{0.4em}%
	\setlength{\topsep}{0.5em}%
	\setlength\leftmargin{2.45em}%
	\setlength\labelwidth{2.05em}%
	\setlength{\labelsep}{0.4em}%
	}%
	}%
{\end{list}}

\renewenvironment{itemize}{
\begin{myitemize}}%
{\end{myitemize}}

\MakePerPage[2]{footnote} 

\makeatletter

\renewenvironment{proof}[1][\proofname]{\par
\pushQED{\qed}%
\normalfont \topsep4\p@\@plus4\p@\relax
\trivlist
\item[\hskip\labelsep
\bfseries
#1\@addpunct{.}]\ignorespaces
}{%
\popQED\endtrivlist\@endpefalse
}
\makeatother

\makeatletter
\newcommand \Dotfill {\leavevmode \leaders \hb@xt@ 6pt{\hss .\hss }\hfill \kern \z@}
\makeatother

\makeatletter
\def\@tocline#1#2#3#4#5#6#7{\relax
  \ifnum #1>\c@tocdepth 
  \else
    \par \addpenalty\@secpenalty\addvspace{#2}%
    \begingroup \hyphenpenalty\@M
    \@ifempty{#4}{%
      \@tempdima\csname r@tocindent\number#1\endcsname\relax
    }{%
      \@tempdima#4\relax
    }%
    \parindent\z@ \leftskip#3\relax \advance\leftskip\@tempdima\relax
    \rightskip\@pnumwidth plus4em \parfillskip-\@pnumwidth
    #5\leavevmode\hskip-\@tempdima
      \ifcase #1
       \or\or \hskip 1.65em \or \hskip 3.3em \else \hskip 4.95em \fi%
      #6\nobreak\relax
    \Dotfill
    \hbox to\@pnumwidth{\@tocpagenum{#7}}\par
    \nobreak
    \endgroup
  \fi}
\makeatother

\makeatletter
\def\l@section{\@tocline{1}{0pt}{1pc}{}{\scshape}}
\renewcommand{\tocsection}[3]{%
\indentlabel{\@ifnotempty{#2}{\ignorespaces#1 #2.\hskip 0.7em}}#3}
\def\l@subsection{\@tocline{2}{0pt}{1pc}{5pc}{}}

\def\l@subsubsection{\@tocline{3}{0pt}{1pc}{7pc}{}}

\makeatother

\newcommand\gln{\mathfrak{gl}}
\newcommand\son{\mathfrak{so}}

\begin{document}

\title[point-to-line polymer]{Point-to-line polymers
 and orthogonal Whittaker functions}

\begin{abstract}
We study a one dimensional 
directed polymer model in an inverse-gamma random environment, known as the {\it log-gamma polymer}, in three different
geometries: point-to-line, point-to-half-line and when the polymer is restricted to a half-space 
with end point lying free on the corresponding half-line.
Via the use of A.N.Kirillov's geometric Robinson-Schensted-Knuth correspondence, we compute the Laplace transform of the
 partition functions in the above geometries in terms of orthogonal Whittaker functions, thus obtaining new connections between the ubiquitous class of Whittaker functions and exactly solvable probabilistic models. In the case of the first two geometries we also provide
multiple contour integral formulae for the corresponding Laplace transforms. Passing to the zero-temperature limit, we obtain new formulae for the corresponding last passage 
percolation problems with exponential weights.
\end{abstract}

\author[E.~Bisi]{Elia Bisi}
\address{Department of Statistics\\
University of Warwick\\
Coventry CV4 7AL, UK}
\email{E.Bisi@warwick.ac.uk}

\author[N.~Zygouras]{Nikos Zygouras}
\address{Department of Statistics\\
University of Warwick\\
Coventry CV4 7AL, UK}
\email{N.Zygouras@warwick.ac.uk}


\maketitle

\tableofcontents

\addtocounter{section}{0}

\section{Introduction}
Recent efforts in understanding the structure that underlies the Kardar-Parisi-Zhang universality class has led to remarkable connections
between probability, combinatorial structures and representation theoretic objects. Some of the highlights include the solvability of the
Asymmetric Simple Exclusion Process (ASEP) by Tracy and Widom \cite{TW09a, TW09b} via the method of Bethe Ansatz, 
the construction of Macdonald Processes by Borodin and Corwin \cite{BC14} and various particle processes
 ($q$-Totally Asymmetric Simple Exclusion Process, $q$-Totally Asymmetric Zero Range Process etc.) that fall within this scope, 
 the stochastic six-vertex model \cite{BP16},
 the Brownian, semi-discrete polymer and its relation to the Quantum Toda hamiltonian as established by O'Connell \cite{O12}
 and the exactly solvable log-gamma directed polymer, which was introduced by Sepp\"al\"ainen~\cite{Sep12} and analyzed in
 ~\cite{COSZ14, OSZ14, BCR13, NZ15}.

In this article we study further the log-gamma polymer and its structure and we establish new connections to 
the representation theoretic object known as Whittaker functions, which appear in various different places such as mirror symmetry
\cite{Giv97} (see \cite{Lam13} for a review)
 and quantum integrable systems \cite{KL01} and are of central importance in the theory of automorphic forms \cite{Bum89, Gold06}.
  
The log-gamma polymer is defined as follows: On the lattice $\{(i,j) \colon (i,j)\in \bbN^2 \}$ we consider a family of independent 
random variables $\{W_{i,j}\colon (i,j)\in \bbN^2\}$ distributed as inverse-gamma variables
\begin{align*}
\bbP(W_{i,j}\in \dd w_{i,j}) = \frac{1}{\Gamma(\gamma_{i,j})} w^{-\gamma_{i,j}} e^{-1/w_{i,j}} \,\frac{\dd w_{i,j}}{w_{i,j}} \, ,\qquad (i,j)\in \bbN^2 \, ,
\end{align*}
with $\gamma_{i,j}$ positive parameters. 
For $(p,q)\in \bbN^2$ fixed, denote by $\Pi_{p,q}$ the set of all directed, nearest neighbor paths from $(1,1)$ to $(p,q)$, called \emph{polymer paths}. The \emph{log-gamma polymer measure} gives to every such path $\pi$ a weight
\begin{align*}
\frac{1}{Z_{p,q}} \prod_{(i,j)\in \pi} W_{i,j} \, ,
\end{align*}
where the normalization
\begin{align}\label{p2p_partition}
Z_{p,q}:=\sum_{\pi\in \Pi_{p,q}} \prod_{(i,j)\in \pi} W_{i,j}
\end{align}
 is called the {\it point-to-point partition function} of the log-gamma polymer.
The characterization {\it point-to-point} is due to the fact that only  paths that start at a fixed point $(1,1)$ and end at a fixed point $(p,q)$ are considered. 
The partition function $Z_{p,q}$ can be viewed as a discrete version of the solution to the Stochastic Heat Equation (SHE) in dimension one
\begin{align}\label{SHE}
\partial_t u =\frac{1}{2}\partial_{xx}^2 u + \dot{W}(t,x) \,u \, ,
\end{align}
with delta initial condition at zero, where $\dot{W}(t,x)$ is space-time white noise.
 The transformation $h=\log u$ of the solution to \eqref{SHE} leads, then,  to the solution
of the KPZ equation
\begin{align*}
\partial_t h =\frac{1}{2}\partial_{xx}^2 h +\frac{1}{2} (\partial_x h)^2+ \dot{W}(t,x) \, ,
\end{align*} 
 under the so-called {\it narrow wedge} initial condition. See \cite{Sep12} for the introduction of log-gamma polymer, as well
 as the more general, recent review on random polymers \cite{Com17}.
 
Using ideas from algebraic combinatorics and in particular A.N.Kirillov's {\it geometric} Robinson-Schensted-Knuth correspondence \cite{K01}
(see also \cite{NY04})\footnote{Originally named {\it tropical} by Kirillov in his article ``Introduction to Tropical Combinatorics'' \cite{K01}.}, \cite{COSZ14, OSZ14} were able to determine the Laplace transform
 of the point-to-point partition function for a log-gamma polymer with parameters $\gamma_{i,j}=\ga_{i}+\gb_{j}$ and make a connection
 to $GL_n(\R)$-Whittaker functions (an earlier connection between the Brownian, semi-discrete polymer and 
 $GL_n(\R)$-Whittaker functions appeared in \cite{O12}).
 In particular, it was established that
  \begin{align}\label{p2p_gln}
 \bbE\Big[e^{-r Z_{n,n}}\Big]=\frac{1}{\prod_{1\leq i,j\leq n}\Gamma(\ga_i+\gb_j)}
 \int_{\R_+^n} e^{-rx_1-1/x_n} \Psi^{\mathfrak{gl}_n}_{\bm{\ga}}(\bx) \Psi^{\mathfrak{gl}_n}_{\bm{\gb}}(\bx) 
 \prod_{i=1}^n\frac{\dd x_i}{x_i} \, ,
 \end{align}
 where $\bx=(x_1,\dots,x_n)$, $\bm{\ga}=(\ga_1,\dots,\ga_n)$, $\bm{\gb}=(\gb_1,\dots,\gb_n)$ and $\Psi^{\mathfrak{gl}_n}_{\bm{\ga}}(\bx)$ are the 
 $GL_n(\R)$-Whittaker functions (see section \ref{sec:glWhittakerFns} for details)\footnote{Even though Whittaker functions are typically associated to Lie groups, for notational convenience we will be mostly using mathfrak 
symbols, $\mathfrak{gl}_n, \mathfrak{so}_{2n+1}$ etc., which are
 normally used for Lie algebras.}. Furthermore, using the Plancherel theory 
 for $GL_n(\R)$-Whittaker functions a contour integral formula was derived for \eqref{p2p_gln},
  which was subsequently turned into a Fredholm determinant
 in \cite{BCR13} allowing to derive the Tracy-Widom GUE asymptotics.
 
In this work we consider the directed polymer with inverse gamma disorder but with the end point lying free on a line. In particular, we will
consider three different geometries of paths (see also Figure \ref{fig:directedPath} for a graphical representation): 
\begin{itemize}
\item[(i)] {\it point-to-line} (or {\it flat}), which we denote $\fPi_{n} $ and consider to be the
set of all directed paths from $(1,1)$ to the line $\{(i,j)\colon i+j = n+1\}$. The corresponding partition function is then
\begin{align}\label{eq:flatPartitionFn}
\fZ_{n}:=\sum_{\pi\in \fPi_{n}} \prod_{(i,j)\in \pi} W_{i,j} \, .
\end{align}
\item[(ii)] 
{\it point-to-half-line} (or {\it half-flat}), which we denote $\hPi_{n} $ and consider to be the
set of all directed paths from $(1,1)$ to the halfline $\{(i,j)\colon i+j=n+1, \,i \leq j\}$.
 The corresponding partition function is then
\begin{align}\label{eq:hFlatPartitionFn}
\hZ_{n}:=\sum_{\pi\in \hPi_{n}} \prod_{(i,j)\in \pi} W_{i,j} \, .
\end{align}
\item[(iii)] 
{\it restricted point-to-half-line} (or {\it restricted half-flat}), which we denote $\rPi_{n} $ and consider to be the
set of all directed paths from $(1,1)$ to the half-line $\{(i,j)\colon i+j = n+1, \,i \leq j\}$ and with paths $\pi\in\rPi_{n}$ restricted to
the half-space  $\{(i,j)\colon i \leq j\}$. The corresponding partition function is then
\begin{align}\label{eq:rFlatPartitionFn}
\rZ_{n}:=\sum_{\pi\in \rPi_{n}} \prod_{(i,j)\in \pi} W_{i,j} \, .
\end{align}
\end{itemize}
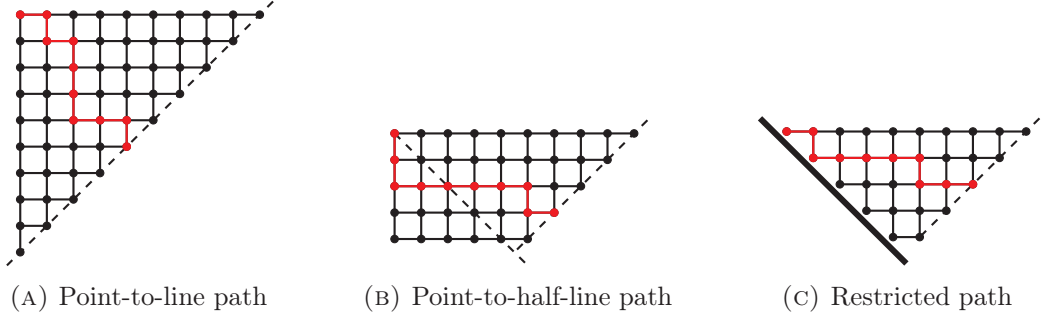
\begin{figure}
\centering
\begin{minipage}[b]{.33\linewidth}
\centering
\begin{tikzpicture}[scale=0.35]
\draw[thick,dashed] (10.5,-.5) -- (.5,-10.5);
\foreach \i in {1,...,10}{
	\draw[thick] (1,-\i) grid (11-\i,-\i);
	\draw[thick] (\i,-1) grid (\i,-11+\i);
	
		\foreach \j in {\i,...,10}{
		\node[draw,circle,inner sep=1pt,fill] at (11-\j,-\i) {};
	}
}

\draw[thick,color=red,-] (1,-1) -- (2,-1) -- (2,-2) -- (3,-2) -- (3,-3) -- (3,-4) -- (3,-5) -- (4,-5) -- (5,-5) -- (5,-6);
\foreach \x in {(1,-1),(2,-1),(2,-2),(3,-2),(3,-3),(3,-4),(3,-5),(4,-5),(5,-5),(5,-6)}{
	\node[draw,circle,inner sep=1pt,fill,red] at \x {};
	}
\end{tikzpicture}
\subcaption{Point-to-line path}
\label{subfig:P2Lpath}
\end{minipage}%
\begin{minipage}[b]{.33\linewidth}
\centering
\begin{tikzpicture}[scale=0.35]
\draw[thick,dashed] (10.5,-.5) -- (5.5,-5.5); \draw[thick,dashed] (1,-1) -- (6,-6);
\foreach \i in {1,...,5}{
	\draw[thick] (1,-\i) grid (11-\i,-\i);
	
		\foreach \j in {\i,...,10}{
		\node[draw,circle,inner sep=1pt,fill] at (11-\j,-\i) {};
	}
}

\foreach \j in {1,...,5}{
	\draw[thick] (\j,-1) grid (\j,-5);
	\draw[thick] (5+\j,-1) grid (5+\j,-6+\j);
}

\draw[thick,color=red,-] (1,-1) -- (1,-2) -- (1,-3) -- (2,-3) -- (3,-3) -- (4,-3) -- (5,-3) -- (6,-3) -- (6,-4) -- (7,-4);
\foreach \x in {(1,-1),(1,-2),(1,-3),(2,-3),(3,-3),(4,-3),(5,-3),(6,-3),(6,-4),(7,-4)}{
	\node[draw,circle,inner sep=1pt,fill,red] at \x {};
	}
\end{tikzpicture}
\subcaption{Point-to-half-line path}
\label{subfig:P2half-Lpath}
\end{minipage}%
\begin{minipage}[b]{.33\linewidth}
\centering
\begin{tikzpicture}[scale=0.35]
\draw[line width=0.8mm] (0,-.5) -- (5.5,-6);
\draw[thick,dashed] (10.5,-.5) -- (6,-5);
\foreach \i in {1,...,5}{
	\draw[thick] (\i,-\i) grid (11-\i,-\i);
	
		\foreach \j in {\i,...,5}{
		\node[draw,circle,inner sep=1pt,fill] at (\j,-\i) {};
		\node[draw,circle,inner sep=1pt,fill] at (11-\j,-\i) {};
	}
}

\foreach \j in {1,...,5}{
	\draw[thick] (\j,-1) grid (\j,-\j);
	\draw[thick] (5+\j,-1) grid (5+\j,-6+\j);
}

\draw[thick,color=red,-] (1,-1) -- (2,-1) -- (2,-2) -- (3,-2) -- (4,-2) -- (5,-2) -- (6,-2) -- (6,-3) -- (7,-3) -- (8,-3);
\foreach \x in {(1,-1),(2,-1),(2,-2),(3,-2),(4,-2),(5,-2),(6,-2),(6,-3),(7,-3),(8,-3)}{
	\node[draw,circle,inner sep=1pt,fill,red] at \x {};
	}
\end{tikzpicture}
\subcaption{Restricted path}
\label{subfig:restrictP2Lpath}
\end{minipage}
\caption{Directed paths in $\N^2$ of length $10$ from the point $(1,1)$ to the line $i+j-1=10$. The three paths, highlighted in red, correspond to three different geometries, as specified. The picture is rotated by $90^{\circ}$ clockwise w.r.t.\ the Cartesian coordinate system, to adapt it to the usual matrix/array indexing.}
\label{fig:directedPath}
\end{figure}

We will compute the Laplace transform of the
above partition functions in terms of Whittaker functions corresponding to the orthogonal group $SO_{2n+1}(\R)$, which
we will denote by $\Psi^{\mathfrak{so}_{2n+1}}_{\bm{\ga}}(\bx)$, with $\bm{\ga}=(\ga_1,\dots,\ga_n)$ and $\bx=(x_1,\dots,x_n)$ (see Section
\ref{sec:orthogonalWhittakerFns} for details), as well as $GL_n(\R)$-Whittaker functions $\Psi^{\mathfrak{gl}_n}_{\bm{\ga}}(\bx)$.
 More precisely, after choosing appropriately the parameters $\gamma_{i,j}$ of the inverse gamma variables, we obtain that
  \begin{align}
\bbE\Big[e^{-r \fZ_{\,2n}}\Big]&=\frac{r^{\sum_{i=1}^n(\ga_i+\gb_i)}}{\fG_{\bm{\alpha},\bm{\beta}}}
   \int_{\R_+^n} e^{-rx_1} \Psi^{\mathfrak{so}_{2n+1}}_{\bm{\ga}}(\bx) \Psi^{\mathfrak{so}_{2n+1}}_{\bm{\gb}}(\bx) 
   \prod_{i=1}^n\frac{\dd x_i}{x_i} \, , \label{p2line_gen1}\\
\E\Big[e^{- r \hZ_{\,2n} }\Big] 
   &= \frac{r^{\sum_{i=1}^{n} (\alpha_i+\beta_i)}}
   {\hG_{\bm{\alpha},\bm{\beta}}}
   \int_{\R_{+}^n} e^{-r x_1}
   \Psi_{\bm{\alpha}}^{\mathfrak{so}_{2n+1}}(\bm{x})
   \Psi_{\bm{\beta}}^{\mathfrak{gl}_n}(\bm{x})
   \prod_{i=1}^n \frac{\diff x_i}{x_i} \, , \qquad \text{and} \label{p2line_gen2}\\
\E\Big[e^{-r  \rZ_{\,2n} }\Big] 
   &= \frac{r^{\sum_{i=1}^{n} \alpha_i }}
   {\rG_{\bm{\alpha}}}
   \int_{\R_{+}^n} e^{-r x_1}
   \Psi_{\bm{\alpha}}^{\mathfrak{so}_{2n+1}}(\bm{x})
   \prod_{i=1}^n \frac{\diff x_i}{x_i} \, , \label{p2line_gen3}
 \end{align}
where $\fG_{\bm{\alpha},\bm{\beta}}$, $\hG_{\bm{\alpha},\bm{\beta}}$ and $\rG_{\bm{\alpha}}$ are suitable normalization constants.
It is interesting to note the structure of these formulae in comparison to formula \eqref{p2p_gln} for the point-to-point polymer. Informally,
one could say that ``opening'' each part of the end point's ``wedge'' to a (diagonally) flat
part corresponds to replacing a $GL_n(\R)$-Whittaker function
with an $SO_{2n+1}(\R)$-Whittaker function. However, a priori there is no obvious reason why this analogy should take place.
  
 Whittaker functions associated to general Lie groups have already appeared in probability in terms of describing the law of Brownian motion
 on these Lie groups conditioned on certain exponential functionals \cite{BO11}, \cite{Ch13}. The latter reference provides
 also  a comprehensive 
 account of various algebraic properties and origins of Whittaker functions. Further extensions in this direction, in both the {\it ``Archimedean'' and
 ``non-Archimedean'' cases}  have been achieved in \cite{Ch15, Ch16}. In \cite{Nte17} $SO_{2n+1}(\R)$-Whittaker functions
 also emerged in the description of the Markovian dynamics of systems of interacting particles restricted by a {\it soft} wall.
 In our setting, $SO_{2n+1}(\R)$-Whittaker functions emerge through a combinatorial analysis of the log-gamma polymer
 via the geometric Robinson-Schensted-Knuth correspondence. 
 Using the framework and properties of geometric $\rsk$ as in \cite{NZ15}, we 
  determine the joint law of all point-to-point partition functions with end point on a line or half line. Subsequently,
   we are able to derive an integral formula for the Laplace transform of the
  various point-to-line partition functions after expressing them as a sum of the corresponding point-to-point ones. 
  These formulae do not
  immediately relate to Whittaker functions but do so after a change of variables and appropriate decompositions
  of the integrals, thus leading to formulae \eqref{p2line_gen1}, \eqref{p2line_gen2}, \eqref{p2line_gen3}.
   Even though simple, the alluded change of variables is remarkable in the sense that it precisely couples the 
   structure of orthogonal Whittaker functions with the combinatorial structure of the point-to-line polymers and their Laplace transforms.
   We should note that we would probably not be able to
  relate to Whittaker functions, had we aimed to compute functionals of the point-to-line partition other than the Laplace transform.
Nevertheless, the Laplace transform is the most relevant functional as it determines the distribution.
Once \eqref{p2line_gen1}, \eqref{p2line_gen2}, \eqref{p2line_gen3} are obtained we go one step further by rewriting the first two of them
as contour integrals involving Gamma functions
(even though we can also formally write \eqref{p2line_gen3} as a contour integral, we miss the necessary
estimates that would fully justify such a representation).
 This is done via the use of Plancherel theory for $GL_n$-Whittaker functions and 
special integral identities of products of $GL_n\times GL_n$ and $GL_n\times SO_{2n+1}$ Whittaker functions due to
Bump-Stade \cite{Bum84, Sta02} and Ishii-Stade \cite{IS13}, respectively. Such integral identities are important in number
theory as they lead to functional equations for $L$-series, facilitating the study of their zeros \cite{Bum89}.
Currently, there do not exist such integral identities for
products of $GL_n\times SO_{2n}$ and that is why we restrict our presentation to polymers $Z_{2n}$ of even length, even though
our combinatorial analysis allows one to write the corresponding formulae \eqref{p2line_gen1}, \eqref{p2line_gen2}, \eqref{p2line_gen3}
for polymers of odd length, as well, in terms of $SO_{2n}(\R)-$Whittaker functions.
 
 Calabrese and Le Doussal studied in \cite{CLeD11, CLeD12} the {\it continuum random polymer} with flat initial conditions. Via the
 non-rigorous approach of Bethe ansatz for the Lieb-Liniger model and the {\it replica trick}, they exhibited that its Laplace transform can
 be written in terms of a Fredholm Pfaffian from which the Tracy-Widom GOE asymptotics are derived. In their method they had first to
 derive a series representation for the half-flat initial condition and then the flat case was deduced from the former via a suitable limit.
 More recently, \cite{G17} applied the method of Calabrese and Le Doussal and of 
 Thiery and Le Doussal \cite{TLeD14} to study, again at non-rigorous level,
 the Laplace transform of the log-gamma polymer with
 end point lying free on a line.
Ortmann-Quastel-Remenik \cite{OQR16, OQR17} have made a number of steps in the approach of Calabrese and Le Doussal
rigorous in the case of the Asymmetric Exclusion Process 
(see also the earlier work of Lee \cite{Lee10}). In the half-flat case \cite{OQR16}
 derived a series formula for the $q$-deformed Laplace
transform of the height function. Formal asymptotics on this formula have indicated that the (centered and rescaled)
 limiting distribution should be given by the one-point marginal distribution of the ${\rm Airy}_{2\to 1} $ process,
  which is expressed in terms of a Fredholm determinant.
 However, a Fredholm structure is not apparent before passing to the limit.
 In the flat case \cite{OQR17}, following \cite{CLeD11, CLeD12}, obtained a series formula for the same $q$-deformed
 Laplace transform of the height function as a limit of the half-flat case.
 The formula obtained for the flat case does not have an apparent Fredholm structure, either. However, for a different 
 $q$-deformation of the Laplace transform a Fredholm Pfaffian appears, but the pitfall of this new deformation is that it does not
 determine the distribution of the height function.
   
Our approach is orthogonal to the methods used in the above works. We do not rely on Bethe ansatz computations but we rather 
explore the underlying combinatorial structure of the log-gamma polymer.  
  Moreover, we do not derive the flat case as a limit of the half-flat but, instead, we work with the common underlying 
  structure, which allows for a more unified and systematic approach giving access to other geometries, as well.
  We do not pursue in this work an asymptotic analysis on the law of the partition functions as our primary focus has
  been the analysis of their combinatorial structure and the links to orthogonal Whittaker functions. 
  We hope, though, that the method developed here can provide a route to the asymptotic analysis of the log-gamma polymer in
  the flat and half-flat geometries and this is currently under investigation. This hope is also reinforced by the fact that
  in the zero temperature case (see the discussion that follows) the formulae that emerge from our approach provide
  alternative derivations of GOE and Airy$_{2\to 1}$ statistics \cite{BZ17}. In particular, they offer an alternative route to Sasamoto's Fredholm determinant formula for GOE \cite{Sa05}.
 
In the zero temperature setting, Baik-Rains \cite{BR01} (see also Ferrari \cite{Fe04} for a continuum, Hammersley last passage percolation model)
studied the point-to-line last passage percolation 
\begin{align*}
\tau^{\rm flat}_{n}:=\max_{\pi \in \fPi_{\,n}} \sum_{(i,j)\in \pi} W_{i,j}
\end{align*}
with geometrically distributed weights $W_{i,j}$
 via the use of the standard Robinson-Schensted-Knuth correspondence and the observation that
\begin{align}\label{symmetric}
\tau^{\rm flat}_{n}=\frac{1}{2}\max_{\pi \in \Pi_{n,n}} \sum_{(i,j)\in \pi} W^{\rm \,s}_{i,j} \, ,
\end{align}
where $\Pi_{n,n}$ is the set of directed paths from $(1,1)$ to $(n,n)$ and the matrix $(W^{\rm \, s}_{i,j}\colon 1\leq i,j\leq n)$
 is symmetric along the anti-diagonal, i.e.\ $W^{\rm \,s}_{n-j+1,n-i+1}=W^{\rm \,s}_{i,j}=W_{i,j}$, for all $(i,j)$ with $i+j\leq n+1$.
 However, in the polymer case (positive temperature) considering a point-to-point directed polymer on a symmetric along the anti-diagonal matrix does not give
 the point-to-line partition functions, as instead of \eqref{symmetric} one obtains that
  $Z^{\rm \,s}_{n,n}=\sum_{(i,j)\colon i+j=n+1} (Z_{i,j})^2$, where 
  $Z^{\rm \,s}_{n,n}$ denotes the point-to-point partition function
  on an antisymmetric matrix. Accordingly, our use of the geometric Robinson-Schensted-Knuth correspondence
   does not go through the route of applying it to antisymmetric
  matrices.
   
From our formula \eqref{p2line_gen1} we can pass to the zero temperature limit by scaling suitably the parameters of the 
inverse gamma variables and obtain the distribution function of $\tau_{2n}^{\rm flat}$ for exponentially distributed weights $W_{i,j}$
 in terms of (the continuous analogue of) symplectic $Sp_{2n}$--Schur functions $\sp_{\bm{\mu}}(\cdot)$. In the discrete setting, i.e.\ for \emph{geometric} weights with distribution
 \begin{align*}
 \P(W_{i,j}=k) \propto (y_i y_{2n+1-j})^k \, , \qquad k=0,1,2,\dots \, ,
 \end{align*}
  our formula would read as
\begin{equation}\label{zero_lim}
\P(\tau_{2n}^{\rm flat} \leq u)
\propto
\,\,\,
\sum_{\mathclap{\substack{\bm{\mu}\in\Z^n, \\ 0\leq \mu_n \leq \dots \leq \mu_1 \leq u}}}
\,
\sp_{\bm{\mu}}(y_1,\dots,y_n) \,\sp_{\bm{\mu}}(y_{2n},\dots,y_{n+1})  \, .
\end{equation} 
The analogue of formula \eqref{zero_lim} for exponential weights, established in \eqref{eq:flatLPP1},
 appears to be new and it is worth comparing it to the corresponding formula of Baik-Rains \cite{BR01} 
 for geometrically distributed weights, 
which is given in terms of a single Schur function as
\begin{equation*}
\P(\tau_{2n}^{\rm flat} \leq u)
\propto\,\,\,
\sum_{\mathclap{\substack{\bm{\mu}\in\Z^{2n}, \\ 0\leq \mu_{2n} \leq \dots \leq \mu_1 \leq u}}}
\,
\schur_{2\bm{\mu}}(y_1,\dots,y_{2n}) \, .
\end{equation*}
  
  Formula \eqref{eq:flatLPP1} comes from the fact that the zero temperature limit of $SO_{2n+1}$-Whittaker functions are 
  continuum versions of
$Sp_{2n}$-Schur functions. This might
seem a bit peculiar as one might expect to get the corresponding {\it orthogonal} Schur functions in the limit.
 The reason for this is that $SO_{2n+1}$-Whittaker functions have an integral representation over an analogue of the  
Gelfand-Tsetlin patterns that correspond to the symplectic group $Sp_{2n}$ (see Definition \ref{Givental_Whit} and Figure
\ref{sympl_GT}), which is dual to the orthogonal group $SO_{2n+1}$.
 This is in agreement with the Casselman-Shalika \cite{CS80}
 formula which describes the (unramified) Whittaker functions of a group $G$ as characters of a
 finite dimensional representation of the dual group of $G$ 
 (see also \cite{Ch16} for a probabilistic approach): the dual group of $SO_{2n+1}$ is $Sp_{2n}$, while the dual of $GL_{n}$ is itself,
 hence $GL_{n}$-Whittaker functions are the analogue of Schur functions, while $SO_{2n+1}$-Whittaker functions are
 the analogue of $Sp_{2n}$-Schur functions.
 
{\bf Organization of the article.}
 In Section~\ref{sec:Whitt} we introduce the Whittaker functions corresponding to $GL_n(\R)$ and $SO_{2n+1}(\R)$ and record the properties
 that will be useful for this work. In Section \ref{sec:P2LLaplaceTransform} we derive formulae \eqref{p2line_gen1}, \eqref{p2line_gen2}, \eqref{p2line_gen3} for the Laplace transforms in terms of Whittaker functions and then in terms of contour integrals 
 in the first two cases. 
  In section~\ref{sec:zero}, we pass to the zero temperature limit and obtain formulae
   for the law of the point-to-line last passage percolation with exponentially distributed waiting times in terms of 
   symplectic and classical Schur functions
   and also in terms of determinantal/Pfaffian formulae. Finally, in Appendix \ref{appendixA} we show the equivalence of the parametrization we 
   adapt for Whittaker functions and the parametrization used in number theory.
\label{sec:intro}

\section{Whittaker functions}\label{sec:Whitt}
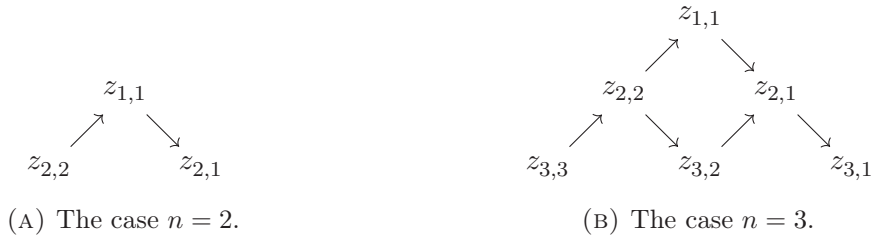
\begin{figure}
\centering

\begin{minipage}[b]{.5\linewidth}
\centering
\begin{tikzpicture}[scale=1]

\node (z11) at (1,-1) {$z_{1,1}$};
\node (z21) at (2,-2) {$z_{2,1}$};
\node (z22) at (0,-2) {$z_{2,2}$};

\draw[->] (z22) -- (z11);
\draw[->] (z11) -- (z21);

\end{tikzpicture}
\subcaption{The case $n=2$.}
\end{minipage}%
\begin{minipage}[b]{.5\linewidth}
\centering
\begin{tikzpicture}[scale=1]

\node (z11) at (1,-1) {$z_{1,1}$};
\node (z21) at (2,-2) {$z_{2,1}$};
\node (z22) at (0,-2) {$z_{2,2}$};
\node (z31) at (3,-3) {$z_{3,1}$};
\node (z32) at (1,-3) {$z_{3,2}$};
\node (z33) at (-1,-3) {$z_{3,3}$};

\draw[->] (z22) -- (z11);
\draw[->] (z11) -- (z21);
\draw[->] (z33) -- (z22);
\draw[->] (z22) -- (z32);
\draw[->] (z32) -- (z21);
\draw[->] (z21) -- (z31);

\end{tikzpicture}
\subcaption{The case $n=3$.}
\end{minipage}

\caption{Triangular arrays as in~\eqref{eq:triangle}. The arrows refer to formula~\eqref{eq:glEnergy}: $\mathcal{E}^{\triangle}(\bm{z})$ is the sum of all $a/b$ such that there is an arrow pointing from $a$ to $b$ in the diagram.}
\label{fig:glWhittakerFn}
\end{figure}

As we already mentioned, Whittaker functions appear in many different contexts and as a result they have various different ways of definition.
In number theory (theory of automorphic forms) they emerge as eigenfunctions of a commuting family of differential operators - the center
of the universal enveloping algebra of the associated group - that have certain invariance properties \cite{Gold06}. In that
context, Jacquet \cite{J67} introduced what is known as {\it Jacquet Whittaker function} via a certain integral representation. 
In a representation theoretic setting, Konstant \cite{Ko78} constructed Whittaker functions as a solution to the quantum Toda lattice.
In quantum cohomology and mirror symmetry, Givental \cite{Giv97} viewed Whittaker functions as solutions to a certain integrable system,
which turned out to be the quantum Toda lattice to which he constructed integral solutions via mirror symmetry. This approach
was extended further for general classical groups by Gerasimov-Lebedev-Oblezin \cite{GLO07, GLO08}. 
From that quantum cohomology setting further integral representations over geometric crystals have also emerged \cite{Lam13}, \cite{Rie12}.
A nice algebraic summary of various incarnations of Whittaker functions appears in \cite{Lam13} and an account touching upon
both algebraic and probabilistic aspects in \cite{Ch13}.
 
 Here, we will deal with $GL_n$ and $SO_{2n+1}$-Whittaker functions and the most relevant representation in our context is that
 of Givental and Gerasimov-Lebedev-Oblezin. In the following subsections we will recall these integral representations as well
 as other aspects of Whittaker functions which are important for our purposes.
\subsection{$\mathfrak{gl}_n$-Whittaker functions}
\label{sec:glWhittakerFns}
Following Givental \cite{Giv97}, see also \cite{GLO07, GLO08},
 we introduce $\mathfrak{gl}_n$-Whittaker functions, as integrals on triangular patterns.
Let $n\geq 1$, and consider a triangular array of depth $n$
\begin{equation}
\label{eq:triangle}
\bm{z} = (z_{i,j}\colon ~~ 1\leq i\leq n, ~~ 1\leq j\leq i)
\end{equation}
with positive entries; examples are given in Figure~\ref{fig:glWhittakerFn}.
Whenever the entries are interlaced, that is
\begin{align}\label{interlace}
z_{i+1,j+1} \leq z_{i,j} \leq z_{i+1,j}  \qquad \text{for $1\leq j\leq i\leq n-1$} \, ,
\end{align}
such arrays are known as \emph{Gelfand-Tsetlin} patterns. Here, we will be working with triangular arrays that do {\it not} satisfy the interlacement condition \eqref{interlace}. We will call such triangular arrays {\it geometric} Gelfand-Tsetlin patterns. Even though geometric Gelfand-Tsetlin patterns do not satisfy \eqref{interlace}, we will impose a potential $\mathcal{E}^{\triangle}(\bm{z})$
on them, which encourages interlacement. This potential is (see Figure \ref{fig:glWhittakerFn} for a graphical representation)
\begin{equation}
\label{eq:glEnergy}
\mathcal{E}^{\triangle}(\bm{z})
:= \sum_{i=1}^{n-1} \sum_{j=1}^i
\Big( \frac{z_{i+1,j+1}}{z_{i,j}} +\frac{z_{i,j}}{z_{i+1,j}} \Big)
\, .
\end{equation}
We will call {\it $i$-th row} of $\bm{z}$ the vector $(z_{i,1},\dots,z_{i,i})$ of all entries with first index equal to $i$.
The set of geometric Gelfand-Tsetlin patterns of depth $n$ and with bottom row equal to a vector $\bx\in \bbR^n$ will be denoted by 
$\mathcal{T}_{n}^{\triangle}(\bm{x})$. 
We also define the {\it type}, $\type(\bm{z}) \in \R_{+}^{n}$ as the vector whose $i$-th component is the ratio between the product of the $i$-th row elements of $\bm{z}$ and the product of its $(i-1)$-th row elements; in other words,
\begin{equation*}
\label{eq:glType}
\type(\bm{z})_i
:= \frac{\prod_{j = 1}^{i} z_{i,j}}{\prod_{j = 1}^{i-1} z_{i-1,j}} \qquad \text{for } i=1, \dots, n
\, ,
\end{equation*}
where empty products are excluded from the expressions (as we will always suppose from now on).
We now define the $\mathfrak{gl}_n$-Whittaker functions via an integral representation:
\begin{definition}
\label{def:glWhittakerFn} 
The \emph{$\mathfrak{gl}_{n}$-Whittaker function} with parameter $\bm{\alpha}\in\C^n$ is given by
\begin{equation}
\label{eq:glWhittakerFn}
\Psi^{\mathfrak{gl}_{n}}_{\bm{\alpha}}(\bm{x})
:= \int_{\mathcal{T}_{n}^{\triangle}(\bm{x})}
\type(\bm{z})^{\bm{\alpha}}
\exp\big(-\mathcal{E}^{\triangle}(\bm{z})\big)
~~\prod_{\mathclap{\substack{1\leq i<n \\ 1\leq j\leq i}}} ~\frac{\diff z_{i,j}}{z_{i,j}} \, ,
\end{equation}
for all $\bm{x}\in\R_{+}^n$, where $\mathcal{T}_{n}^{\triangle}(\bm{x})$ denotes the set of all triangular arrays $\bm{z}$ of depth $n$ with positive entries and $n$-th row equal to $\bm{x}$, and
\[
\type(\bm{z})^{\bm{\alpha}}
:= \prod_{i=1}^{n} \type(\bm{z})_i^{\alpha_i}
\, .
\]
\end{definition}
For example, $\Psi^{\mathfrak{gl}_1}_{\alpha}(x) = x^{\alpha}$, and
\begin{equation}
\label{eq:gl_2WhittakerFn}
\Psi^{\mathfrak{gl}_2}_{(\alpha_1,\alpha_2)}(x_1,x_2)
= \int_{\R_{+}} z^{\alpha_1}\Big(\frac{x_1 x_2}{z}\Big)^{\alpha_2} \exp\Big(-\frac{x_2}{z}-\frac{z}{x_1}\Big) \frac{\diff z}{z} \, .
\end{equation}

The representation of $\mathfrak{gl}_n$-Whittaker functions in Definition \ref{def:glWhittakerFn} 
has a recursive structure: setting $\Psi_{\emptyset}^{\mathfrak{gl}_{0}}(\emptyset) :=1$, it turns out that for all $n\geq 1$,
 $\bm{\alpha}=(\alpha_1,\dots,\alpha_n)\in\C^n$ and $\bm{x}=(x_1,\dots,x_n)\in\R_{+}^n$
\begin{equation}
\label{eq:glWhittakerRecurs}
\Psi_{\bm{\alpha}}^{\mathfrak{gl}_{n}}(\bm{x})
= \int_{\R_{+}^{n-1}}
Q^{\mathfrak{gl}_{n}}_{\alpha_n}(\bm{x},\bm{u})
\Psi_{\widetilde{\bm{\alpha}}}^{\mathfrak{gl}_{n-1}}(\bm{u})
\prod_{j=1}^{n-1} \frac{\diff u_j}{u_j} \, ,
\end{equation}
where $\widetilde{\bm{\alpha}}:=(\alpha_1,\dots,\alpha_{n-1})$ and the kernel is defined by
\[
Q^{\mathfrak{gl}_{n}}_{\alpha_n}(\bm{x},\bm{u})
= \bigg( \frac{\prod_{j=1}^n x_j}{\prod_{j=1}^{n-1} u_j} \bigg)^{\alpha_n}
\prod_{j=1}^{n-1} \exp\Big(-\frac{x_{j+1}}{u_j} - \frac{u_j}{x_j}\Big) \, .
\]

An easy-to-deduce property of $\mathfrak{gl}_n$-Whittaker functions, which will be useful for us, is that for $c\in\C$
\begin{equation}
\label{eq:glWhittakerFnTranslation}
\Psi^{\mathfrak{gl}_n}_{\bm{\alpha}+c}(\bm{x})
= \bigg(\prod_{i=1}^n x_i\bigg)^c
\Psi^{\mathfrak{gl}_n}_{\bm{\alpha}}(\bm{x}) \, ,
\end{equation}
where $\bm{\alpha}+c$ stands for $(\alpha_1+c,\dots,\alpha_n+c)$. 
Another property of Whittaker functions, which is not obvious from the
integral formula, but comes from their construction, is that they are invariant under the action of the corresponding Weyl group on the
(spectral) parameters $\bm{\ga}$. For the case of $\mathfrak{gl}_n$-Whittaker this means that $\Psi^{\mathfrak{gl}_n}_{\bm{\ga}}(\cdot)$
is invariant under permutation of the entries of $\bm{\ga}=(\ga_1,\dots,\ga_n)$.
 
There is a distinguished differential operator diagonalized by $\mathfrak{gl}_n$-Whittaker function, 
which is the quantum Toda hamiltonian. 
In particular, if we set $\psi_{\bm{\gl}}^{\mathfrak{gl}_n}(x_1,\dots,x_n) := \Psi^{\gln_n}_{\bm{\lambda}}(e^{x_1},\dots,e^{x_n})$, then
$ \psi^{\gln_n}_{\bm{\lambda}}$ is the unique eigenfunction with moderate growth of the operator
\begin{align*}
-\Delta+2\sum_{i=1}^{n-1} e^{-\langle\mathfrak{a}_i,\bm{x}\rangle} \, ,
\end{align*}
with eigenvalue $-|\bm{\lambda}|^2:=-\sum_{i=1}^n\lambda_i^2$. 
Here $\mathfrak{a}_i \in \mathbb{R}^n$ ($i=1,\dots,n-1$) are the positive roots of the Lie group $GL_n(\mathbb{R})$, i.e.\ $\langle\mathfrak{a}_i,\bm{x}\rangle=x_{i+1}-x_i$.
As eigenfunctions of a self-adjoint operator, Whittaker functions come with a harmonic analysis, which is very useful for our purposes and is summarized in

\begin{theorem}[\cite{STS94, KL01}]
\label{thm:Plancherel}
The integral transform
\[
\hat f(\bm{\lambda}) := \int_{\R_{+}^n} f(\bm{x}) \Psi^{\gln_n}_{\bm{\lambda}}(\bm{x}) \prod_{i=1}^n \frac{\dd x_i}{x_i}
\]
defines an isometry from $L^2(\R_{+}^n, \prod_{i=1}^n \dd x_i/x_i)$ to 
$L^2_{\sym}(\iota\R^n,s_n(\bm{\lambda}) \dd \bm{\lambda})$, where $\iota=\sqrt{-1}$, $L^2_{\sym}$ denotes the space of square integrable functions that are symmetric in their variables, and 
\begin{equation}
\label{eq:sklyaninMeasure}
s_n(\bm{\lambda})
:= \frac{1}{(2\pi\iota)^n n!} \prod_{i\neq j} \Gamma(\lambda_i-\lambda_j)^{-1}
\end{equation}
is the density of the \emph{Sklyanin measure}.
Namely, for all $f,g\in L^2(\R_{+}^n, \prod_{i=1}^n \dd x_i/x_i)$ it holds that
\[
\int_{\R_{+}^n} f(\bm{x}) \overline{g(\bm{x})} \prod_{i=1}^n\frac{\dd x_i}{x_i} 
= \int_{\iota\R^n} \hat{f}(\bm{\lambda}) \overline{\hat{g}(\bm{\lambda})} s_n(\bm{\lambda}) \dd \bm{\lambda} \, .
\]
\end{theorem}

The computation of certain integrals of Whittaker functions plays an important role in the theory of $L$-functions, as it is related
to certain functional equations \cite{Bum84}. One such integral formula was conjectured by Bump \cite{Bum84} and proved for $\gln_3$
and in the general case by Stade \cite{Sta02}. This is the following

\begin{theorem}
\label{thm:Stade}
 Suppose $r>0$ and $\bm{\alpha},\bm{\beta}\in\C^n$ such that $\Re (\alpha_i+\beta_j)>0$ for all $i,j$. Then
\begin{align}
\label{eq:Stade}
\int_{\R_{+}^n } e^{-r x_1} \Psi^{\gln_n}_{\bm{\alpha}}(\bm{x}) \Psi^{\gln_n}_{\bm{\beta}}(\bm{x})
\prod_{i=1}^n \frac{dx_i}{x_i}
= r^{-\sum_{k=1}^n (\alpha_k+\beta_k)} \prod_{i,j=1}^n \Gamma(\alpha_i+\beta_j) \, .
\end{align}
\end{theorem}

A  bijective proof of identity \eqref{eq:Stade} was given in \cite{OSZ14} via the use of the geometric $\rsk$
 correspondence. Subsequently, together with 
 Theorem \ref{thm:Plancherel}, this identity played an important role in the computation of the Laplace transform of the point-to-point partition function of the log-gamma polymer.

\subsection{$\mathfrak{so}_{2n+1}$-Whittaker functions}
\label{sec:orthogonalWhittakerFns}

\begin{figure}
\centering

\begin{minipage}[b]{.5\linewidth}
\centering
\begin{tikzpicture}[scale=1]

\node (z11) at (1,-1) {$z_{1,1}$};
\node (z21) at (2,-2) {$z_{2,1}$};
\node (z22) at (0,-2) {};

\draw[->] (z22) -- (z11);
\draw[->] (z11) -- (z21);

\draw (0,-0.6) -- (0,-2.3);

\end{tikzpicture}
\subcaption{The case $n=1$.}
\end{minipage}%
\begin{minipage}[b]{.5\linewidth}
\centering
\begin{tikzpicture}[scale=1]

\node (z11) at (1,-1) {$z_{1,1}$};
\node (z21) at (2,-2) {$z_{2,1}$};
\node (z22) at (0,-2) {};
\node (z31) at (3,-3) {$z_{3,1}$};
\node (z32) at (1,-3) {$z_{3,2}$};
\node (z41) at (4,-4) {$z_{4,1}$};
\node (z42) at (2,-4) {$z_{4,2}$};
\node (z43) at (0,-4) {};

\draw[->] (z22) -- (z11);
\draw[->] (z11) -- (z21);
\draw[->] (z32) -- (z21);
\draw[->] (z21) -- (z31);
\draw[->] (z43) -- (z32);
\draw[->] (z32) -- (z42);
\draw[->] (z42) -- (z31);
\draw[->] (z31) -- (z41);

\draw (0,-0.6) -- (0,-4.3);

\end{tikzpicture}
\subcaption{The case $=2$.}\label{sympl_GT}
\end{minipage}

\caption{Half-triangular arrays as in~\eqref{eq:halfTriangle}. The arrows refer to formula~\eqref{eq:SOenergy}: $\mathcal{E}^{\llrighttriangle[0.15]}(\bm{z})$ is the sum of all $a/b$ such that there is an arrow pointing from $a$ to $b$ in the diagram. The convention is that all the numbers on the vertical wall are $1$, so that the only inhomogeneous addends in $\mathcal{E}^{\llrighttriangle[0.15]}(\bm{z})$ are $1/z_{2k-1,k}$ for $1\leq k\leq n$.}
\label{fig:soWhittakerFn}
\end{figure}
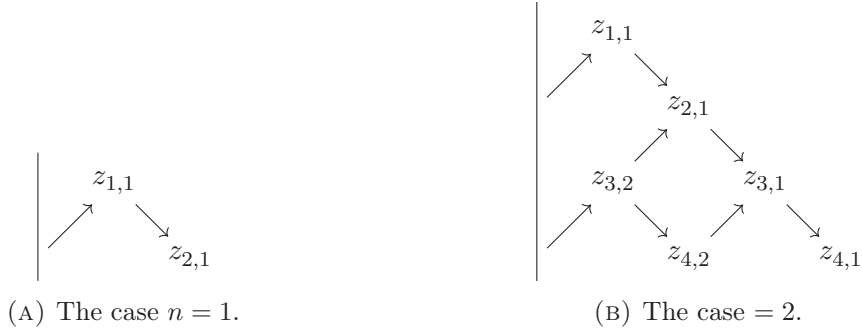

Similarly to the $\gln_n$ case, we will define $\son_{2n+1}$-Whittaker functions as integrals of half-triangular arrays.
Such definition was first given in \cite{GLO07, GLO08} but also emerged naturally in \cite{Nte17} in the study of a system of interacting
particles via {\it intertwining} and Markovian dynamics.
Let $n\geq 1$, and let us consider a half-triangular array of depth $2n$
\begin{equation}
\label{eq:halfTriangle}
\bm{z} = (z_{i,j}\colon ~~ 1\leq i\leq 2n, ~~ 1\leq j\leq \ceil{i/2})
\end{equation}
with positive entries; examples are given in Figure~\ref{fig:soWhittakerFn}.
These arrays correspond to {\it symplectic Gelfand-Tsetlin patterns} \cite{Sun90},
 when the entries are interlaced, that is
\begin{align*}
z_{i+1,j+1} \leq z_{i,j} \leq z_{i+1,j} \qquad \text{for } 1 \leq i\leq 2n-1\, , \,\,1\leq j\leq \ceil{i/2} \, ,
\end{align*}
with the understanding that $z_{i,j}$ is set to be zero when $(i,j)\notin \{1\leq i\leq 2n, ~~ 1\leq j\leq \ceil{i/2}\}$. As in the $\gln_{n}$
 case, we will
be working with half-triangular arrays that do not satisfy the interlacing condition \eqref{eq:halfTriangle} 
but are, nevertheless, encouraged 
to do so through the potential (see Figure \ref{fig:soWhittakerFn} for a graphical representation)
\begin{equation}
\label{eq:SOenergy}
\mathcal{E}^{\llrighttriangle[0.15]}(\bm{z})
:= \sum_{i=1}^{2n-1} \sum_{j=1}^{\ceil{i/2}}
\Big( \frac{z_{i+1,j+1}}{z_{i,j}} +\frac{z_{i,j}}{z_{i+1,j}} \Big)
\, ,
\end{equation}
with the convention that $z_{i,j}:=1$ if $j>\ceil{i/2}$.
 
We call {\it $i$-th row} of $\bm{z}$ the vector $(z_{i,1},\dots,z_{i,\ceil{i/2}})$ of all entries with first index equal to $i$
and we denote by $\mathcal{T}_{2n}^{\llrighttriangle[0.15]}(\bm{x})$ the set of all half-triangular arrays $\bz$ of depth $2n$ with positive entries and bottom
row equal to $\bx$. 
We also define the {\it type}, $\type(\bm{z}) \in \R_{+}^{2n}$, as the vector whose $i$-th component is the ratio between the product of the $i$-th row elements of $\bm{z}$ and the product of its $(i-1)$-th row elements; in other words,
\begin{equation*}
\label{eq:type}
\type(\bm{z})_i
:= \frac{\prod_{j = 1}^{\ceil{i/2}} z_{i,j}}{\prod_{j = 1}^{\ceil{(i-1)/2}} z_{i-1,j}} \qquad \text{for } i=1, \dots, 2n
\, .
\end{equation*}
Finally, for $\bm{\beta}=(\beta_1,\dots,\beta_n)\in\C^n$, we define $\bm{\beta}^{\pm}\in\C^{2n}$ by
\[
\bm{\beta}^{\pm} :=
(\beta_1,-\beta_1,\beta_2,-\beta_2,\dots,\beta_n,-\beta_n) \, .
\]
We are now able to define the orthogonal Whittaker functions via an integral representation~\cite{GLO07, GLO08}.
 \begin{definition}\label{Givental_Whit}
\label{def:soWhittakerFn}
The \emph{$\mathfrak{so}_{2n+1}$-Whittaker function} with parameter $\bm{\beta} \in \C^n$ is given by
\begin{equation}
\label{eq:soWhittakerFn}
\Psi^{\mathfrak{so}_{2n+1}}_{\bm{\beta}}(\bm{x})
:= \int_{\mathcal{T}_{2n}^{\llrighttriangle[0.15]}(\bm{x})}
\type(\bm{z})^{\bm{\beta}^{\pm}}
\exp\Big(-\mathcal{E}^{\llrighttriangle[0.15]}(\bm{z})\Big)
\,\,\,\,\prod_{\mathclap{\substack{1\leq i<2n, \\ 1\leq j\leq \ceil{i/2}}}} \,\,\frac{\diff z_{i,j}}{z_{i,j}} \, ,
\end{equation}
for all $\bm{x}\in\R_{+}^n$, where $\mathcal{T}_{2n}^{\llrighttriangle[0.15]}(\bm{x})$ denotes the set of all half-triangular arrays $\bm{z}$ of depth $2n$ with positive entries and $(2n)$-th row equal to $\bm{x}$, and
\[
\type(\bm{z})^{\bm{\beta}^{\pm}}
:= \prod_{i=1}^{2n} \type(\bm{z})_i^{\beta^{\pm}_i}
= \prod_{k=1}^n \type(\bm{z})_{2k-1}^{\beta_k} \type(\bm{z})_{2k}^{-\beta_k} 
\, .
\]
\end{definition}
Again, even though not obvious from this definition, $\Psi^{\mathfrak{so}_{2n+1}}_{\bm{\gb}}(\cdot)$
 is invariant under permutations and reflections
of the (spectral) parameters $\bm{\gb}=(\gb_1,\dots,\gb_n)$, where by reflection we mean multiplication of the entries of $\bm{\gb}$ by $\pm1$. 

As an example, the $\mathfrak{so}_3$-Whittaker function is given by
\begin{equation}
\label{eq:so_3WhittakerFn}
\Psi^{\mathfrak{so}_3}_{\beta}(x)
= \int_{\R_{+}} \bigg(\frac{z^2}{x}\bigg)^{\beta} \exp\Big(-\frac{1}{z}-\frac{z}{x}\Big) \frac{\diff z}{z} \, .
\end{equation}

The recursive structure of $\mathfrak{so}_{2n+1}$-Whittaker functions is
\begin{equation}
\label{eq:soWhittakerRecurs}
\Psi_{\bm{\beta}}^{\mathfrak{so}_{2n+1}}(\bm{x})
= \int_{\R_{+}^{n-1}}
Q^{\mathfrak{so}_{2n+1}}_{\beta_n}(\bm{x},\bm{u})
\Psi_{\widetilde{\bm{\beta}}}^{\mathfrak{so}_{2n-1}}(\bm{u})
\prod_{j=1}^{n-1} \frac{\diff u_j}{u_j} \, ,
\end{equation}
for all $n\geq 1$, $\bm{\beta}=(\beta_1,\dots,\beta_n)\in\C^n$,  
$\widetilde{\bm{\beta}}:=(\beta_1,\dots,\beta_{n-1})$ and $\bm{x}=(x_1,\dots,x_n)\in\R_{+}^n$,
where we set $\Psi_{\emptyset}^{\mathfrak{so}_{1}}(\emptyset) :=1$ and
 the kernel $Q^{\mathfrak{so}_{2n+1}}_{\beta_n}$ is defined by
\[
\begin{split}
Q^{\mathfrak{so}_{2n+1}}_{\beta_n}(\bm{x},\bm{u})
:= \int_{\R_{+}^n} \bigg( \frac{\prod_{j=1}^n v_j^2}{\prod_{j=1}^n x_j \prod_{j=1}^{n-1} u_j} \bigg)^{\beta_n}
&\prod_{j=1}^{n-1} \exp\Big(-\frac{v_{j+1}}{u_j} - \frac{u_j}{v_j}\Big) \\
&\times \prod_{j=1}^n \exp\Big(-\frac{x_{j+1}}{v_j} - \frac{v_j}{x_j}\Big)
\prod_{j=1}^{n} \frac{\diff v_j}{v_j} \, ,
\end{split}
\]
with the convention that $x_{n+1}:=1$.
Similarly to the $\gln_n$ case, $\son_{2n+1}$-Whittaker functions are  eigenfunctions of the quantum Toda hamiltonian of type $B$
\begin{align*}
-\Delta+2\sum_{i=1}^{n-1} e^{\langle \mathfrak{b}_i,\bm{x}\rangle} +e^{-\langle \mathfrak{b}_n,\bm{x} \rangle}.
\end{align*}
Here $\mathfrak{b}_i\in \mathbb{R}^n$ ($i=1,\dots,n$) are the positive roots of the Lie group $SO_{2n+1}(\mathbb{R})$, i.e.\ $\langle\mathfrak{b}_i,\bm{x}\rangle=x_{i+1}-x_i$ for $i=1,\dots,n-1$ and $\langle\mathfrak{b}_n,\bm{x}\rangle=x_n$.

The following integral identity, which will play an important role in our polymer analysis and is of similar nature as the Bump-Stade identity~\eqref{eq:Stade}, has been proven by Ishii and Stade~\cite{IS13}:
\begin{theorem}
\label{thm:IshiiStade}
Let $\bm{\alpha},\bm{\beta}, \in\C^n$, where $\Re(\alpha_i) > \abs{\Re(\beta_j)}$ for all $i,j$. Then
\begin{equation}
\label{eq:IshiiStade}
\int_{\R_{+}^n}
\Psi_{-\bm{\alpha}}^{\mathfrak{gl}_n}(\bm{x})
\Psi_{\bm{\beta}}^{\mathfrak{so}_{2n+1}}(\bm{x})
\prod_{i=1}^n \frac{\diff x_i}{x_i}
= \frac{\prod_{1\leq i,j\leq n}
\Gamma(\alpha_i + \beta_j)
\Gamma(\alpha_i - \beta_j)}
{\prod_{1\leq i<j\leq n} \Gamma(\alpha_i+\alpha_j)} \, .
\end{equation}
\end{theorem}

Note that, thanks to the restriction on the parameters in the above formula, the arguments of all Gamma functions on the right-hand side have positive real parts.
Let us point out that the parametrization used for Whittaker functions in number theory (in particular in \cite{IS13})
is different from ours. In Appendix \ref{appendixA} we will show the correspondence between the different parametrizations and 
the equivalence between \eqref{eq:IshiiStade} and the corresponding integral formulae in  \cite{IS13}.

\section{Laplace transforms of point-to-line partition functions}
\label{sec:P2LLaplaceTransform}

In this section, we compute the Laplace transform of the point-to-line, the point-to-half-line and the restricted point-to-half-line log-gamma polymer partition functions at any even time $2n$. We first express these as integrals involving Whittaker functions and next as contour integrals of Gamma functions.
 
In order to study the distribution of point-to-line partition functions, we first need the joint law of the point-to-point partition functions 
along a ``fixed time'' line. This can be done by using the geometric 
Robinson-Schensted-Knuth correspondence ($\grsk$) for polygonal arrays and its properties, as these were established in~\cite{NZ15}. 
We start by recalling the $\grsk$ construction via {\it local moves} and its relevant properties. We should note that 
the inductive procedure for $\grsk$ 
that we present here is a little different from the inductions presented in \cite{NZ15, OSZ14}.

\subsection{Geometric $\rsk$}
The $\grsk$ is a bijective map between polygonal arrays with positive entries, which can be defined as a sequence of \emph{local moves}.
A \emph{polygonal array} is defined to be an array $\bm{t} = \{t_{i,j} \colon (i,j)\in \mathcal{I}\}$ with positive entries and indexed by a finite set $\mathcal{I}\subseteq\N\times\N$ satisfying: if $(i,j)\in\mathcal{I}$, then $(i-1,j)\in\mathcal{I}$ if $i>1$, and $(i,j-1)\in\mathcal{I}$ if $j>1$. We denote by $\arrays$ the set of all polygonal arrays indexed by $\mathcal{I}$. We say that $(i,j)\in\mathcal{I}$ is a \emph{border index} for $\mathcal{I}$, or for $\bm{t}\in\arrays$, if not all three sites $(i,j+1),(i+1,j),(i+1,j+1)$ belong to $\mathcal{I}$; we call it \emph{outer index} if, furthermore, $\{(i,j+1),(i+1,j),(i+1,j+1)\} \cap \mathcal{I} = \emptyset$. 
We denote the set of outer indices of $\mathcal{I}$ by $\mathcal{I}_{\rm out}$. 
See Figure~\ref{fig:polygonalArray} for a graphical interpretation of $\bm{t}$ and its (outer) border indices.

\begin{figure}
\centering
\begin{minipage}[b]{.5\linewidth}
\centering
\begin{tikzpicture}[scale=0.8, every node/.style={transform shape}]
\node (t11) at (1,-1) {$t_{1,1}$};
\node (t12) at (2,-1) {$t_{1,2}$};
\node (t13) at (3,-1) {$t_{1,3}$};
\node (t14) at (4,-1) {$t_{1,4}$};
\node[draw,circle,inner sep=1pt] (t15) at (5,-1) {$t_{1,5}$};
\node[draw,circle,inner sep=1pt] (t16) at (6,-1) {$\red{t_{1,6}}$};

\node (t21) at (1,-2) {$t_{2,1}$};
\node (t22) at (2,-2) {$t_{2,2}$};
\node (t23) at (3,-2) {$t_{2,3}$};
\node[draw,circle,inner sep=1pt] (t24) at (4,-2) {$t_{2,4}$};
\node[draw,circle,inner sep=1pt] (t25) at (5,-2) {$\red{t_{2,5}}$};

\node (t31) at (1,-3) {$t_{3,1}$};
\node (t32) at (2,-3) {$t_{3,2}$};
\node[draw,circle,inner sep=1pt] (t33) at (3,-3) {$t_{3,3}$};
\node[draw,circle,inner sep=1pt] (t34) at (4,-3) {$\red{t_{3,4}}$};

\node (t41) at (1,-4) {$t_{4,1}$};
\node[draw,circle,inner sep=1pt] (t42) at (2,-4) {$t_{4,2}$};
\node[draw,circle,inner sep=1pt] (t43) at (3,-4) {$\red{t_{4,3}}$};

\node[draw,circle,inner sep=1pt] (t51) at (1,-5) {$t_{5,1}$};
\node[draw,circle,inner sep=1pt] (t52) at (2,-5) {$\red{t_{5,2}}$};

\node[draw,circle,inner sep=1pt] (t61) at (1,-6) {$\red{t_{6,1}}$};

\draw[->] (0.4,-0.4) -- (t11);

\draw[->] (t11) -- (t12);
\draw[->] (t12) -- (t13);
\draw[->] (t13) -- (t14);
\draw[->] (t14) -- (t15);
\draw[->] (t15) -- (t16);

\draw[->] (t21) -- (t22);
\draw[->] (t22) -- (t23);
\draw[->] (t23) -- (t24);
\draw[->] (t24) -- (t25);

\draw[->] (t31) -- (t32);
\draw[->] (t32) -- (t33);
\draw[->] (t33) -- (t34);

\draw[->] (t41) -- (t42);
\draw[->] (t42) -- (t43);

\draw[->] (t51) -- (t52);

\draw[->] (t11) -- (t21);
\draw[->] (t21) -- (t31);
\draw[->] (t31) -- (t41);
\draw[->] (t41) -- (t51);
\draw[->] (t51) -- (t61);

\draw[->] (t12) -- (t22);
\draw[->] (t22) -- (t32);
\draw[->] (t32) -- (t42);
\draw[->] (t42) -- (t52);

\draw[->] (t13) -- (t23);
\draw[->] (t23) -- (t33);
\draw[->] (t33) -- (t43);

\draw[->] (t14) -- (t24);
\draw[->] (t24) -- (t34);

\draw[->] (t15) -- (t25);
\end{tikzpicture}
\subcaption{Triangular array}
\label{subfig:triangularArray}
\end{minipage}%
\begin{minipage}[b]{.5\linewidth}
\centering
\begin{tikzpicture}[scale=0.8, every node/.style={transform shape}]
\node (t11) at (1,-1) {$t_{1,1}$};
\node (t12) at (2,-1) {$t_{1,2}$};
\node (t13) at (3,-1) {$t_{1,3}$};
\node (t14) at (4,-1) {$t_{1,4}$};
\node (t15) at (5,-1) {$t_{1,5}$};
\node[draw,circle,inner sep=1pt] (t16) at (6,-1) {$t_{1,6}$};

\node (t21) at (1,-2) {$t_{2,1}$};
\node (t22) at (2,-2) {$t_{2,2}$};
\node[draw,circle,inner sep=1pt] (t23) at (3,-2) {$t_{2,3}$};
\node[draw,circle,inner sep=1pt] (t24) at (4,-2) {$t_{2,4}$};
\node[draw,circle,inner sep=1pt] (t25) at (5,-2) {$t_{2,5}$};
\node[draw,circle,inner sep=1pt] (t26) at (6,-2) {$\red{t_{2,6}}$};

\node (t31) at (1,-3) {$t_{3,1}$};
\node (t32) at (2,-3) {$t_{3,2}$};
\node[draw,circle,inner sep=1pt] (t33) at (3,-3) {$t_{3,3}$};

\node (t41) at (1,-4) {$t_{4,1}$};
\node (t42) at (2,-4) {$t_{4,2}$};
\node[draw,circle,inner sep=1pt] (t43) at (3,-4) {$t_{4,3}$};

\node (t51) at (1,-5) {$t_{5,1}$};
\node[draw,circle,inner sep=1pt] (t52) at (2,-5) {$t_{5,2}$};
\node[draw,circle,inner sep=1pt] (t53) at (3,-5) {$\red{t_{5,3}}$};

\node[draw,circle,inner sep=1pt] (t61) at (1,-6) {$t_{6,1}$};
\node[draw,circle,inner sep=1pt] (t62) at (2,-6) {$\red{t_{6,2}}$};

\draw[->] (0.4,-0.4) -- (t11);

\draw[->] (t11) -- (t12);
\draw[->] (t12) -- (t13);
\draw[->] (t13) -- (t14);
\draw[->] (t14) -- (t15);
\draw[->] (t15) -- (t16);

\draw[->] (t21) -- (t22);
\draw[->] (t22) -- (t23);
\draw[->] (t23) -- (t24);
\draw[->] (t24) -- (t25);
\draw[->] (t25) -- (t26);

\draw[->] (t31) -- (t32);
\draw[->] (t32) -- (t33);

\draw[->] (t41) -- (t42);
\draw[->] (t42) -- (t43);

\draw[->] (t51) -- (t52);
\draw[->] (t52) -- (t53);

\draw[->] (t61) -- (t62);

\draw[->] (t11) -- (t21);
\draw[->] (t21) -- (t31);
\draw[->] (t31) -- (t41);
\draw[->] (t41) -- (t51);
\draw[->] (t51) -- (t61);

\draw[->] (t12) -- (t22);
\draw[->] (t22) -- (t32);
\draw[->] (t32) -- (t42);
\draw[->] (t42) -- (t52);
\draw[->] (t52) -- (t62);

\draw[->] (t13) -- (t23);
\draw[->] (t23) -- (t33);
\draw[->] (t33) -- (t43);
\draw[->] (t43) -- (t53);

\draw[->] (t14) -- (t24);

\draw[->] (t15) -- (t25);

\draw[->] (t16) -- (t26);
\end{tikzpicture}
\subcaption{Generic polygonal array}
\label{subfig:polygonalArray}
\end{minipage}
\caption{Examples of polygonal array, where all the border indices are circled and the outer border indices are also highlighted in red. The arrows refer to formula~\eqref{eq:energy}: $\mathcal{E}(\bm{t})$ is the sum of all $a/b$ such that there is an arrow pointing from $a$ to $b$ in the diagram; the arrow pointing to $t_{1,1}$ corresponds to the term $1/t_{1,1}$.}
\label{fig:polygonalArray}
\end{figure}
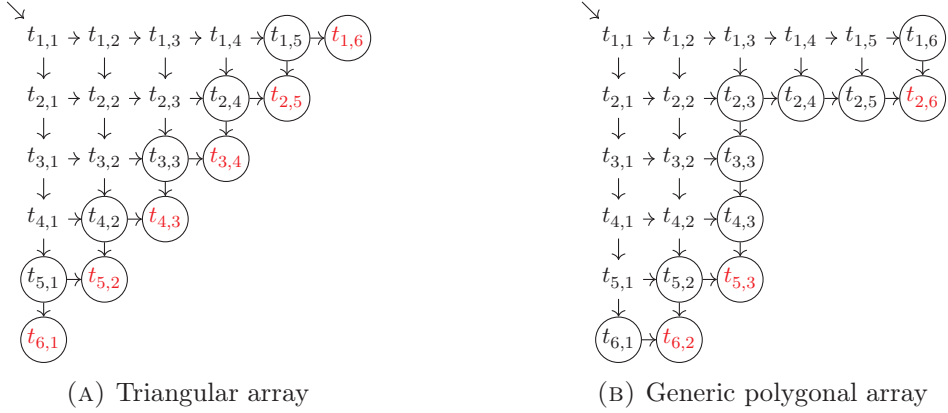

We now construct the $\grsk$ map explicitly. Define first the following maps acting on $\bm{w}\in\arrays$, with the convention that $w_{0,j} = w_{i,0} =0$ but $w_{1,0} + w_{0,1} = 1$:
\begin{itemize}
\item
for all $(i,j)\in\mathcal{I}$, $a_{i,j}$ replaces $w_{i,j}$ with
\[
w_{i,j}(w_{i-1,j}+w_{i,j-1})
\]
and leaves all other entries of $\bm{w}$ unchanged;
\item
for all \emph{non}-border indices $(i,j)\in\mathcal{I}$, $b_{i,j}$ replaces $w_{i,j}$ with
\[
\frac{1}{w_{i,j}} (w_{i-1,j} + w_{i,j-1}) \bigg(\frac{1}{w_{i+1,j}} + \frac{1}{w_{i,j+1}} \bigg)^{-1}
\]
and leaves all other entries of $\bm{w}$ unchanged.
\end{itemize}
Operations $b_{ij}$ and $a_{ij}$ are related to Bender-Knuth transformations and in the setting of geometric $\rsk$ correspondence 
were first introduced by Kirillov \cite{K01}.  
We now define the operation
\[
\rho_{i,j} :=
\bigcirc_{k\geq 1} b_{i-k,j-k} \circ a_{i,j} \, ,
\]
where $\bigcirc_{k\geq 1}$ indicates a sequence of compositions in which
 $b_{i-k,j-k}$ appears in the composition only if $(i-k,j-k)\in\mathcal{I}$.
The maps $a_{i,j}$'s and $b_{i,j}$'s are called \emph{local moves} because they act on arrays only locally (see Figure \ref{subfig:localMoves}). It is also clear from this figure that two local moves indexed by $(i,j)$ and $(i',j')$ commute if $\abs{i-i'} + \abs{j-j'} > 1$. Consequently, the order of the sequence of local moves making up a single $\rho_{i,j}$ does not matter. Moreover, $\rho_{i,j}$ and $\rho_{i',j'}$ commute whenever the diagonals that $(i,j)$ and $(i',j')$ belong to are neither the same nor consecutive, i.e.\ $\abs{(j-i)-(j'-i')}>1$.
 
\begin{figure}
\centering
\begin{minipage}[b]{.5\linewidth}
\centering
\begin{tikzpicture}[scale=0.7]
\foreach \x in {(1,-1),(2,-1),(3,-1),(4,-1),(5,-1),(6,-1),(1,-2),(2,-2),(3,-2),(4,-2),(5,-2),(6,-2),(1,-3),(2,-3),(3,-3),(4,-3),(5,-3),(6,-3),(1,-4),(2,-4),(3,-4),(4,-4),(5,-4),(1,-5),(2,-5),(3,-5),(1,-6),(2,-6),(3,-6)}{
	\node[draw,circle,inner sep=1pt,outer sep=2pt,fill] at \x {};
	}

\node[draw,circle,inner sep=2pt,outer sep=2pt,fill,red] (t22) at (2,-2) {};
\node[inner sep=1pt,outer sep=2pt] (t21) at (1,-2) {};
\node[inner sep=1pt,outer sep=2pt] (t12) at (2,-1) {};
\node[inner sep=1pt,outer sep=2pt] (t23) at (3,-2) {};
\node[inner sep=1pt,outer sep=2pt] (t32) at (2,-3) {};
\draw[->, red, thick] (t21) -- (t22);
\draw[->, red, thick] (t12) -- (t22);
\draw[->, red, thick] (t23) -- (t22);
\draw[->, red, thick] (t32) -- (t22);
\node at (2.6,-1.5) {\footnotesize \textcolor{red}{$b_{2,2}$}};

\node[draw,circle,inner sep=2pt,outer sep=2pt,fill,blue] (t15) at (5,-1) {};
\node[inner sep=1pt,outer sep=2pt] (t14) at (4,-1) {};
\node[inner sep=1pt,outer sep=2pt] (t25) at (5,-2) {};
\node[inner sep=1pt,outer sep=2pt] (t16) at (6,-1) {};
\draw[->, blue, thick] (t14) -- (t15);
\draw[->, blue, thick] (t25) -- (t15);
\draw[->, blue, thick] (t16) -- (t15);
\node at (5,-0.5) {\footnotesize \textcolor{blue}{$b_{1,5}$}};

\node[draw,circle,inner sep=2pt,outer sep=2pt,fill,green] (t53) at (3,-5) {};
\node[inner sep=1pt,outer sep=2pt] (t52) at (2,-5) {};
\node[inner sep=1pt,outer sep=2pt] (t43) at (3,-4) {};
\draw[->, green, thick] (t52) -- (t53);
\draw[->, green, thick] (t43) -- (t53);
\node at (3.7,-4.8) {\footnotesize \textcolor{green}{$a_{5,3}$}};

\end{tikzpicture}
\subcaption{Commuting local moves}
\label{subfig:localMoves}
\end{minipage}%
\begin{minipage}[b]{.5\linewidth}
\centering
\begin{tikzpicture}[scale=0.7]
\foreach \x in {(1,-1),(2,-1),(3,-1),(4,-1),(5,-1),(6,-1),(1,-2),(2,-2),(3,-2),(4,-2),(5,-2),(6,-2),(1,-3),(2,-3),(3,-3),(4,-3),(5,-3),(6,-3),(1,-4),(2,-4),(3,-4),(4,-4),(5,-4),(1,-5),(2,-5),(3,-5),(1,-6),(2,-6),(3,-6)}{
	\node[draw,circle,inner sep=1pt,outer sep=2pt,fill] at \x {};
	}

\node[draw,circle,inner sep=2pt,outer sep=2pt,fill,red] (t43) at (3,-4) {};
\node[draw,circle,inner sep=2pt,outer sep=2pt,fill,red] (t32) at (2,-3) {};
\node[draw,circle,inner sep=2pt,outer sep=2pt,fill,red] (t21) at (1,-2) {};
\node[inner sep=1pt,outer sep=2pt] (t42) at (2,-4) {};
\node[inner sep=1pt,outer sep=2pt] (t31) at (1,-3) {};
\node[inner sep=1pt,outer sep=2pt] (t44) at (4,-4) {};
\node[inner sep=1pt,outer sep=2pt] (t33) at (3,-3) {};
\node[inner sep=1pt,outer sep=2pt] (t22) at (2,-2) {};
\node[inner sep=1pt,outer sep=2pt] (t11) at (1,-1) {};

\draw[->, red, thick] (t42) -- (t43);
\draw[->, red, thick] (t33) -- (t43);
\draw[->, red, thick] (t31) -- (t32);
\draw[->, red, thick] (t22) -- (t32);
\draw[->, red, thick] (t33) -- (t32);
\draw[->, red, thick] (t42) -- (t32);
\draw[->, red, thick] (t31) -- (t21);
\draw[->, red, thick] (t22) -- (t21);
\draw[->, red, thick] (t11) -- (t21);

\draw[thick,dashed,red] (0.5,-1.5) -- (3.5,-4.5);
\node at (3.8,-4.8) {\footnotesize \textcolor{red}{$\rho_{4,3}$}};

\node[draw,circle,inner sep=2pt,outer sep=2pt,fill,blue] (t45) at (5,-4) {};
\node[draw,circle,inner sep=2pt,outer sep=2pt,fill,blue] (t34) at (4,-3) {};
\node[draw,circle,inner sep=2pt,outer sep=2pt,fill,blue] (t23) at (3,-2) {};
\node[draw,circle,inner sep=2pt,outer sep=2pt,fill,blue] (t12) at (2,-1) {};
\node[inner sep=1pt,outer sep=2pt] (t35) at (5,-3) {};
\node[inner sep=1pt,outer sep=2pt] (t24) at (4,-2) {};
\node[inner sep=1pt,outer sep=2pt] (t13) at (3,-1) {};

\draw[->, blue, thick] (t44) -- (t45);
\draw[->, blue, thick] (t35) -- (t45);
\draw[->, blue, thick] (t44) -- (t34);
\draw[->, blue, thick] (t35) -- (t34);
\draw[->, blue, thick] (t33) -- (t34);
\draw[->, blue, thick] (t24) -- (t34);
\draw[->, blue, thick] (t33) -- (t23);
\draw[->, blue, thick] (t24) -- (t23);
\draw[->, blue, thick] (t13) -- (t23);
\draw[->, blue, thick] (t22) -- (t23);
\draw[->, blue, thick] (t11) -- (t12);
\draw[->, blue, thick] (t22) -- (t12);
\draw[->, blue, thick] (t13) -- (t12);

\draw[thick,dashed,blue] (1.5,-0.5) -- (5.5,-4.5);
\node at (5.8,-4.8) {\footnotesize \textcolor{blue}{$\rho_{4,5}$}};

\end{tikzpicture}
\subcaption{Commuting $\rho_{i,j}$'s}
\label{subfig:diagonalMaps}
\end{minipage}%

\caption{Graphical representation of how local moves $a_{i,j}$'s and $b_{i,j}$'s and maps $\rho_{i,j}$'s that compose the $\grsk$ correspondence act on a polygonal array. The arrows point from a node involved in the definition of a local move to a colored node, which corresponds to the entry that is modified by the local move. One can see that any two local moves commute if they are indexed by lattice vertices that are \emph{not} nearest neighbors, and any two maps $\rho_{i,j}$'s commute if they are indexed by vertices that do \emph{not} belong to neighboring diagonals.}
\label{fig:gRSKlocalMoves}
\end{figure}
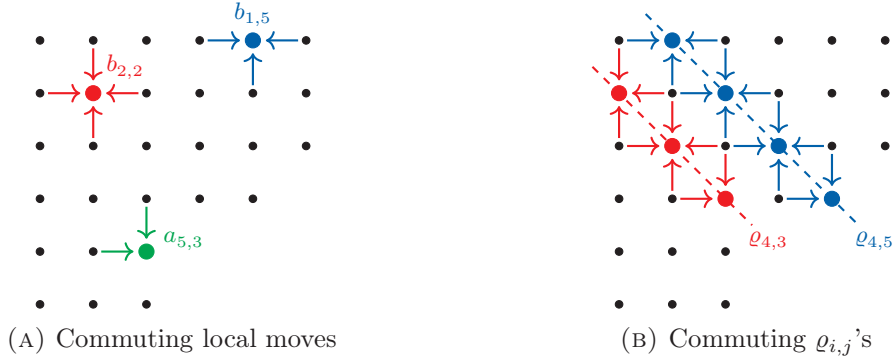
Given a set of indices $\mathcal{I}$, we construct the $\grsk \colon \arrays \to \arrays$ inductively as follows:
We start by $\grsk(\emptyset):= \emptyset$. Let $\mathcal{I}_{\rm out}$ be the set of outer indices of $\mathcal{I}$ and 
$\mathcal{I}^{\circ}:=\mathcal{I}\setminus \mathcal{I}_{\rm out} $. For $\bm{w}\in \arrays$ we define
\begin{align}\label{rsk_constr}
\grsk(\bm{w}) := \bigcirc_{(i,j)\in \mathcal{I}_{\rm out}} \, \rho_{i,j} \big(\grsk(\bm{w}^\circ)\,\sqcup \bm{w}^{\rm out} \big) \, ,
\end{align}
where $\bm{w}^\circ=\{w_{i,j}\colon (i,j)\in \mathcal{I}^\circ\}$, 
$\bm{w}^{\rm out}=\{w_{i,j}\colon (i,j)\in \mathcal{I}_{\rm out}\}$ and $\sqcup$ denotes concatenation.
In words, $\grsk$ is first applied to the array $\bm{w^\circ}$, ignoring entries $\bm{w}^{\rm out}$,
 and then the output is concatenated with the entries $\bm{w}^{\rm out}$ and subsequently
  the maps $\rho_{i,j}, (i,j)\in \cI_{\rm out}$ are applied
  to the new array.
 
We note that since distinct outer indices are never on the same diagonal nor on consecutive diagonals, all 
$\rho_{i,j}$'s indexed by the outer indices of a given array commute and so 
the order in which these maps are composed in \eqref{rsk_constr} is irrelevant.
 
In order to state the main properties of $\grsk$, it is convenient to introduce the following definitions. We denote by $\tau_k$ the product of all 
 elements on the $k$-th diagonal of $\bm{t}$:
\begin{equation}
\label{eq:tau}
\tau_k:=\prod_{\substack{(i,j)\in\mathcal{I} , \\ j-i=k}} t_{i,j} \, .
\end{equation}
We set the \emph{energy} of $\bm{t}$ to be
\begin{equation}
\label{eq:energy}
\mathcal{E}(\bm{t})
:= \frac{1}{t_{1,1}}
+ \sum_{(i,j)\in\mathcal{I}} \frac{t_{i-1,j}+t_{i,j-1}}{t_{i,j}}
\, ,
\end{equation}
with the convention that $t_{i,j}:=0$ when $(i,j)\notin \mathcal{I}$. See Figure~\ref{fig:polygonalArray} for a graphical interpretation of the energy of $\bm{t}$.
\begin{proposition}[{\cite[Prop.~2.6, 2.7]{NZ15}}]
\label{prop:gRSKproperties}
Let $\bm{w}\in\arrays$, $\bm{t} := \grsk(\bm{w})$ and let $(p,q)\in \mathcal{I}$ be a border index. Then
\begin{enumerate}
\item
\label{prop:gRSKproperties_partitionFn}
If $\Pi_{p,q}$ is the set of all directed paths from $(1,1)$ to $(p,q)$, then
\[
t_{p,q} = \sum_{\pi\in\Pi_{p,q}} \prod_{(i,j)\in \pi} w_{i,j} \, .
\]
\item
\label{prop:gRSKproperties_type}
If $(p-1,q)$ is a border index or $p=1$, then
\[
\prod_{j=1}^{q} w_{p,j}
= \frac{\tau_{q-p}}{\tau_{q-p+1}}
\, .
\]
Analogously, if $(p,q-1)$ is a border index or $q=1$, then
\[
\prod_{i=1}^{p} w_{i,q}
= \frac{\tau_{q-p}}{\tau_{q-p-1}}
\, .
\]
\item
\label{prop:gRSKproperties_invWeights}
It holds that
\[
\sum_{(i,j)\in\mathcal{I}} \frac{1}{w_{i,j}}
= \mathcal{E}(\bm{t})
\, .
\]
\item
\label{prop:gRSKproperties_Jacobian}
The transformation
\[
(\log w_{i,j}, \, (i,j)\in\mathcal{I})
\mapsto
(\log t_{i,j}, \, (i,j)\in\mathcal{I})
\]
has Jacobian equal to $\pm 1$.
\end{enumerate}
\end{proposition}

Property~\ref{prop:gRSKproperties_partitionFn} explains how the point-to-point polymer partition functions can be expressed in terms of the 
$\grsk$ correspondence. In light of this connection, the other properties turn out to be useful in computations related to the log-gamma polymer, as it will become clear soon. We also remark that property~\ref{prop:gRSKproperties_type} is easily seen to be equivalent to the following: if $(p,q)$ is a border index, then
\begin{equation}
\label{eq:gRSKprodSubarray}
\prod_{i=1}^p \prod_{j=1}^q w_{i,j} = \tau_{q-p} \, .
\end{equation}

\subsection{Point-to-line polymer}
\label{subsec:P2LpartFn}
The exactly solvable parametrization of the inverse-gamma variables in the point-to-line geometry is given by
\begin{definition}\label{def:log}
An $(\bm{\ga}, \bm{\gb},\gamma)$-log-gamma measure on the lattice $\{(i,j)\colon i+j\leq 2n+1\}$ is the law of a family of 
independent random variables $\{W_{i,j}\colon i+j\leq 2n+1\}$ distributed as follows:
\begin{equation}
\label{eq:logGammaDistribution}
W_{i,j}^{-1} \sim
\begin{cases}
{\rm Gamma}(\alpha_i + \beta_j + \gamma,1) &1\leq i,j\leq n \, , \\
{\rm Gamma}(\alpha_i + \alpha_{2n-j+1},1) &1\leq i\leq n\, , \,\, n < j\leq 2n-i+1 \, , \\
{\rm Gamma}(\beta_{2n-i+1} + \beta_j,1) &1\leq j\leq n \, , \,\, n < i \leq 2n-j+1 \, ,
\end{cases}
\end{equation}
for some $\bm{\alpha},\bm{\beta}\in\R_{+}^n$ and $\gamma\geq 0$. 
\end{definition}

\begin{remark}{\rm
The choice of the parameters in Definition \ref{subsec:P2LpartFn} is tailored so that it fits the link between Whittaker functions and $\grsk$. In particular, this is due to property $(ii)$ in Proposition \ref{prop:gRSKproperties} and the presence of the type of geometric Gelfand-Tsetlin patterns in the integral
formula for Whittaker functions, cf.\ Definitions \ref{def:glWhittakerFn} and \ref{Givental_Whit}. This will become clear in the proofs of Lemmas
\ref{lemma:P2PjointDistr} and \ref{lemma:P2PhalfJointDistr} below.
    We have also included an extra parameter $\gamma$ in the distribution of the weights $W_{i,j}$ for $1\leq i,j\leq n$. This might seem rather unnatural, but it will turn out to be useful in the proof of Theorem~\ref{thm:P2LcontourInt} to obtain contour integral formulae. More specifically, the Plancherel theorem for $\mathfrak{gl}_n$-Whittaker functions can be applied in~\eqref{eq:plancherel1step} thanks to estimation~\eqref{eq:estimationL2}, which in turn relies on the presence of the parameter $\gamma$ in~\eqref{eq:P2LLaplaceTransf}.
}
\end{remark}

We first compute the joint law of all the point-to-point partition functions at a fixed time horizon.
The next proposition is a modification of ~\cite[{Thm 3.5}]{NZ15}, which accommodates the extra parameter $\gamma$ in
 Definition \ref{def:log}. 
\begin{lemma}
\label{lemma:P2PjointDistr} 
For the $(\bm{\alpha},\bm{\beta},\gamma)$-log-gamma polymer, the joint distribution of the point-to-point partition functions at time $2n$ is
\begin{align}\label{eq:P2PjointDistr}
\P(Z_{i,2n+1-i} \in \diff x_i ~~ \forall i=1,\dots,2n)
= \frac{1}{\fG_{\bm{\alpha},\bm{\beta},\gamma}}
\fPhi_{\bm{\alpha},\bm{\beta},\gamma}(\bm{x})
\prod_{i=1}^{2n} \frac{\diff x_i}{x_i}
\end{align}
for $\bm{x} \in\R_{+}^{2n}$.
The normalization constant $\fG_{\bm{\alpha},\bm{\beta},\gamma}$ and the function $\fPhi_{\bm{\alpha},\bm{\beta},\gamma}$ are given by
\begin{align}
\label{eq:P2Lnormalization}
\fG_{\bm{\alpha},\bm{\beta},\gamma}
:= &\prod_{1\leq i,j\leq n} \!\! \Gamma(\alpha_i + \beta_j + \gamma)
\prod_{1\leq i\leq j\leq n} \!\! \Gamma(\alpha_i + \alpha_j) \Gamma(\beta_i + \beta_j) \, , \\
\label{eq:P2LpartFnsDensity}
\fPhi_{\bm{\alpha},\bm{\beta},\gamma}(\bm{x})
:= &\int_{\fT_{2n}(\bm{x})}
\tau_0^{-\gamma}
\prod_{k=1}^n
\bigg(\frac{\tau_{2n-2k+1}^2}{\tau_{2n-2k+2}\tau_{2n-2k}}\bigg)^{-\alpha_k}
\bigg(\frac{\tau_{-2n+2k-1}^2}{\tau_{-2n+2k-2}\tau_{-2n+2k}}\bigg)^{-\beta_k}
e^{-\mathcal{E}(\bm{t})}
\prod_{i+j\leq 2n} \!\! \frac{\diff t_{i,j}}{t_{i,j}},\notag\\
\end{align}
where $\fT_{2n}(\bm{x})$ denotes the set of all triangular arrays $\bm{t} = \{t_{i,j}\colon i+j\leq 2n+1\} \in \R_{+}^{n(2n+1)}$ with positive entries such that $(t_{1,2n},t_{2,2n-1},\dots,t_{2n,1}) = \bm{x}$.
\end{lemma}
\begin{proof}
Set $\mathcal{I}:=\{ (i,j)\in\N\times\N \colon i+j\leq 2n+1\}$. According to~\eqref{eq:logGammaDistribution}, the joint law of the weight triangular array $\bm{W} = \{ W_{i,j}\colon (i,j)\in\mathcal{I}\}$ for the $(\bm{\alpha},\bm{\beta},\gamma)$-log-gamma measure is given by
\begin{equation}
\label{eq:P2LlogGammaDensity}
\begin{split}
\P(\bm{W}\in \diff\bm{w})
= &\prod_{i,j = 1}^n \frac{w_{i,j}^{-\alpha_i-\beta_j-\gamma}}{\Gamma(\alpha_i + \beta_j+\gamma)}
\prod_{\substack{1\leq i\leq n \\ n < j \leq 2n-i+1}}
\!\!\!\!
\frac{w_{i,j}^{-\alpha_i - \alpha_{2n-j+1}}}{\Gamma(\alpha_i + \alpha_{2n-j+1})}
\prod_{\substack{1\leq j\leq n \\ n < i \leq 2n-j+1}}
\!\!\!\!
\frac{w_{i,j}^{-\beta_{2n-i+1} - \beta_j}}{\Gamma(\beta_{2n-i+1} + \beta_j)} \\
&\times \exp\bigg(-\sum_{(i,j)\in\mathcal{I}} \frac{1}{w_{i,j}}\bigg)
\prod_{(i,j)\in\mathcal{I}} \frac{\diff w_{i,j}}{w_{i,j}} \, .
\end{split}
\end{equation}
We now rewrite the above in terms of the image array $\bm{T}= \{ T_{i,j}\colon (i,j)\in\mathcal{I} \}$ of $\bm{W}$ under the $\grsk$ bijection. By formula~\eqref{eq:gRSKprodSubarray}, the product of $w_{i,j}$'s that are raised to power $-\gamma$ is $\prod_{i,j=1}^n w_{i,j} = \tau_0$. Property~\ref{prop:gRSKproperties_type} of Proposition~\ref{prop:gRSKproperties} yields
\[
\prod_{j=1}^{2n-k+1} \!\!\! w_{k,j} = \frac{\tau_{2n-2k+1}}{\tau_{2n-2k+2}} \, , \quad\qquad
\prod_{i=1}^{2n-k+1} \!\!\! w_{i,k} = \frac{\tau_{-2n+2k-1}}{\tau_{-2n+2k-2}}
\]
for $1\leq k\leq 2n$ with the convention that $\tau_{2n}=\tau_{-2n}=1$.
 The product of all $w_{i,j}$'s that are raised to power $-\alpha_k$ in~\eqref{eq:P2LlogGammaDensity} is written as
\[
 \prod_{j=1}^{2n-k+1} \!\!\! w_{k,j} \prod_{i=1}^k w_{i,2n-k+1}
= \frac{\tau_{2n-2k+1}}{\tau_{2n-2k+2}}
\frac{\tau_{2n-2k+1}}{\tau_{2n-2k}}
\, ,
\]
and similarly the product of all $w_{i,j}$'s that are raised to power $-\beta_k$ is written as
\[
 \prod_{i=1}^{2n-k+1} \!\!\! w_{i,k} \prod_{j=1}^k w_{2n-k+1,j}
= \frac{\tau_{-2n+2k-1}}{\tau_{-2n+2k-2}}
\frac{\tau_{-2n+2k-1}}{\tau_{-2n+2k}}
\, .
\]
Using property~\ref{prop:gRSKproperties_invWeights} of Proposition~\ref{prop:gRSKproperties} for dealing with the exponential term in~\eqref{eq:P2LlogGammaDensity}, and property~\ref{prop:gRSKproperties_Jacobian} regarding the volume preserving property of
 the differential form, we obtain:
\[
\begin{split}
\P(\bm{T}\in\diff\bm{t})
= \, &\prod_{1\leq i,j\leq n} \frac{1}{\Gamma(\alpha_i + \beta_j+\gamma)}
\prod_{1\leq i\leq j\leq n} \frac{1}{\Gamma(\alpha_i + \alpha_j) \Gamma(\beta_i + \beta_j)} \\
&\times \tau_0^{-\gamma}
\prod_{k=1}^n
\bigg(\frac{\tau_{2n-2k+1}^2}{\tau_{2n-2k+2}\tau_{2n-2k}}\bigg)^{-\alpha_k}
\bigg(\frac{\tau_{-2n+2k-1}^2}{\tau_{-2n+2k-2}\tau_{-2n+2k}}\bigg)^{-\beta_k}
e^{-\mathcal{E}(\bm{t})}
\prod_{(i,j)\in\mathcal{I}} \frac{\diff t_{i,j}}{t_{i,j}} \, .
\end{split}
\]
Finally, property~\ref{prop:gRSKproperties_partitionFn} of Proposition~\ref{prop:gRSKproperties}  allows one to write the joint law of
 $(Z_{1,2n},Z_{2,2n-1},\dots,Z_{2n,1})$  as in \eqref{eq:P2PjointDistr}
 by integrating the above over $\fT_{2n}(\bm{x})$.
\end{proof}

We can now derive the Whittaker integral formula for the Laplace transform of $\fZ_{2n}$.
\begin{theorem}
\label{thm:P2LLaplaceTransf}
The Laplace transform of the point-to-line partition function $\fZ_{2n}$ for the
$(\bm{\alpha},\bm{\beta},\gamma)$-log-gamma polymer can be written in terms of orthogonal Whittaker functions as:
\begin{equation}
\label{eq:P2LLaplaceTransf}
\E\big[e^{- r \fZ_{2n}}\big] 
= \frac{r^{\sum_{k=1}^{n} (\alpha_k+\beta_k +\gamma)}}
{\fG_{\bm{\alpha},\bm{\beta},\gamma}}
\int_{\R_{+}^n} 
\bigg(\prod_{i=1}^n x_i\bigg)^{\gamma}
e^{-r x_1}
\Psi_{\bm{\alpha}}^{\mathfrak{so}_{2n+1}}(\bm{x})
\Psi_{\bm{\beta}}^{\mathfrak{so}_{2n+1}}(\bm{x})
\prod_{i=1}^n \frac{\diff x_i}{x_i}
\end{equation}
for all $r>0$, where $\fG_{\bm{\alpha},\bm{\beta},\gamma}$ is defined by~\eqref{eq:P2Lnormalization}.
\end{theorem}

\begin{proof}
Given that $\fZ_{2n}=\sum_{i=1}^{2n} Z_{i,2n-i+1}$, we have via Lemma~\ref{lemma:P2PjointDistr} that
\[
\begin{split}
\E\big[e^{- r \fZ_{2n}}\big]
= \E\Big[ e^{-r \sum_{i = 1}^{2n} Z_{i,2n+1-i} } \Big] 
= \frac{1}{\fG_{\bm{\alpha},\bm{\beta},\gamma}}
\int_{\R^{2n}_+} 
\exp\bigg(-r \sum_{i = 1}^{2n} x_i\bigg) \,\,
\fPhi_{\bm{\alpha},\bm{\beta},\gamma}(\bm{x})
\prod_{i=1}^{2n} \frac{\diff x_i}{x_i} \, .
\end{split}
\]
Using definition~\eqref{eq:P2LpartFnsDensity} of $\fPhi_{\bm{\alpha},\bm{\beta}, \gamma}$
and expressing part of the integrand in the right hand side in 
terms of variables $t_{i,j}$'s of the corresponding triangular array (recall notations~\eqref{eq:tau}, \eqref{eq:energy}), we obtain
\[
\begin{split}
\E\big[e^{- r \fZ_{2n}}\big]
= \, &\frac{1}{\fG_{\bm{\alpha},\bm{\beta},\gamma}} 
 \int_{\R^{2n}_+}
\prod_{i=1}^{2n} \frac{\diff t_{i,2n-i+1}}{t_{i,2n-i+1}}
\exp\bigg(-r \sum_{i = 1}^{2n} t_{i,2n-i+1}\bigg) \\
&\times \int_{\R^{n(2n-1)}_+}
\tau_0^{-\gamma}
\prod_{k=1}^n
\bigg(\frac{\tau_{2n-2k+1}^2}{\tau_{2n-2k+2}\tau_{2n-2k}}\bigg)^{-\alpha_k}
\bigg(\frac{\tau_{-2n+2k-1}^2}{\tau_{-2n+2k-2}\tau_{-2n+2k}}\bigg)^{-\beta_k} \\
&\times \exp\bigg(-\frac{1}{t_{1,1}}
- \sum_{i>1,j} \frac{t_{i-1,j}}{t_{i,j}}
- \sum_{j>1,i} \frac{t_{i,j-1}}{t_{i,j}} \bigg)
\prod_{i+j\leq 2n} \frac{\diff t_{i,j}}{t_{i,j}}
\, ,
\end{split}
\]
where the implicit range of indices for $(i,j)$ is $i+j\leq 2n+1$.
We now change variables by setting 
\begin{align}\label{change}
t_{i,j} = (r s_{i,j})^{-1}, \qquad \text{for all}\quad  (i,j) .
\end{align}
Pictorially, this change of variables reverses the arrows in Figure \ref{subfig:triangularArray}. 
This can be seen if one recalls (see caption of Figure  \ref{fig:polygonalArray}) that
an ``arrow'' $t_{i,j}\to t_{i+1,j}$ or $t_{i,j}\to t_{i,j+1}$ in the figure represents a summand 
$t_{i,j}/t_{i+1,j}$ or $t_{i,j} / t_{i,j+1}$, respectively, in the functional  $\mathcal{E}(\bm{t})$. The change of 
variables \eqref{change} will transform these ratios to $t_{i+1,j}/t_{i,j}$ and $ t_{i,j+1} / t_{i,j} $, which by the
same convention can be represented in the diagrams as $t_{i+1,j} \to t_{i,j}$ and $ t_{i,j+1} \to t_{i,j} $.

Recalling, now, \eqref{eq:tau} we obtain
\[
\begin{split}
\bigg(\frac{\tau_{2n-2k+1}^2}{\tau_{2n-2k+2}\tau_{2n-2k}}\bigg)^{-\alpha_k}
&= \bigg(\frac{\prod_{j-i=2n-2k+1} t_{i,j}^2}
{\prod_{j-i=2n-2k+2} t_{i,j} \prod_{j-i=2n-2k} t_{i,j}} \bigg)^{-\alpha_k} \\
&= r^{\alpha_k} \bigg(\frac{\prod_{j-i=2n-2k+1} s_{i,j}^2}
{\prod_{j-i=2n-2k+2} s_{i,j} \prod_{j-i=2n-2k} s_{i,j}} \bigg)^{\alpha_k}
= r^{\alpha_k} \bigg(\frac{\sigma_{2n-2k+1}^2}{\sigma_{2n-2k+2} \sigma_{2n-2k}} \bigg)^{\alpha_k} \, ,
\end{split}
\]
where analogously to \eqref{eq:tau} we set $\sigma_k:=\prod_{j-i=k} s_{i,j}$.
Similarly, we have
\[
\bigg(\frac{\tau_{-2n+2k-1}^2}{\tau_{-2n+2k-2}\tau_{-2n+2k}}\bigg)^{-\beta_k}
= r^{\beta_k} \bigg(\frac{\sigma_{-2n+2k-1}^2}{\sigma_{-2n+2k-2} \sigma_{-2n+2k}} \bigg)^{\beta_k} \, 
\quad\text{and} \quad
\tau_0^{-\gamma}
= r^{n\gamma} \sigma_0^{\gamma} \, .
\]
Moreover, the volume is preserved under this change of variables, that is
\[
\prod_{i+j\leq 2n} \frac{\diff t_{i,j}}{t_{i,j}}
= \prod_{i+j\leq 2n} \frac{\diff s_{i,j}}{s_{i,j}} \, .
\]
We thus obtain
\[
\begin{split}
\E\big[e^{- r \fZ_{2n}}\big] 
= \, &\frac{r^{\sum_{k=1}^{n}(\alpha_k + \beta_k + \gamma)}  }{\fG_{\bm{\alpha},\bm{\beta},\gamma}}
\int_{\R^{2n}_+}
\prod_{i=1}^{2n} \frac{\diff s_{i,2n-i+1}}{s_{i,2n-i+1}} \,\,
\exp\Big(- \sum_{i=1}^n \frac{1}{s_{i,2n-i+1}}\Big) \\
&\times \int_{\R^{n(2n-1)}_+}
\sigma_0^{\gamma}
\prod_{k=1}^n
\bigg(\frac{\sigma_{2n-2k+1}^2}{\sigma_{2n-2k+2}\sigma_{2n-2k}}\bigg)^{\alpha_k}
\bigg(\frac{\sigma_{-2n+2k-1}^2}{\sigma_{-2n+2k-2}\sigma_{-2n+2k}}\bigg)^{\beta_k} \\
&\times 
\exp\bigg( -r s_{1,1}
- \sum_{i>1,j} \frac{s_{i,j}}{s_{i-1,j}}
- \sum_{j>1,i} \frac{s_{i,j}}{s_{i,j-1}} \bigg)
\prod_{i+j\leq 2n} \frac{\diff s_{i,j}}{s_{i,j}}
\, .
\end{split}
\]
We now change the order in which variables are integrated in the above expression: we first integrate over the two triangular arrays $\{s_{i,j}\}_{i<j}$ and $\{s_{i,j}\}_{j<i}$ into which the whole triangular shape (see Figure~\ref{subfig:triangularArray}) is divided by the main diagonal
$\{(i,i)\colon i\geq 1\}$; next, we integrate over the diagonal variables $s_{1,1},\dots,s_{n,n}$. This way, we obtain:
\[
\begin{split}
&\E\Big[e^{- r \fZ_{2n}}\Big]
= \frac{r^{\sum_{k=1}^n(\alpha_k+\beta_k+\gamma)}}{\fG_{\bm{\alpha},\bm{\beta},\gamma}}
\int_{\R^n_+}
\prod_{i=1}^n \frac{\diff s_{i,i}}{s_{i,i}}
\bigg(\prod_{i=1}^n s_{i,i}\bigg)^{\gamma}
e^{-r s_{1,1} } \\
&\times  \!\! \int_{\R^{n^2}_+}
\prod_{k=1}^n
\bigg(\frac{\sigma_{2n-2k+1}^2}{\sigma_{2n-2k+2}\sigma_{2n-2k}}\bigg)^{\alpha_k}
\exp\bigg(
- \sum_{i=1}^n \frac{1}{s_{i,2n+1-i}}
- \sum_{1<i\leq j} \frac{s_{i,j}}{s_{i-1,j}}
- \sum_{i<j} \frac{s_{i,j}}{s_{i,j-1}} \bigg)
\prod_{i<j} \frac{\diff s_{i,j}}{s_{i,j}} \\
&\times \!\! \int_{\R^{n^2}_+}
\prod_{k=1}^n
\bigg(\frac{\sigma_{-2n+2k-1}^2}{\sigma_{-2n+2k-2}\sigma_{-2n+2k}}\bigg)^{\beta_k}
\exp\bigg(
- \sum_{j=1}^n \frac{1}{s_{2n+1-j,j}}
- \sum_{j<i} \frac{s_{i,j}}{s_{i-1,j}}
- \sum_{1<j\leq i} \frac{s_{i,j}}{s_{i,j-1}} \bigg)
\prod_{j<i} \frac{\diff s_{i,j}}{s_{i,j}} \, .
\end{split}
\]
Comparing with Definition~\ref{def:soWhittakerFn}, we identify the second and the third integral in the above formula as $\mathfrak{so}_{2n+1}$-Whittaker functions, both with shape variables $s_{1,1},\dots,s_{n,n}$ and parameters $\bm{\alpha}$ and $\bm{\beta}$ respectively. This concludes the proof of~\eqref{eq:P2LLaplaceTransf}.
\end{proof}

\subsection{Point-to-half-line polymer}
\label{subsec:P2HLpartFn}
The only weights involved in the point-to-half-line partition function are $\{ W_{i,j} \colon i\leq n, \, i+j\leq 2n+1\}$, see Figure
\ref{subfig:P2half-Lpath}. When restricted to such a trapezoidal array, the $(\bm{\alpha},\bm{\beta},\gamma)$-log-gamma measure ~\eqref{eq:logGammaDistribution} coincides with the law of a $(\bm{\alpha},\bm{\beta}+\gamma,0)$-log-gamma measure.
 For this reason, in the following, we will assume without loss of generality that $\gamma=0$ and we will
  refer to the corresponding model as the half-flat $(\bm{\alpha},\bm{\beta})$-log-gamma polymer.
We first give an expression for the joint law of the point-to-point partition functions on a ``half-line''. Again, the proof will be based on Theorem~\ref{prop:gRSKproperties}.

\begin{lemma}
\label{lemma:P2PhalfJointDistr}
For the half-flat $(\bm{\alpha},\bm{\beta})$-log-gamma polymer, the joint distribution of the point-to-point partition functions on the half-line $\{i+j=2n+1,\,~ i\leq j\}$ is 
\[
\P(\hZ_{i,2n+1-i} \in \diff x_i ~~ \forall i=1,\dots,n)
=
\frac{1}{\hG_{\bm{\alpha},\bm{\beta}}}
\hPhi_{\bm{\alpha},\bm{\beta}}(\bm{x})
\prod_{i=1}^n \frac{\diff x_i}{x_i}
\]
for $\bm{x}\in\R_{+}^n$. The normalization constant $\hG_{\bm{\alpha},\bm{\beta}}$ and the function $\hPhi_{\bm{\alpha},\bm{\beta}}$ are given by
\begin{align}
\label{eq:P2HLnormalization}
\hG_{\bm{\alpha},\bm{\beta}}
&:= \prod_{1\leq i,j\leq n} \!\! \Gamma(\alpha_i + \beta_j)
\prod_{1\leq i\leq j\leq n} \!\! \Gamma(\alpha_i + \alpha_j) \, , \\
\label{eq:P2HLpartFnsDensity}
\hPhi_{\bm{\alpha},\bm{\beta}}(\bm{x})
&:= \int_{\hT_{\,2n}(\bm{x})}
\prod_{k=1}^n
\bigg(\frac{\tau_{2n-2k+1}^2}{\tau_{2n-2k+2}\tau_{2n-2k}}\bigg)^{-\alpha_k}
\Big( \frac{\tau_{-n+k}}{\tau_{-n+k-1}} \Big)^{-\beta_{k}}
e^{-\mathcal{E}(\bm{t})}
\prod_{\substack{i\leq n \\ i+j\leq 2n}} \frac{\diff t_{i,j}}{t_{i,j}}
\, ,
\end{align}
where $\hT_{\,2n}(\bm{x})$ denotes the set of all trapezoidal arrays $\bm{t} = \{ t_{i,j} \colon i\leq n, \, i+j\leq 2n+1\} \in \R_{+}^{(3n^2+n)/2}$ with positive entries such that $(t_{1,2n},t_{2,2n-1},\dots,t_{n,n+1}) = \bm{x}$.
\end{lemma}

\begin{proof}
Setting $\mathcal{I}:=\{ (i,j)\in\N\times\N \colon i\leq n , \, i+j\leq 2n+1\}$, the joint law of the weight trapezoidal array $\bm{W} = \{ W_{i,j} \colon (i,j)\in\mathcal{I}\}$ for the half-flat $(\bm{\alpha},\bm{\beta})$-log-gamma polymer is given by
\begin{equation}
\begin{split}
\label{eq:trapezoidLogGammaWeights}
\P(\bm{W}\in \diff\bm{w})
= &\prod_{1\leq i,j\leq n} \frac{w_{i,j}^{-\alpha_i-\beta_j}}{ \Gamma(\alpha_i + \beta_j)}
\prod_{\substack{1\leq i\leq n \\ n < j \leq 2n-i+1}}
\!\!\!\!
\frac{w_{i,j}^{-\alpha_i - \alpha_{2n-j+1}}}{\Gamma(\alpha_i + \alpha_{2n-j+1})} \\
&\times \exp\bigg(-\sum_{(i,j)\in\mathcal{I}} \frac{1}{w_{i,j}}\bigg)
\prod_{(i,j)\in\mathcal{I}} \frac{\diff w_{i,j}}{w_{i,j}} \, .
\end{split}
\end{equation}
We now rewrite the above in terms of the image array $\bm{T}= \{ T_{i,j} \colon (i,j)\in\mathcal{I} \}$ of $\bm{W}$ under the $\grsk$ bijection. The powers of $w_{i,j}$'s are sorted out by noting that
\[
\prod_{j=1}^{2n-k+1} w_{k,j} = \frac{\tau_{2n-2k+1}}{\tau_{2n-2k+2}} \, , \quad\qquad
\prod_{i=1}^{n} w_{i,k} = \frac{\tau_{-n+k}}{\tau_{-n+k-1}} \, , \quad\qquad
\prod_{i=1}^{k} w_{i,2n-k+1} = \frac{\tau_{2n-2k+1}}{\tau_{2n-2k}}
\]
for $1\leq k\leq n$, thanks to property~\ref{prop:gRSKproperties_type} of Proposition~\ref{prop:gRSKproperties}. Using property~\ref{prop:gRSKproperties_invWeights} for dealing with the exponential in~\eqref{eq:trapezoidLogGammaWeights}, and property~\ref{prop:gRSKproperties_Jacobian} for the differential form, we obtain:
\[
\begin{split}
\P(\bm{T}\in\diff\bm{t})
= \, &\prod_{1\leq i,j\leq n} \frac{1}{ \Gamma(\alpha_i + \beta_j)}
\prod_{1\leq i\leq j\leq n} \frac{1}{ \Gamma(\alpha_i + \alpha_j)} \\
&\times \prod_{k=1}^n
\bigg(\frac{\tau_{2n-2k+1}^2}{\tau_{2n-2k+2}\tau_{2n-2k}}\bigg)^{-\alpha_k}
\Big( \frac{\tau_{-n+k}}{\tau_{-n+k-1}} \Big)^{-\beta_{k}}
e^{-\mathcal{E}(\bm{t})}
\prod_{(i,j)\in\mathcal{I}} \frac{\diff t_{i,j}}{t_{i,j}} \, .
\end{split}
\]
Finally, property~\ref{prop:gRSKproperties_partitionFn} allows writing the joint law of $(Z_{1,2n},Z_{2,2n-1},\dots,Z_{n,n+1})$ by integrating the above over $\hT_{2n}(\bm{x})$.
\end{proof}

\begin{theorem}
\label{thm:P2HLLaplaceTransf}
The Laplace transform of the point-to-half-line partition function $\hZ_{2n}$ for the half-flat
$(\bm{\alpha},\bm{\beta})$-log-gamma polymer can be written in terms of Whittaker functions as:
\begin{equation}
\label{eq:P2HLLaplaceTransf}
\E\Big[e^{- r \hZ_{\,2n}}\Big] 
= \frac{r^{\sum_{k=1}^{n} (\alpha_k+\beta_k)}}
{\hG_{\bm{\alpha},\bm{\beta}}}
\int_{\R_{+}^n} e^{-r x_1}
\Psi_{\bm{\alpha}}^{\mathfrak{so}_{2n+1}}(\bm{x})
\Psi_{\bm{\beta}}^{\mathfrak{gl}_n}(\bm{x})
\prod_{i=1}^n \frac{\diff x_i}{x_i}
\end{equation}
for all $r>0$, where $\hG_{\bm{\alpha},\bm{\beta}}$ is given by~\eqref{eq:P2HLnormalization}.
\end{theorem}

\begin{proof}
Given that $\hZ_{2n}=\sum_{i=1}^{n} Z_{i,2n-i+1}$, we have via Lemma~\ref{lemma:P2PhalfJointDistr} that
\[
\begin{split}
\E\Big[e^{- r \hZ_{\,2n}}\Big]
= \E\Big[ e^{-r \sum_{i = 1}^{n} Z_{i,2n+1-i} } \Big]
= \frac{1}{\hG_{\bm{\alpha},\bm{\beta}}}
\int_{\R^n_+} 
\exp\bigg(-r \sum_{i = 1}^n x_i\bigg) \,\,
\hPhi_{\bm{\alpha},\bm{\beta}}(\bm{x})
\prod_{i=1}^n \frac{\diff x_i}{x_i} \, .
\end{split}
\]
Using definition~\eqref{eq:P2HLpartFnsDensity} of $\hPhi_{\bm{\alpha},\bm{\beta}}$ and performing the same change of variables 
\eqref{change} as in the proof of Theorem~\ref{thm:P2LLaplaceTransf}, we obtain
\[
\begin{split}
\E\Big[e^{- r \hZ_{\,2n}}\Big] 
= \, &\frac{r^{\sum_{k=1}^n(\alpha_k+\beta_k)}}{\hG_{\bm{\alpha},\bm{\beta}}}
\int_{\R^{n}_+}
\prod_{i=1}^n \frac{\diff s_{i,2n+1-i}}{s_{i,2n+1-i}} \,\,
\exp\Big(- \sum_{i = 1}^n \frac{1}{s_{i,2n+1-i}}\Big) \\
&\times \int_{\R^{(3n^2-n)/2}_+}
\prod_{k=1}^{n}
\bigg(\frac{\sigma_{2n-2k+1}^2}{\sigma_{2n-2k+2}\sigma_{2n-2k}}\bigg)^{\alpha_k}
\Big( \frac{\sigma_{-n+k} }{\sigma_{-n+k-1}}   \Big)^{\beta_k} \\
&\times 
\exp\bigg( -r s_{1,1}
- \sum_{i>1,j} \frac{s_{i,j}}{s_{i-1,j}}
- \sum_{j>1,i} \frac{s_{i,j}}{s_{i,j-1}} \bigg)
\prod_{\substack{i\leq n \\ i+j\leq 2n}} \frac{\diff s_{i,j}}{s_{i,j}}
\, ,
\end{split}
\]
where the implicit range of indices for $(i,j)$ is $i\leq n$, $i+j\leq 2n+1$, and we set $\sigma_k := \prod_{j-i=k} s_{i,j}$.
We now change the order in which variables are integrated in the above expression: we first integrate over the two triangular arrays $\{s_{i,j}\}_{i<j}$ and $\{s_{i,j}\}_{j<i}$ into which the whole trapezoidal shape (see Figure~\ref{subfig:P2half-Lpath}) is divided by the main diagonal
$\{(i,i)\colon i\geq 1\}$; next, we integrate over the diagonal variables $s_{1,1},\dots,s_{n,n}$. In this way we obtain:
\[
\begin{split}
&\E\Big[e^{- r \hZ_{\,2n}}\Big]
= \frac{r^{\sum_{k=1}^n(\alpha_k+\beta_k)}}{\hG_{\bm{\alpha},\bm{\beta}}}
\int_{\R^n_+}
\prod_{i=1}^n \frac{\diff s_{i,i}}{s_{i,i}}
e^{-r s_{1,1} } \\
&\times \int_{\R^{n^2}_+}
\prod_{k=1}^n
\bigg(\frac{\sigma_{2n-2k+1}^2}{\sigma_{2n-2k+2}\sigma_{2n-2k}}\bigg)^{\alpha_k}
\exp\bigg(
- \sum_{i=1}^n \frac{1}{s_{i,2n+1-i}}
- \sum_{1<i\leq j} \frac{s_{i,j}}{s_{i-1,j}}
- \sum_{i<j} \frac{s_{i,j}}{s_{i,j-1}} \bigg)
\prod_{i<j} \frac{\diff s_{i,j}}{s_{i,j}} \\
&\times \int_{\R^{n(n-1)/2}_+}
\prod_{k=1}^n
\Big(\frac{\sigma_{-n+k} }{\sigma_{-n+k-1} }\Big)^{\beta_k}
 \exp\bigg(
- \sum_{j<i} \frac{s_{i,j}}{s_{i-1,j}}
- \sum_{1<j\leq i} \frac{s_{i,j}}{s_{i,j-1}} \bigg)
\prod_{j<i} \frac{\diff s_{i,j}}{s_{i,j}} \, .
\end{split}
\]
Comparing with Definition~\ref{def:soWhittakerFn} and~\ref{def:glWhittakerFn}, we identify the second integral as an $\mathfrak{so}_{2n+1}$-Whittaker function with parameters $\bm{\alpha}$, and the third integral as a $\mathfrak{gl}_n$-Whittaker function with parameters $\bm{\beta}$, both with shape variables $s_{1,1},\dots,s_{n,n}$. This concludes the proof of~\eqref{eq:P2HLLaplaceTransf}.
\end{proof}

\begin{remark}{\rm
Taking the limit $r\to 0$ in~\eqref{eq:P2HLLaplaceTransf} and since $\E\big[e^{-r \hZ_{\,2n}} \big] \to 1$ and $r^{\sum_{k=1}^{n} (\alpha_k+\beta_k)} \to 0$, we observe that the integral
\[
\int_{\R_{+}^n} 
\Psi_{\bm{\alpha}}^{\mathfrak{so}_{2n+1}}(\bm{x})
\Psi_{\bm{\beta}}^{\mathfrak{gl}_{n}}(\bm{x})
\prod_{i=1}^n \frac{\diff x_i}{x_i}
\]
diverges for all $\bm{\alpha},\bm{\beta}\in\R_{+}^n$. This does not contradict Ishii-Stade identity, as in~\eqref{eq:IshiiStade} the parameters of the $\mathfrak{gl}_n$-Whittaker function are required to have negative real part.}
\end{remark}

\subsection{Restricted and symmetric point-to-line polymers}
\label{sec:restricted}

We now study the point-to-line polymer restricted to stay in a half plane, i.e.\ not allowed to go below the main diagonal. We 
approach this by noting that the {\it restricted} polymer is closely connected to a symmetric polymer, i.e.\ a polymer whose weight array $\bm{W}=\{W_{i,j}, \, i+j\leq 2n+1\}$ satisfies $W_{i,j}=W_{j,i}$ for all $i,j$. Indeed, for each time a given restricted polymer path touches the main diagonal $i=j$ (including the starting point $(1,1)$), the symmetric polymer partition function counts the weight of that path twice. Therefore, the partition function of a symmetric polymer is given by:
\[
Z^{\mathrm{sym}}_{2n} = \sum_{\pi\in\rPi_{2n}} 2^{\# \{i\colon (i,i)\in\pi\} } \prod_{(i,j)\in \pi} W_{i,j}
= \sum_{\pi\in\rPi_{2n}} \prod_{(i,j)\in \pi} (1+\delta_{i,j})W_{i,j}
\, ,
\]
where $\rPi_{2n}$ is the set of all paths $\pi$ of length $2n$ starting from $(1,1)$ such that $i\leq j$ for all $(i,j)\in\pi$ (``restricted'' paths). It follows that the partition functions of the symmetric and restricted polymers are equal in distribution when the weights of the restricted polymer are doubled on the diagonal. To see what this practically means in the log-gamma case, we refer the reader to Definition~\ref{def:restrLogGammaDistribution} and Remark~\ref{rem:symm-restrPolymer}. Without loss of generality, we may then restrict ourselves to study the symmetric polymer.

We first give the natural definitions of transposition and symmetry in this setting. We define the \emph{transpose} of an index set $\mathcal{I}$ as the index set $\mathcal{I}^\top := \{(i,j)\in\N\times\N \colon (j,i)\in\mathcal{I}\}$; similarly, we define the \emph{transpose} $\bm{t}^\top$ of a polygonal array $\bm{t}$ by setting $t^\top_{i,j} := t_{j,i}$ for all $(i,j)\in\mathcal{I}^\top$.
An index set $\mathcal{I}$ will be called \emph{symmetric} if $\mathcal{I}=\mathcal{I}^\top$, and a polygonal array $\bm{t}$ indexed by a symmetric $\mathcal{I}$ will be called \emph{symmetric} if $\bm{t}=\bm{t}^\top$.

Properties $(i),(ii),(iii)$ in Proposition \ref{prop:gRSKproperties}, for the $\grsk$ 
with respect to input arrays without any symmetry constraint,
transfer directly to the case of symmetric arrays. The volume preserving property is also satisfied: 
\begin{proposition}
\label{prop:symmetricgRSK}
Let $\bm{w}\in\arrays$ 
 and $\bm{t} := \grsk(\bm{w})$. Then $\grsk\big(\bm{w}^\top\big) = \bm{t}^\top$. In particular, if $\bm{w}$ is symmetric, so is $\bm{t}$. 
 Moreover, in the symmetric case, the transformation
\begin{align}\label{eq:symmetricgRSK}
(\log w_{i,j} \colon (i,j)\in\mathcal{I}, \, i\leq j)
\mapsto
(\log t_{i,j} \colon (i,j)\in\mathcal{I}, \, i\leq j)
\end{align}
has Jacobian equal to $\pm 1$.
\end{proposition}
\begin{proof}
The fact that $\grsk\big(\bm{w}^\top\big)=\grsk(\bm{w})^\top$ is an easy consequence of the inductive construction~\eqref{rsk_constr} of $\grsk$, since local moves are clearly symmetric, in the sense that $a_{j,i}\big(\bm{w}^\top\big) = a_{i,j}(\bm{w})^\top$, and the same holds for $b_{i,j}$.
Let us now check the volume preserving property in the case of symmetric $\bm{w}$. In the case of a \emph{square} symmetric array $\bm{w}$, this has been already checked in~\cite[Thm~5.2]{OSZ14}.
On the other hand, every symmetric array $\bm{w}$ will contain a subarray $\bm{w}|_{\mathcal{I}'}$, where $\mathcal{I}'$ is the biggest square subset of $\mathcal{I}$ with upper left index $(1,1)$.
By the inductive construction~\eqref{rsk_constr} of $\grsk$, we can obtain the $\grsk$ image of $\bm{w}$ by first applying the $\grsk$ mapping
to $\bm{w}|_{\mathcal{I}'}$ and then insert the rest of the entries via a suitable sequence of mappings $\rho_{k,l}$ with $(k,l)\in\mathcal{I}\setminus\mathcal{I}'$.
Since $\mathcal{I}'$ is square, we know that the claim holds for $\mathcal{I}'$, i.e.\ the transformation
\begin{align*}
 (\log w_{i,j} \colon (i,j)\in\mathcal{I}', \, i\leq j)
\mapsto
(\log t_{i,j} \colon (i,j)\in\mathcal{I}', \, i\leq j) 
\end{align*}
 is volume preserving.
Next, we apply the sequence of moves $\rho_{k,l}$  for
$(k,l)\in\mathcal{I}\setminus\mathcal{I}'$ with $k\leq l$ in the order specified by the inductive construction of $\grsk$. 
Now, every $(k,l)\in\mathcal{I}\setminus\mathcal{I}'$ with $k\leq l$ actually satisfies $k<l$ hence, crucially, the corresponding mapping $\rho_{k,l}$ does \emph{not} involve any symmetric
 variables, i.e.\ acts on the entries indexed by $i\leq j$ only and does {\it not} involve entries indexed by $i> j$.
It follows that, since all mappings $\rho_{k,l}$'s are volume preserving in logarithmic variables (as local moves trivially are), after applying all $\rho_{k,l}$'s with $(k,l)\in\mathcal{I}\setminus\mathcal{I}'$, the volume of the upper ``triangular'' part of the array is still preserved, thus leading to the volume
preserving property of the map \eqref{eq:symmetricgRSK}.
\end{proof}

The exactly solvable distribution on symmetric arrays
that links to $\mathfrak{so}_{2n+1}$-Whittaker functions is given by
\begin{definition}
\label{def:restrLogGammaDistribution}
 For a triangular index set $\{i+j\leq 2n+1\}$ and a symmetric weight array $\bm{W}=\{W_{i,j} \colon i+j\leq 2n+1\}$,
 we define the symmetric $(\bm{\alpha},\gamma)$-log-gamma measure to be the law on $\bm{W}$ when the entries 
on or above the diagonal are independent and distributed as
\begin{equation}
\label{eq:restrLogGammaDistribution}
W_{i,j}^{-1} \sim
\begin{cases}
{\rm Gamma}(\alpha_i + \gamma,1/2) &1\leq i=j \leq n \, , \\
{\rm Gamma}(\alpha_i + \alpha_j + 2\gamma,1) &1\leq i<j\leq n \, , \\
{\rm Gamma}(\alpha_i + \alpha_{2n-j+1},1) &1\leq i\leq n, \, n < j\leq 2n-i+1 \, ,
\end{cases}
\end{equation}
for some $\bm{\alpha}\in\R_+^n$ and $\gamma\geq 0$. 
We will refer to the directed polymer on such arrays as the symmetric $(\bm{\alpha},\gamma)$-log-gamma polymer.
\end{definition}
Let us note that the joint law of the upper entries of $\bm{W}$ is
\begin{equation}
\label{eq:symP2LlogGammaDensity}
\begin{split}
&\P(W_{i,j}\in \diff w_{i,j} \,\,\, \forall i\leq j)
= \prod_{i=1}^n \frac{w_{i,i}^{-\alpha_i - \gamma}}
{2^{\alpha_i+\gamma} \Gamma(\alpha_i+\gamma)}
\prod_{1\leq i< j\leq n} \frac{w_{i,j}^{-\alpha_i-\alpha_j-2\gamma}}{\Gamma(\alpha_i + \alpha_j+2\gamma)} \\
&\qquad \times \prod_{\substack{1\leq i\leq n \\ n < j \leq 2n-i+1}}
\frac{w_{i,j}^{-\alpha_i - \alpha_{2n-j+1}}}{\Gamma(\alpha_i + \alpha_{2n-j+1})} 
 \exp\bigg(-\sum_{i=1}^n \frac{1}{2w_{i,i}}
-\sum_{i<j} \frac{1}{w_{i,j}}\bigg)
\prod_{i\leq j} \frac{\diff w_{i,j}}{w_{i,j}} \, .
\end{split}
\end{equation}
\begin{remark}
\label{rem:symm-restrPolymer}
{\rm
Based on the discussion on the relation between symmetric and restricted polymer, we can easily conclude that
the exactly solvable measure for a restricted polymer is deduced from the symmetric law and amounts to only modifying the law of the
diagonal entries of the latter to $W_{i,i}^{-1} \sim {\rm Gamma}(\alpha_i + \gamma,1) $. We call this measure the restricted $(\bm{\alpha},\gamma)$-log-gamma measure and the corresponding restricted polymer the restricted $(\bm{\alpha},\gamma)$-log-gamma polymer.}
\end{remark}
Due to the symmetry of the output array $\bm{T}=\grsk(\bm{W})$, which follows from Proposition \ref{prop:symmetricgRSK}, we have that 
\begin{equation}
\label{eq:symmP2LpartFn}
Z^{\mathrm{sym}}_{2n} = 2 \sum_{i=1}^n Z^{\mathrm{sym}}_{i,2n-i+1} \, .
\end{equation}
The next proposition gives the joint law of the point-to-point partition functions in the right hand side.
The proof follows the same steps as that of Lemma \ref{lemma:P2PjointDistr} 
 (up to incorporating appropriately the symmetry condition) and so we omit it.
\begin{lemma}
\label{lemma:symP2PjointDistr}
For the symmetric $(\bm{\alpha},\gamma)$-log-gamma polymer, the joint law of the point-to-point partition functions at time $2n$ is
\[
\P(Z^{\mathrm{sym}}_{i,2n+1-i} \in 2^{-1} \diff x_i ~~ \forall i=1,\dots,n)
=
\frac{1}{\rG_{\bm{\alpha},\gamma}}
\rPhi_{\bm{\alpha},\gamma}(\bm{x})
\prod_{i=1}^n \frac{\diff x_i}{x_i}
\]
for $\bm{x}\in\R_+^n$.
 The normalization constant $\rG_{\bm{\alpha},\gamma}$ and the function $\rPhi_{\bm{\alpha},\gamma}$ are given by
\begin{align}
\label{eq:symP2Lnormalization}
\rG_{\bm{\alpha},\gamma}
:= &\prod_{i=1}^n \Gamma(\alpha_i + \gamma)
\prod_{1\leq i < j\leq n} \!\! \Gamma(\alpha_i + \alpha_j +2\gamma)
\prod_{1\leq i \leq j\leq n} \!\!
\Gamma(\alpha_i + \alpha_j) \, , \\
\label{eq:symP2LpartFnsDensity}
\begin{split}
\rPhi_{\bm{\alpha},\gamma}(\bm{x})
:= &\int_{\rT_{2n}(\bm{x})}
\tau_0^{-\gamma}
\prod_{k=1}^n
\bigg(\frac{\tau_{2n-2k+1}^2}{\tau_{2n-2k+2}\tau_{2n-2k}}\bigg)^{-\alpha_k} \\
&\times \exp\bigg(-\frac{1}{t_{1,1}} -\sum_{1<i\leq j} \frac{t_{i-1,j}}{t_{i,j}} - \sum_{i<j} \frac{t_{i,j-1}}{t_{i,j}} \bigg)
\prod_{\substack{i\leq j, \\ i+j\leq 2n}} \frac{\diff t_{i,j}}{t_{i,j}}
\, ,
\end{split}
\end{align}
where the implicit index range is $i+j\leq 2n+1, i\leq j$, and $\rT_{2n}(\bm{x})$ denotes the set of all triangular arrays 
$\{ t_{i,j} \colon i+j\leq 2n+1, \, i\leq j\} \in \R_+^{n(n+1)}$ with positive entries such that $(t_{1,2n},t_{2,2n-1},\dots,t_{n,n+1}) = \bm{x}$.
\end{lemma}

Lemma \ref{lemma:symP2PjointDistr} allows us
  to obtain a Whittaker integral formula for the Laplace transform of $Z^{\mathrm{sym}}_{2n}$ and $\rZ_{2n}$ as stated in the 
  next theorem. The proof is omitted as it follows the same steps as in Theorems 
  \ref{thm:P2LLaplaceTransf} and \ref{thm:P2HLLaplaceTransf}.
\begin{theorem}
\label{thm:symP2LLaplaceTransf}
The Laplace transform of the point-to-line partition functions for the symmetric $(\bm{\alpha},\gamma)$-log-gamma polymer and the restricted $(\bm{\alpha},\gamma)$-log-gamma polymer, denoted by $Z^{\mathrm{sym}}_{2n}$ and $\rZ_{2n}$ respectively, can be written in terms of orthogonal Whittaker functions as:
\begin{equation}
\label{eq:symP2LLaplaceTransf}
\E\Big[e^{- r \rZ_{2n}}\Big]
= \E\Big[e^{- r Z^{\mathrm{sym}}_{2n}}\Big] 
= \frac{r^{\sum_{k=1}^{n} (\alpha_k +\gamma)}}
{\rG_{\bm{\alpha},\gamma}}
\int_{\R_+^n} 
\bigg(\prod_{i=1}^n x_i\bigg)^{\gamma}
e^{-r x_1}
\Psi_{\bm{\alpha}}^{\mathfrak{so}_{2n+1}}(\bm{x})
\prod_{i=1}^n \frac{\diff x_i}{x_i}
\end{equation}
for all $r>0$, where $\rG_{\bm{\alpha},\gamma}$ is defined by~\eqref{eq:symP2Lnormalization}.
\end{theorem}

\subsection{Contour integrals}
\label{subsec:ContourIntegrals}
We now write the integrals of Whittaker functions obtained in~\eqref{eq:P2LLaplaceTransf} and~\eqref{eq:P2HLLaplaceTransf} as contour integrals.

In formula~\eqref{eq:P2LLaplaceTransf} for the Laplace transform of the point-to-line partition function, the integral 
\[
\int_{\R_{+}^n} 
\bigg(\prod_{i=1}^n x_i\bigg)^{\gamma}
e^{-r x_1}
\Psi_{\bm{\alpha}}^{\mathfrak{so}_{2n+1}}(\bm{x})
\Psi_{\bm{\beta}}^{\mathfrak{so}_{2n+1}}(\bm{x})
\prod_{i=1}^n \frac{\diff x_i}{x_i}
\]
is analogous, at least for $\gamma=0$, to the Bump-Stade identity~\eqref{eq:Stade}, where $\mathfrak{gl}_n$-Whittaker functions are replaced with the corresponding orthogonal ones. However, a closed formula for our integral does not appear in the literature, and there are actually some reasons why one may not expect it to be computable in terms of products and ratios of gamma functions, see \cite{Bum89},
section 2.6.
 However, we can still turn formula~\eqref{eq:P2LLaplaceTransf} into a contour integral of Gamma functions through the
 Plancherel theorem of $\mathfrak{gl}_n$-Whittaker function and a combination of both the Bump-Stade and Ishii-Stade identities. 
 The key tool is the following lemma.

\begin{lemma}
\label{lemma:WhittakerTransforms}
The $\mathfrak{gl}_n$-Whittaker isometry between the spaces $L^2(\R^n_+, \prod_{i=1}^n \diff x_i / x_i)$ and $L^2_{\sym}(\iota\R^n, s_n(\bm{\lambda}) \diff\bm{\lambda})$ defined in Theorem~\ref{thm:Plancherel} maps
\begin{align*}
 {\rm (a)} \qquad f(\bm{x}) := e^{-r x_1} \Psi^{\mathfrak{gl}_n}_{\bm{\alpha}}
\quad&\longmapsto\quad
\hat{f}(\bm{\lambda}) := r^{-\sum_{i=1}^n (\lambda_i + \alpha_i)}
\prod_{1\leq i,j\leq n} \!\! \Gamma(\lambda_i + \alpha_j) \\
\intertext{for all $r>0$ and $\bm{\alpha}\in\C^n$ such that $\Re(\alpha_j) >0$ for all $j$, and}
{\rm (b)} \quad g(\bm{x}) := \bigg(\prod_{i=1}^n x_i \bigg)^{-s} \Psi^{\mathfrak{so}_{2n+1}}_{\bm{\beta}}(\bm{x})
\quad&\longmapsto\quad
\hat{g}(\bm{\lambda})
:= \frac{\prod_{1\leq i,j\leq n} \Gamma(s-\lambda_i+\beta_j) \Gamma(s-\lambda_i-\beta_j)}{\prod_{1\leq i<j\leq n} \Gamma(2s-\lambda_i - \lambda_j)}
\end{align*}
for all $s\in\C$ and $\bm{\beta}\in\C^n$ such that $\Re(s) > \abs{\Re(\beta_j)}$ for all $j$.
\end{lemma}

\begin{proof}
(a) Assuming that $f$ is square-integrable, the Bump-Stade identity~\eqref{eq:Stade} implies that $\hat{f}$ is indeed the $\mathfrak{gl}_n$-Whittaker transform of $f$. To prove that $f$ belongs to $L^2(\R^n_+, \prod_{i=1}^n \diff x_i / x_i)$, we will show instead the equivalent statement that $\hat{f}$ is in $L^2_{\sym}(\iota\R^n, s_n(\bm{\lambda}) \diff\bm{\lambda})$. It is clear that $\hat{f}$ is a symmetric function, and has no poles, thanks to the assumption that each $\alpha_j$ has positive real part.
Recalling the Stirling approximation of the Gamma function
\begin{equation}
\label{eq:StirlingApprox}
\abs{\Gamma(x+iy)}
\sim
\sqrt{2\pi} \abs{y}^{x-\frac{1}{2}} e^{-\frac{\pi}{2} \abs{y}}
\qquad
\text{as } \abs{y} \to \infty \, ,
\end{equation}
we can compute the asymptotics of $\abs{\hat{f}(\bm{\lambda})}^2 s_n(\bm{\lambda})$ as $\abs{\lambda_i} \to\infty$ for all $i$ and $\abs{\lambda_i - \lambda_j} \to\infty$ for all $i< j$ (which is when $s_n(\bm{\lambda})$ has the worst diverging behavior):
\[
\begin{split}
\abs{\hat{f}(\bm{\lambda})}^2 s_n(\bm{\lambda})
&= \frac{r^{-2\sum_{i} \Re(\alpha_i)}\prod_{i,j} \abs{\Gamma(\lambda_i + \alpha_j)}^2}{(2\pi)^n n! \prod_{i\neq j} \abs{\Gamma(\lambda_i-\lambda_j)}}
\sim \frac{\prod_{i,j}
e^{-\pi\abs{\lambda_i}}}
{\prod_{i< j} e^{-\pi\abs{\lambda_i-\lambda_j}}} \\
&= \exp\bigg(-\pi n \sum_{i} \abs{\lambda_i} + \pi \sum_{i<j} \abs{\lambda_i - \lambda_j} \bigg)
\leq \exp\bigg(- \pi \sum_i \abs{\lambda_i} \bigg) \, .
\end{split}
\]
Here, the symbol $\sim$ denotes asymptotic behavior up to multiplicative constants and powers, and for the last step we have used the following rough estimate:
\begin{equation}
\label{eq:roughEstimate}
\sum_{i<j} \abs{\lambda_i \pm \lambda_j}
\leq \sum_{i<j} \left(\abs{\lambda_i} + \abs{\lambda_j}\right)
= (n-1) \sum_i \abs{\lambda_i} \, .
\end{equation}
This proves that $\abs{\hat{f}(\bm{\lambda})}^2 s_n(\bm{\lambda})$ is integrable on $\iota \R^n$.
 
(b) Consider now $g$ and $\hat{g}$. The fact that the latter is indeed the $\mathfrak{gl}_n$-Whittaker transform of the former follows from property~\eqref{eq:glWhittakerFnTranslation} and Ishii-Stade identity~\eqref{eq:IshiiStade}. Let us prove that $\hat{g}$ belongs to  $L^2_{\sym}(\iota\R^n, s_n(\bm{\lambda}) \diff\bm{\lambda})$. Again, $\hat{g}$ is a symmetric function, and has no poles, thanks to the assumption that $\Re(s)>\abs{\Re(\beta_j)}$ for all $j$. Using~\eqref{eq:StirlingApprox} and~\eqref{eq:roughEstimate}, we compute the asymptotics (up to constants) of $\abs{\hat{g}(\bm{\lambda})}^2 s_n(\bm{\lambda})$ as $\abs{\lambda_i} \to\infty$ for all $i$ and $\abs{\lambda_i \pm \lambda_j} \to\infty$ for all $i< j$:
\[
\begin{split}
\abs{\hat{g}(\bm{\lambda})}^2 s_n(\bm{\lambda})
&= \frac{\prod_{i,j}
\abs{\Gamma(s-\lambda_i+\beta_j)
\Gamma(s-\lambda_i-\beta_j)}^2}
{(2\pi)^n n! \prod_{i<j}
\abs{\Gamma(2s-\lambda_i - \lambda_j)}^2
\prod_{i\neq j}
\abs{\Gamma(\lambda_i - \lambda_j)}}
\sim \frac{\prod_{i,j}
e^{-2\pi\abs{\lambda_i}}}
{\prod_{i<j}
e^{-\pi\abs{\lambda_i + \lambda_j}}
e^{-\pi\abs{\lambda_i-\lambda_j} }} \\
&= \exp\bigg(-2n\pi \sum_{i} \abs{\lambda_i} + \pi \sum_{i<j} \abs{\lambda_i + \lambda_j} + \pi \sum_{i<j} \abs{\lambda_i - \lambda_j} \bigg)
\leq \exp\bigg(-2 \pi \sum_i \abs{\lambda_i} \bigg) \, ,
\end{split}
\]
which proves the integrability of $\abs{\hat{g}(\bm{\lambda})}^2 s_n(\bm{\lambda})$ on $\iota \R^n$.
\end{proof}

\begin{theorem}
\label{thm:P2LcontourInt}
The Laplace transform of the point-to-line partition function $Z_{2n}$ for the
$(\bm{\alpha},\bm{\beta},\gamma)$-log-gamma polymer is given by
\begin{equation}
\label{eq:P2LcontourInt}
\begin{split}
\E\Big[e^{- r \fZ_{\,2n}}\Big] 
= &\frac{r^{\sum_{k=1}^{n} (\alpha_k+\beta_k)}}
{\fG_{\bm{\alpha},\bm{\beta},\gamma}}
\int_{(\epsilon + \iota \R)^n} 
s_n(\bm{\rho}) \diff \bm{\rho}
\int_{(\delta + \iota \R)^n}
s_n(\bm{\lambda}) \diff \bm{\lambda} \,\,
r^{-\sum_{i=1}^n (\lambda_i+\rho_i+\gamma)}  \\
&\times \frac{\prod_{1\leq i,j\leq n}
\Gamma(\lambda_i + \rho_j + \gamma)
\Gamma(\lambda_i + \alpha_j)
\Gamma(\lambda_i - \alpha_j)
\Gamma(\rho_i + \beta_j)
\Gamma(\rho_i - \beta_j)}
{\prod_{1\leq i<j\leq n}
\Gamma(\lambda_i+\lambda_j)
\Gamma(\rho_i + \rho_j) }
\end{split}
\end{equation}
for all $r>0$, where $\fG_{\bm{\alpha},\bm{\beta},\gamma}$ is the constant defined in~\eqref{eq:P2Lnormalization}, $s_n(\bm{\lambda}) \diff \bm{\lambda}$ is the Sklyanin measure as in~\eqref{eq:sklyaninMeasure} and $\delta, \epsilon$ satisfy $\delta > \alpha_j$ and $\epsilon > \beta_j$ for all $j$.
The contour integral~\eqref{eq:P2LcontourInt} is absolutely convergent.
\end{theorem}

\begin{proof}
We start from formula~\eqref{eq:P2LLaplaceTransf} and apply the $\mathfrak{gl}_n$-Whittaker-Plancherel
Theorem~\ref{thm:Plancherel} in a two-step procedure.

The integral appearing in formula~\eqref{eq:P2LLaplaceTransf} can be written as
\begin{equation}
\label{eq:plancherel1step}
\int_{\R_{+}^n} 
\bigg(\prod_{i=1}^n x_i\bigg)^{\gamma}
e^{-r x_1}
\Psi_{\bm{\alpha}}^{\mathfrak{so}_{2n+1}}(\bm{x})
\Psi_{\bm{\beta}}^{\mathfrak{so}_{2n+1}}(\bm{x})
\prod_{i=1}^n \frac{\diff x_i}{x_i}
= \int_{\R_{+}^n}
f(\bm{x}) g(\bm{x})
\prod_{i=1}^n \frac{\diff x_i}{x_i} \, ,
\end{equation}
where
\[
f(\bm{x}) :=
\bigg(\prod_{i=1}^n x_i\bigg)^{\gamma+\epsilon}
e^{-r x_1}
\Psi_{\bm{\alpha}}^{\mathfrak{so}_{2n+1}}(\bm{x}) \, ,
\qquad\quad
g(\bm{x}) :=
\bigg(\prod_{i=1}^n x_i\bigg)^{-\epsilon}
\Psi_{\bm{\beta}}^{\mathfrak{so}_{2n+1}}(\bm{x}) \, .
\]
By Lemma~\ref{lemma:WhittakerTransforms}, since $\epsilon > \beta_j > 0$ for all $j$, $g$ belongs to $L^2(\R_{+}^n, \prod_{i=1}^n \diff x_i/x_i)$ and satisfies
\[
\overline{\hat{g}(\bm{\rho})}
= \frac{\prod_{1\leq i,j\leq n} \Gamma(\epsilon+\rho_i+\beta_j) \Gamma(\epsilon+\rho_i-\beta_j)}{\prod_{1\leq i<j\leq n} \Gamma(2\epsilon +\rho_i + \rho_j)}
\]
for all $\bm{\rho}\in\iota\R^n$.
On the other hand, applying Theorem~\ref{thm:P2LLaplaceTransf} in the case where $\bm{\ga}=\bm{\gb}$, $\gamma$ is replaced by $2(\gamma+\epsilon)$ and $r$ is replaced by $2r$, we obtain that 
\begin{equation}
\label{eq:estimationL2}
\int_{\R_{+}^n} \abs{f(\bm{x})}^2 \prod_{i=1}^n \frac{\diff x_i}{x_i}
= \frac{\fG_{\bm{\alpha},\bm{\alpha},2(\gamma+\epsilon)}}
{(2r)^{2\sum_{k=1}^{n} (\alpha_k +\gamma + \epsilon)}}
\E\big[e^{-2r \tilde{Z}^{\rm flat}_{2n} }\big]
< \infty \, ,
\end{equation}
where $\tilde{Z}^{\rm flat}_{2n}$ is the point-to-line partition function of the $(\bm{\alpha},\bm{\alpha},2(\gamma+\epsilon))$-log-gamma polymer. This proves that $f$ also belongs to $L^2(\R_{+}^n, \prod_{i=1}^n \diff x_i/x_i)$, so we can apply the $\mathfrak{gl}_n$-Whittaker-Plancherel theorem in~\eqref{eq:plancherel1step} and obtain:
\begin{equation}
\begin{split}
\label{eq:plancherel2step}
&\int_{\R_{+}^n} 
\bigg(\prod_{i=1}^n x_i\bigg)^{\gamma}
e^{-r x_1}
\Psi_{\bm{\alpha}}^{\mathfrak{so}_{2n+1}}(\bm{x})
\Psi_{\bm{\beta}}^{\mathfrak{so}_{2n+1}}(\bm{x})
\prod_{i=1}^n \frac{\diff x_i}{x_i} \\
= &\int_{(\epsilon +\iota\R)^n}
\hat{f}(\bm{\rho}-\epsilon)
\frac{\prod_{1\leq i,j\leq n}
\Gamma(\rho_i+\beta_j)
\Gamma(\rho_i-\beta_j)}
{\prod_{1\leq i<j\leq n}
\Gamma(\rho_i + \rho_j)}
s_n(\bm{\rho}) \diff \bm{\rho}
\, ,
\end{split}
\end{equation}
after the change of variables $\bm{\rho} \mapsto \bm{\rho}-\epsilon$. To compute $\hat{f}(\bm{\rho}-\epsilon)$, we first notice that by property~\eqref{eq:glWhittakerFnTranslation}
\[
\begin{split}
\hat{f}(\bm{\rho}-\epsilon)
&= \int_{\R_{+}^n}
\bigg(\prod_{i=1}^n x_i\bigg)^{\gamma+\epsilon}
e^{-r x_1}
\Psi_{\bm{\alpha}}^{\mathfrak{so}_{2n+1}}(\bm{x})
\Psi_{\bm{\rho}-\epsilon}^{\mathfrak{gl}_n}(\bm{x})
\prod_{i=1}^n \frac{\diff x_i}{x_i} \\
&= \int_{\R_{+}^n}
\Big[
e^{-r x_1}
\Psi_{\bm{\rho}+\gamma+\delta}^{\mathfrak{gl}_n}(\bm{x}) \Big]
\bigg[
\bigg(\prod_{i=1}^n x_i\bigg)^{-\delta}
\Psi_{\bm{\alpha}}^{\mathfrak{so}_{2n+1}}(\bm{x}) \bigg]
\prod_{i=1}^n \frac{\diff x_i}{x_i} \, .
\end{split}
\]
By Lemma~\ref{lemma:WhittakerTransforms}, since $\gamma\geq 0$, $\delta > \alpha_j>0$ and $\Re(\rho_j)=\epsilon$ for all $j$, the two functions in the square brackets belong to $L^2(\R_{+}^n, \prod_{i=1}^n \diff x_i/x_i)$, with $\mathfrak{gl}_n$-Whittaker transforms given by the same lemma. Applying the Plancherel theorem again, we then obtain
\[
\hat{f}(\bm{\rho}-\epsilon)
= \int_{(\delta+\iota\R)^n}
r^{-\sum_{i=1}^n(\lambda_i + \rho_i + \gamma)} \frac{
\prod_{1\leq i,j\leq n}
\Gamma(\lambda_i + \rho_j + \gamma) \Gamma(\lambda_i + \alpha_j)
\Gamma(\lambda_i - \alpha_j)}
{\prod_{1\leq i<j\leq n}
\Gamma(\lambda_i + \lambda_j)}
s_n(\bm{\lambda}) \diff \bm{\lambda} \, ,
\]
after shifting the contour integral by $\delta$. Plugging the latter formula into~\eqref{eq:plancherel2step} concludes the proof of~\eqref{eq:P2LcontourInt}.

Finally, we are going to show that the integral in \eqref{eq:P2LcontourInt} 
is absolutely convergent, so that the order of integration with respect to $\bm{\lambda}$ and $\bm{\rho}$ does not matter. Note first that the integrand has no poles thanks to the choice of $\delta$ and $\epsilon$.
So we just need to check integrability as $\abs{\Im(\lambda_i)},\abs{\Im(\rho_i)}\to\infty$ for all $i$ and $\abs{\Im(\lambda_i \pm \lambda_j)}, \abs{\Im(\rho_i \pm \rho_j)} \to\infty$ for all $i< j$. Using the asymptotics of the Gamma function~\eqref{eq:StirlingApprox}, it suffices to check the integrability of
\[
\frac{\prod_{i,j}
e^{-\frac{\pi}{2} \abs{\Im(\lambda_i +\rho_j)}}
e^{-\pi \abs{\Im(\lambda_i)}}
e^{-\pi \abs{\Im(\rho_i)}}}
{\prod_{i<j}
e^{-\frac{\pi}{2} \abs{\Im(\lambda_i+\lambda_j)}}
e^{-\frac{\pi}{2} \abs{\Im(\rho_i+\rho_j)}} }
\prod_{i \neq j}
e^{\frac{\pi}{2} \abs{\Im(\lambda_i - \lambda_j)}}
e^{\frac{\pi}{2} \abs{\Im(\rho_i-\rho_j)}} \, .
\]
Since we are looking at the regime of large imaginary parts for $\bm{\rho}$ an $\bm{\lambda}$, we may assume that they
are purely imaginary, hence we need to estimate
\begin{equation}
\label{eq:absContourInt}
-\sum_{i,j}
\big(\abs{\lambda_i + \rho_j} + 2\abs{\lambda_i} + 2\abs{\rho_i} \big)
+ \sum_{i<j}
\big(
\abs{\lambda_i + \lambda_j}
+ \abs{\rho_i + \rho_j} \big)
+\sum_{i \neq j}
\big(\abs{\lambda_i - \lambda_j}
+ \abs{\rho_i - \rho_j} \big) \, .
\end{equation} 
At this stage, since the imaginary $\iota$ will be absorbed by the absolute value, we may assume that $\bm{\lambda}$ and
 $\bm{\rho}$ are real variables and, furthermore, since the above expression is symmetric, we may assume that
\[
\lambda_1\geq \lambda_2 \geq \dots \geq \lambda_n 
\quad\qquad\text{and}\quad\qquad
\rho_1 \leq \rho_2 \leq \dots \leq \rho_n \, .
\]
This will then allow the bound
\[
\begin{split}
&\sum_{i \neq j} \big(\abs{\lambda_i - \lambda_j}
+ \abs{\rho_i - \rho_j} \big)
= 2 \sum_{i<j} (\lambda_i - \lambda_j + \rho_j - \rho_i )
= 2 \sum_{i<j} \abs{(\lambda_i+\rho_j)
- (\rho_i+\lambda_j)} \\
= &\sum_{i,j} \abs{(\lambda_i+\rho_j)
- (\rho_i+\lambda_j)}
\leq \sum_{i,j} \abs{\lambda_i + \rho_j}
+ \sum_{i,j} \abs{\rho_i + \lambda_j}
= 2 \sum_{i,j} \abs{\lambda_i + \rho_j} \, .
\end{split}
\]
Using the latter estimate and the one given in~\eqref{eq:roughEstimate}, we obtain
\[
\begin{split}
\eqref{eq:absContourInt}
&\leq
-\sum_{i,j} \abs{\lambda_i + \rho_j}
-2n \sum_i \big(\abs{\lambda_i}
+\abs{\rho_i} \big)
+(n-1)\sum_i \big(\abs{\lambda_i} + \abs{\rho_i} \big)
+2\sum_{i,j} \abs{\lambda_i + \rho_j} \\
&= (-2n+n-1)\sum_i \big(\abs{\lambda_i} + \abs{\rho_i} \big)
+ \sum_{i,j} \abs{\lambda_i+\rho_j} \\
&\leq (-n-1)\sum_i \big(\abs{\lambda_i} + \abs{\rho_i} \big)
+ n\sum_{i} \abs{\lambda_i}
+ n\sum_{j} \abs{\rho_j} \\
&= -
\sum_i \big(\abs{\lambda_i} + \abs{\rho_i}\big) \, ,
\end{split}
\]
hence the exponential of~\eqref{eq:absContourInt} is integrable for $(\bm{\lambda},\bm{\rho})\in\R^{2n}$.
\end{proof}

A contour integral formula for the point-to-half-line partition function is easier to obtain
 because it requires to apply the $\mathfrak{gl}_n$-Whittaker-Plancherel theorem only once.

\begin{theorem}
\label{thm:P2HLcontourInt}
The Laplace transform of the point-to-half-line partition function $\hZ_{2n}$ for the
$(\bm{\alpha},\bm{\beta})$-log-gamma polymer is given by
\begin{equation}
\label{eq:P2HLcontourInt}
\begin{split}
\E\Big[e^{- r \hZ_{\,2n}}\Big] 
= &\frac{r^{\sum_{k=1}^{n} (\alpha_k+\beta_k)}}
{\hG_{\bm{\alpha},\bm{\beta}}}
\int_{(\delta + \iota \R)^n}
s_n(\bm{\lambda}) \diff \bm{\lambda} \,\,
r^{-\sum_{i=1}^n (\lambda_i+\beta_i)}  \\
&\times \frac{\prod_{1\leq i,j\leq n}
\Gamma(\lambda_i + \alpha_j)
\Gamma(\lambda_i - \alpha_j)
\Gamma(\lambda_i + \beta_j)}
{\prod_{1\leq i<j\leq n}
\Gamma(\lambda_i+\lambda_j) },
\end{split}
\end{equation}
for all $r>0$, where $\hG_{\bm{\alpha},\bm{\beta}}$ is the constant defined in~\eqref{eq:P2HLnormalization}, $s_n(\bm{\lambda}) \diff \bm{\lambda}$ is the Sklyanin measure as in~\eqref{eq:sklyaninMeasure}, and $\delta$ is chosen such that $\delta > \alpha_j$ for all $j$.
\end{theorem}

\begin{proof}
It suffices to write the integral on the right-hand side of~\eqref{eq:P2HLLaplaceTransf} as
\[
\int_{\R_{+}^n} e^{-r x_1}
\Psi_{\bm{\alpha}}^{\mathfrak{so}_{2n+1}}(\bm{x})
\Psi_{\bm{\beta}}^{\mathfrak{gl}_n}(\bm{x})
\prod_{i=1}^n \frac{\diff x_i}{x_i}
= \int_{\R_{+}^n}
\Big[ e^{-r x_1}
\Psi_{\bm{\beta}+\delta}^{\mathfrak{gl}_n}(\bm{x}) \Big]
\bigg[
\bigg(\prod_{i=1}^n x_i\bigg)^{-\delta}
\Psi_{\bm{\alpha}}^{\mathfrak{so}_{2n+1}}(\bm{x}) \bigg]
\prod_{i=1}^n \frac{\diff x_i}{x_i},
\]
where we multiplied and divided by $(\prod_{i=1}^n x_i)^{\delta}$ and used property \eqref{eq:glWhittakerFnTranslation}. 
We now apply Theorem~\ref{thm:Plancherel} to the two functions in the square brackets, 
whose $\mathfrak{gl}_n$-Whittaker-transforms have been computed in Lemma~\ref{lemma:WhittakerTransforms}.
\end{proof}

\section{Zero temperature limit}
\label{sec:zero}

In this section, we derive the zero temperature limit of the formulae provided by Theorems~\ref{thm:P2LLaplaceTransf}, \ref{thm:P2HLLaplaceTransf} and \ref{thm:symP2LLaplaceTransf} for the Laplace transforms of the flat, half-flat and restricted half-flat log-gamma polymer partition functions respectively, deriving what seem to be new formulae for the law of the last passage percolation with exponential waiting times in these three path geometries.

Let us first define, for a given triangular array $\bm{W}=\{W_{i,j}\colon i+j\leq N+1\}$, the zero temperature analogue of~\eqref{eq:flatPartitionFn}, i.e.\ the point-to-line last passage percolation time:
\begin{align*}
\fTau_N &:= \max_{\pi \in \fPi_N} \sum_{(i,j) \in \pi} W_{i,j}
\, .
\end{align*}
Similarly, the zero temperature analogues of~\eqref{eq:hFlatPartitionFn} and~\eqref{eq:rFlatPartitionFn} are the point-to-half-line and restricted point-to-half-line last passage percolation times:
\[
\hTau_N := \max_{\pi\in \hPi_N} \sum_{(i,j)\in \pi} W_{i,j} \, ,
\qquad\quad
\rTau_N := \max_{\pi\in \rPi_N} \sum_{(i,j)\in \pi} W_{i,j} \, .
\]

The following technical proposition, whose proof is easy and omitted, explains how the zero temperature limit works. It is not specific to the log-gamma distribution.
\begin{proposition}
\label{prop:zeroTempLimit}
Let $Z_N^{(\epsilon)}$ be the polymer partition function corresponding to any of the path geometries considered, with disorder given by independent positive weights $\bm{W}^{(\epsilon)}=\{W_{i,j}^{(\epsilon)}\}$, whose distributions depend on a parameter $\epsilon>0$. Let $\tau_N$ be the last passage percolation in the same geometry, with independent positive continuous waiting times $\bm{W}=\{W_{i,j}\}$.
Assuming that each $\epsilon \log W_{i,j}^{(\epsilon)}$ converges in distribution to $W_{i,j}$ as $\epsilon \downarrow 0$, we have:
\begin{enumerate}
\item
$\epsilon \log Z_N^{(\epsilon)}
\xrightarrow[\epsilon\downarrow 0]{(d)}
\tau_N$;
\item
\label{prop:zeroTempLimit_LaplaceTransf}
$\E\left[\exp\left(-e^{{-u/\epsilon}} Z_N^{(\epsilon)} \right)\right]
\xrightarrow{\epsilon \downarrow 0} \P(\tau_N \leq u)$, for all $u\in\R$.
\end{enumerate}
\end{proposition}

On the other hand, it is easy to check that, if $W^{(\epsilon)}$ is inverse-gamma distributed with parameter $\epsilon \gamma$ and $W$ is exponentially distributed with rate $\gamma$ for some $\gamma>0$, then
\[
\epsilon \log W^{(\epsilon)} \xrightarrow[\epsilon \downarrow 0]{(d)} W \, .
\]
Joining this observation to Proposition~\ref{prop:zeroTempLimit}-\ref{prop:zeroTempLimit_LaplaceTransf}, it is now clear how to recover the distribution of the last passage percolation with exponentially distributed waiting times from the rescaled Laplace transforms of the log-gamma polymer partition functions, in any of the three path geometries.

\subsection{The point-to-line last passage percolation}

In the zero temperature limit of the $(\epsilon\bm{\alpha},\epsilon\bm{\beta}, 0)$-log-gamma polymer partition function (refer to Definition~\ref{def:log} and set $\gamma:=0$), we will be able to obtain the distribution of the point-to-line last passage percolation $\fTau_{2n}$ with independent waiting times distributed as follows:
\begin{equation}
\label{eq:expDistribution}
W_{i,j} \sim
\begin{cases}
{\rm Exp}(\alpha_i + \beta_j) &1\leq i,j\leq n \, , \\
{\rm Exp}(\alpha_i + \alpha_{2n-j+1}) &1\leq i\leq n \, , \,\, n < j\leq 2n-i+1 \, , \\
{\rm Exp}(\beta_{2n-i+1} + \beta_j) &1\leq j\leq n \, , \,\, n < i \leq 2n-j+1 \, .
\end{cases}
\end{equation}

Baik and Rains~\cite{BR01} derived the law of $\fTau_N$ when $W_{i,j}$ are independent, geometrically distributed variables
with parameters  $1-y_i y_{N+1-j}$. In particular, they established that 
\begin{equation}
\label{eq:baikRains}
\P(\fTau_N \leq u)
= \prod_{1\leq i\leq j\leq N}
\!\!\!
(1-y_i y_j)
\!\!\!\!\!\!\!\!
\sum_{\substack{\bm{\mu}\in\Z^N, \\ 0\leq \mu_N \leq \dots \leq \mu_1 \leq u}}
\!\!\!\!\!\!\!\!\!\!
\schur_{2\bm{\mu}}(y_1,\dots,y_N) \, ,
\end{equation}
where $s_{2\bm{\mu}}$ is the Schur polynomial with shape $2\bm{\mu}$.  
On the other hand, the formula~\eqref{eq:flatLPP1} we obtain in the zero temperature limit of~\eqref{eq:P2LLaplaceTransf} gives the law of $\tau^{\rm flat}_{\,2n}$ with exponential weights in terms of a symplectic Cauchy-like identity, i.e.\ as an integral of (the continuum analogue of) \emph{two symplectic} Schur functions. Such a formula thus looks essentially different from the integral of (the continuum analogue of) \emph{one classical} Schur function that one would obtain by taking the exponential limit of~\eqref{eq:baikRains}.
 
Symplectic Schur polynomials can be defined in terms of symplectic tableaux~\cite{FK97} 
or equivalently in terms of symplectic Gelfand-Tsetlin patterns as 
\[
\sp_{\bm{\mu}}(\bm{y})
:=
\!\!\!
\sum_{\substack{\bm{z}\in \Z^{n^2} \\ \cap \GT_{2n}^{\llrighttriangle[0.15]}\left(\bm{\mu}\right)}}
\!\!\!
\prod_{k=1}^{n} y_k^{2\abs{\bm{z}_{2k-1}}
- \abs{\bm{z}_{2k-2}} - \abs{\bm{z}_{2k}}} 
\]
for $\bm{\mu}\in\Z^n$ with $0\leq \mu_n\leq \dots\leq \mu_1$. Here, $\GT_{2n}^{\llrighttriangle[0.15]}(\bm{x})$ is the set of all \emph{symplectic Gelfand-Tsetlin patterns} $\bm{z}$ with real entries, depth $2n$ and last row $(z_{1,1},\dots,z_{n,1})$ equal to $\bm{x}$; by this, we mean that the pattern $\bm{z}=\{z_{i,j}\colon 1\leq i\leq n, \, 1\leq j\leq \ceil{i/2}\}$ satisfies the interlacing conditions:
\[
z_{i+1,j+1}\leq z_{i,j}\leq z_{i+1,j}
\qquad\quad \text{for} \quad 1\leq i\leq 2n-1 \, , \quad 1\leq j\leq \ceil{i/2} \, ,
\]
with the convention that $z_{i,j}:=0$ when $j>\ceil{i/2}$.
Symplectic Schur polynomials are characters of the irreducible representations of $Sp_{2n}$ (see e.g.~\cite{Sun90}) hence, by the Weyl character formula~\cite[24.18]{FH91}, they can also be written as a ratio of determinants:
\begin{align}\label{eq:SpWeylChar}
\sp_{\bm{\mu}}(\bm{y})
= \frac{\det\left(y_j^{\mu_i+n-i+1} - y_j^{-(\mu_i+n-i+1)}\right)_{1\leq i,j\leq n}}
{\det\left(y_j^{n-i+1} - y_j^{-(n-i+1)}\right)_{1\leq i,j\leq n}} \, .
\end{align}
The denominator can be expressed more explicitly~\cite[24.17]{FH91}:
\[
\det\left(y_j^{n-i+1} - y_j^{-(n-i+1)}\right)_{1\leq i,j\leq n}
= \prod_{1\leq i<j\leq n}
(y_i - y_j)(y_i y_j - 1)
\prod_{i=1}^n (y_i^2-1)y_i^{-n} \, .
\]

We first prove a formula for the rescaled limit of orthogonal Whittaker functions: the proof relies on the fact that these can be approximated by symplectic Schur polynomials, which can in turn be expressed as a ratio of determinants by the Weyl character formula.
\begin{proposition}
\label{prop:soWhittakerFnScaling}
Let $\bm{\alpha},\bm{x}\in\R^n$ and let $\GT_{2n}^{\llrighttriangle[0.15]}(\bm{x})$ be the set of symplectic Gelfand-Tsetlin patterns with shape $\bm{x}$.
We then have that
\begin{align}
\label{eq:SOWhittakerRescaling}
\lim_{\epsilon\downarrow 0}
\epsilon^{n^2} 
\Psi_{\epsilon\bm{\alpha}}^{\mathfrak{so}_{2n+1}}\big(e^{x_1/\epsilon},\dots,e^{x_n/\epsilon}\big)
&= \sp^{\rm cont}_{\bm{\ga}}(\bm{x})
\1_{\{0\leq x_n\leq \dots \leq x_1\}} \, ,
\end{align}
where
\begin{equation}
\label{eq:contSp}
\begin{split}
\sp^{\rm cont}_{\bm{\ga}}(\bm{x})
&:= \int_{\GT_{2n}^{\llrighttriangle[0.15]}(\bm{x})}
\prod_{k=1}^n
e^{ \alpha_k \left(2\abs{\bm{z}_{2k-1}}
- \abs{\bm{z}_{2k-2}} - \abs{\bm{z}_{2k}}\right)}
\!\!\!
\prod_{\substack{1\leq i<2n \\ 1\leq j\leq \ceil{i/2}}}
\!\!\!\!\!
\diff z_{i,j}
\end{split}
\end{equation}
is a continuum version of the symplectic Schur function, and $\abs{\bm{z}_i} :=\sum_{j=1}^{\ceil{i/2}} z_{i,j}$.
Moreover, $\sp^{\rm cont}_{\bm{\ga}}$ has a determinantal form:
\begin{equation}
\label{eq:contSp2}
\sp^{\rm cont}_{\bm{\ga}}(\bm{x})
= \frac{\det\big(e^{\alpha_j x_i} - e^{-\alpha_j x_i}\big)_{1\leq i,j\leq n}}{\prod_{1\leq i<j\leq n}(\alpha_i-\alpha_j) \prod_{1\leq i\leq j\leq n}(\alpha_i + \alpha_j)} \, .
\end{equation}
\end{proposition}

\begin{remark}
{ \rm
When $\alpha_i=\alpha_j$ for some $i,j$, \eqref{eq:contSp2} is still valid in the limit as $\alpha_i-\alpha_j \to 0$. In particular, when all $\alpha_i$'s are equal to a given $\alpha$, we have that
\[
\sp^{\rm cont}_{\bm{\ga}}(\bm{x})
= (-1)^{n(n-1)/2}
\frac{\det\big(x_i^{j-1} (e^{\alpha x_i} +(-1)^j e^{-\alpha x_i}) \big)_{1\leq i,j\leq n}}
{(2\alpha)^{n(n+1)/2} \prod_{j=1}^n (j-1)!} \, .
\]
This formula is an immediate consequence of the following fact: if the functions $f_1,\dots,f_n$ are differentiable $n-1$ times at $\alpha$, then
\[
\frac{\det(f_i(\alpha_j))_{\leq i,j\leq n}}{\prod_{1\leq i<j\leq n} (\alpha_j - \alpha_i)}
\longrightarrow
\frac{W(f_1,\dots,f_n)(\alpha)}
{\prod_{j=1}^n (j-1)!} \quad \text{as } \alpha_1,\dots,\alpha_n \to \alpha \, ,
\]
where $W(f_1,\dots,f_n)(\alpha) = \det\big(f_i^{(j-1)}(\alpha)\big)_{1\leq i,j\leq n}$ is the Wronskian at $\alpha$.
}
\end{remark}

\begin{proof}
In Definition~\ref{def:soWhittakerFn}, we change variables by setting $z_{i,j}\mapsto e^{z_{i,j}/\epsilon}$ for all $1\leq i< 2n$ and $1\leq j\leq \ceil{i/2}$ and obtain
\[
\begin{split}
&\epsilon^{n^2} 
\Psi_{\epsilon\bm{\alpha}}^{\mathfrak{so}_{2n+1}}\big(e^{x_1/\epsilon},\dots,e^{x_n/\epsilon}\big) 
=\int_{\mathcal{T}_{2n}^{\llrighttriangle[0.15]}(\bm{x})}
\prod_{k=1}^n
\exp\left( \frac{2\abs{\bm{z}_{2k-1}}}{\epsilon}
- \frac{\abs{\bm{z}_{2k-2}}}{\epsilon} - \frac{\abs{\bm{z}_{2k}}}{\epsilon} \right)^{\epsilon \alpha_k} \\
&\qquad\qquad\times
\prod_{i=1}^{2n-1} \prod_{j=1}^{\ceil{i/2}}
\exp\big(-e^{(z_{i+1,j+1}-z_{i,j})/\epsilon}\big)
\exp\big(-e^{(z_{i,j}-z_{i+1,j})/\epsilon}\big)
\!\!\!
\prod_{\substack{1\leq i<2n \\ 1\leq j\leq \ceil{i/2}}}
\!\!\!\!\!
\diff z_{i,j} \, ,
\end{split}
\]
with the convention that $z_{i,j}:=0$ when $j>\ceil{i/2}$.
Since $\exp\big(-e^{(a-b)/\epsilon}\big) \xrightarrow{\epsilon\downarrow 0} \1_{\{a\leq b\}}$ for all $a\neq b$, we then have
\[
\prod_{i=1}^{2n-1} \prod_{j=1}^{\ceil{i/2}}
\exp\big(-e^{(z_{i+1,j+1}-z_{i,j})/\epsilon}\big)
\exp\big(-e^{(z_{i,j}-z_{i+1,j})/\epsilon}\big)
\xrightarrow{\epsilon\downarrow 0}
\1_{\GT_{2n}^{\llrighttriangle[0.15]}(\bm{x})}(\bm{z})
\1_{\{0\leq x_n\leq \dots\leq x_1\}}
\]
for a.e.\ $\bm{z}\in \mathcal{T}_{2n}^{\llrighttriangle[0.15]}(\bm{x})$. By dominated convergence, we thus obtain
\begin{align*}
&\lim_{\epsilon\downarrow 0} \epsilon^{n^2} 
\Psi_{\epsilon\bm{\alpha}}^{\mathfrak{so}_{2n+1}}\big(e^{x_1/\epsilon},\dots,e^{x_n/\epsilon}\big) \\
=~ &\1_{\{0\leq x_n\leq \dots\leq x_1\}}
\int_{\GT_{2n}^{\llrighttriangle[0.15]}(\bm{x})}
\prod_{k=1}^n
e^{ \alpha_k \left(2\abs{\bm{z}_{2k-1}}
- \abs{\bm{z}_{2k-2}} - \abs{\bm{z}_{2k}}\right)}
\!\!\!
\prod_{\substack{1\leq i<2n \\ 1\leq j\leq \ceil{i/2}}}
\!\!\!\!\!
\diff z_{i,j} \, .
\end{align*}
The latter integral is equal to the function $\sp^{\rm cont}_{\bm{\ga}}(\bm{x})$ defined in~\eqref{eq:contSp}. 
By Riemann sum approximation, we can rewrite it as
\[
\begin{split}
\sp^{\rm cont}_{\bm{\ga}}(\bm{x})
&= \lim_{\delta\downarrow 0}
\!\!\!\!
\sum_{\substack{\bm{z}\in \Z^{n^2} \\ \cap \GT_{2n}^{\llrighttriangle[0.15]}\left(\floor{\bm{x}/\delta}\right)}}
\!\!\!\!\!\!\!\!
\delta^{n^2}
\prod_{k=1}^{n-1} e^{ \delta\alpha_k \left(2\abs{\bm{z}_{2k-1}}
- \abs{\bm{z}_{2k-2}} - \abs{\bm{z}_{2k}}\right)}
e^{ \delta\alpha_n \left(2\abs{\bm{z}_{2n-1}}
- \abs{\bm{z}_{2n-2}}\right)
- \alpha_n\abs{\bm{x}} } \\
&= \lim_{\delta\downarrow 0}
\delta^{n^2}
e^{ -\alpha_n \abs{\bm{x}} + \delta\alpha_n \abs{\floor{\bm{x}/\delta}} }
\sp_{\floor{\bm{x}/\delta}}\big(e^{\delta\alpha_1},\dots,e^{\delta\alpha_n}\big) \, ,
\end{split}
\]
where $\sp_{\floor{\bm{x}/\delta}}$ is the symplectic Schur polynomial with shape $\floor{\bm{x}/\delta}$. 
From the Weyl character formula~\eqref{eq:SpWeylChar} for $Sp_{2n}$ we have:
\[
\begin{split}
\sp_{\floor{\bm{x}/\delta}}\big(e^{\delta\alpha_1},\dots,e^{\delta\alpha_n}\big)
&= \frac{\det\left(e^{\delta \alpha_j\left(\floor{x_i/\delta} + n-i+1\right)} - e^{-\delta \alpha_j\left(\floor{x_i/\delta} + n-i+1\right)}\right)_{1\leq i,j\leq n}}
{\prod_{1\leq i<j\leq n} \big(e^{\delta \alpha_i} - e^{\delta \alpha_j}\big) \big(e^{\delta (\alpha_i+\alpha_j)}-1\big)
\prod_{i=1}^n \big(e^{2\delta \alpha_i}-1\big)
e^{-\delta\alpha_i n}} \\
&\asymptotic{\delta \downarrow 0} \frac{\det\big(e^{\alpha_j x_i} - e^{-\alpha_j x_i}\big)_{1\leq i,j\leq n}}{\delta^{n^2} \prod_{1\leq i<j\leq n}(\alpha_i-\alpha_j)
(\alpha_i+\alpha_j) \prod_{i=1}^n (2\alpha_i)} \, .
\end{split}
\]
Noting finally that $e^{ -\alpha_n \abs{\bm{x}} + \delta\alpha_n \abs{\floor{\bm{x}/\delta}} } \xrightarrow{\delta\downarrow 0} 1$, \eqref{eq:contSp2} follows.
\end{proof}

\begin{theorem}
\label{thm:flatLPP}
Let $\fTau_{\,2n}$ be the point-to-line last passage percolation with exponentially distributed waiting times as in~\eqref{eq:expDistribution}. Then, for all $u>0$:
\begin{align}
\label{eq:flatLPP1}
\P(\fTau_{2n}\leq u)
= \frac{H_{\bm{\ga},\bm{\gb}}}{e^{u \sum_{k=1}^n( \alpha_k+\beta_k)}}
\int_{\{0\leq x_n\leq \cdots \leq x_1\leq u\}}  
\sp^{\rm cont}_{\bm{\ga}}(\bm{x})  \sp^{\rm cont}_{\bm{\gb}}(\bm{x}) \prod_{i=1}^n \diff x_i \, ,
\end{align}
where the function $\sp^{\rm cont}_{\bm{\ga}}$ is defined in \eqref{eq:contSp}, and
\begin{align*}
H_{\bm{\ga},\bm{\gb}}
:=\prod_{1\leq i\leq j\leq n} \!\! (\alpha_i + \alpha_j)(\beta_i + \beta_j)
\prod_{1\leq i,j\leq n} \! (\alpha_i + \beta_j)
\end{align*}
 is a normalizing factor. 
We can further write \eqref{eq:flatLPP1} in a determinantal form as
\begin{equation}
\label{eq:flatLPP2}
\P(\fTau_{2n}\leq u) 
= \frac{1}{C_{\bm{\alpha},\bm{\beta}}}
\det\left( e^{-u ( \alpha_i+\beta_j)}
\int_0^u
\big(e^{\alpha_i x} - e^{-\alpha_i x}\big)
\big(e^{\beta_j x} - e^{-\beta_j x}\big) \diff x
\right)_{1\leq i,j\leq n} \, ,
\end{equation}
where $C_{\bm{\alpha},\bm{\beta}}$ is a Cauchy's determinant:
\begin{equation}
\label{eq:CauchyDet}
C_{\bm{\alpha},\bm{\beta}}
:= \det\left( \frac{1}{\alpha_i + \beta_j} \right)_{1\leq i,j\leq n}
= \frac{\prod_{1\leq i<j\leq n} (\alpha_i -\alpha_j)(\beta_i - \beta_j)}
{\prod_{1\leq i,j\leq n} (\alpha_i + \beta_j)} \, .
\end{equation}
\end{theorem}
\begin{proof}
By Proposition~\ref{prop:zeroTempLimit}, we just need to compute $\lim_{\epsilon \downarrow 0} \E\left[\exp\left(e^{-u/\epsilon}Z_{2n}^{(\epsilon)}\right)\right]$, where $Z_{2n}^{(\epsilon)}$ is the point-to-line 
$(\epsilon \bm{\alpha},\epsilon \bm{\beta},0)$-log-gamma polymer partition function.
Formula~\eqref{eq:P2LLaplaceTransf} with $\gamma=0$ yields:
\[
\E\left[\exp\left(-e^{-u/\epsilon}Z_{2n}^{(\epsilon)}\right)\right] 
= \frac{e^{-u \sum_{k=1}^{n} (\alpha_k + \beta_k) }}
{\fG_{\epsilon\bm{\alpha},\epsilon\bm{\beta},0}}
\int_{\R_{+}^n}
e^{-e^{-u/\epsilon} x_1}
\Psi_{\epsilon\bm{\alpha}}^{\mathfrak{so}_{2n+1}} (\bm{x})
\Psi_{\epsilon\bm{\beta}}^{\mathfrak{so}_{2n+1}} (\bm{x})
\prod_{i=1}^n \frac{\diff x_i}{x_i} \, .
\]
The integral can be rewritten, changing variables $x_i \mapsto e^{x_i/\epsilon}$ for $1\leq i\leq n$, as follows:
\[
\begin{split}
&\quad\, \int_{\R_{+}^n}
e^{-e^{-u/\epsilon} x_1}
\Psi_{\epsilon\bm{\alpha}}^{\mathfrak{so}_{2n+1}} (\bm{x})
\Psi_{\epsilon\bm{\beta}}^{\mathfrak{so}_{2n+1}} (\bm{x})
\prod_{i=1}^n \frac{\diff x_i}{x_i} \\
&= \epsilon^{-2n^2}
\int_{\R^n} e^{-e^{(x_1-u)/\epsilon}}
\epsilon^{n^2} \Psi_{\epsilon\bm{\alpha}}^{\mathfrak{so}_{2n+1}} \left(e^{x_1/\epsilon},\dots,e^{x_n/\epsilon}\right)
\epsilon^{n^2} \Psi_{\epsilon\bm{\beta}}^{\mathfrak{so}_{2n+1}} \left(e^{x_1/\epsilon},\dots,e^{x_n/\epsilon}\right)
\prod_{i=1}^n \frac{\diff x_i}{\epsilon} \\
&\asymptotic{\epsilon\downarrow 0}
\epsilon^{-2n^2-n}
\int_{\{0\leq x_n\leq \cdots \leq x_1\leq u\}}  
\sp^{\rm cont}_{\bm{\ga}}(\bm{x})  \sp^{\rm cont}_{\bm{\gb}}(\bm{x})
\prod_{i=1}^n \diff x_i \, ,
\end{split}
\]
where the asymptotics follow from Proposition~\ref{prop:soWhittakerFnScaling} and the fact that
$e^{-e^{(x_1-u)/\epsilon}} \xrightarrow{\epsilon \downarrow 0} \1_{\{x_1\leq u\}}$ for all $x_1\neq u$.
On the other hand, using the definition of $\fG_{\epsilon\bm{\alpha},\epsilon\bm{\beta},0}$ given in~\eqref{eq:P2Lnormalization} and the asymptotics of the Gamma function near $0$, we have:
\[
\begin{split}
\fG_{\epsilon\bm{\alpha},\epsilon\bm{\beta},0}
&= \prod_{1\leq i\leq j\leq n} \!\!
\Gamma (\epsilon(\alpha_i + \alpha_j) )
\Gamma (\epsilon(\beta_i + \beta_j) )
\prod_{1\leq i,j\leq n} \!\!
\Gamma\left(\epsilon(\alpha_i + \beta_j)\right) \\
&\asymptotic{\epsilon\downarrow 0}
\epsilon^{-2n^2-n} \!\!
\prod_{1\leq i\leq j\leq n} 
\frac{1}{(\alpha_i + \alpha_j)(\beta_i + \beta_j)}
\prod_{1\leq i,j\leq n} 
\frac{1}{\alpha_i + \beta_j} \, .
\end{split}
\]
Thus, \eqref{eq:flatLPP1} easily follows from the combination of the foregoing formulae.

We now further elaborate~\eqref{eq:flatLPP1} by making use of the determinantal formula~\eqref{eq:contSp2}:
\[
\begin{split}
&\quad\, e^{u \sum_{k=1}^n (\alpha_k + \beta_k)} \P(\fTau_{2n} \leq u) \\
&= \int_{\{0\leq x_n\leq \cdots \leq x_1\leq u\}}
\frac{H_{\bm{\ga},\bm{\gb}} \, \det\big(e^{\alpha_j x_i} - e^{-\alpha_j x_i}\big)_{1\leq i,j\leq n}
\det\big(e^{\beta_j x_i} - e^{-\beta_j x_i}\big)_{1\leq i,j\leq n}}
{\prod_{1\leq i<j\leq n}(\alpha_i-\alpha_j) (\beta_i - \beta_j) \prod_{1\leq i\leq j\leq n}(\alpha_i + \alpha_j) (\beta_i + \beta_j)}
\prod_{i=1}^n \diff x_i \\
&= \frac{1}{C_{\bm{\alpha},\bm{\beta}}} \frac{1}{n!}
\int_{[0,u]^n}
\det\left(e^{\alpha_j x_i} - e^{-\alpha_j x_i}\right)_{1\leq i,j\leq n}
\det\left(e^{\beta_j x_i} - e^{-\beta_j x_i}\right)_{1\leq i,j\leq n}
\prod_{i=1}^n \diff x_i \\
&= \frac{1}{C_{\bm{\alpha},\bm{\beta}}}
\det\left(\int_0^u
\left(e^{\alpha_i x} - e^{-\alpha_i x}\right)
\left(e^{\beta_j x} - e^{-\beta_j x}\right) \diff x
\right)_{1\leq i,j\leq n} \, .
\end{split}
\]
In the latter computation, we have used: the fact that, by the alternating property of the determinant, the integral over $\{0\leq x_n \leq \dots \leq x_1\leq u\}$ is invariant by applying any permutation to the variables $x_i$'s; the definition of $H_{\bm{\ga},\bm{\gb}}$ and $C_{\bm{\alpha},\bm{\beta}}$; and the Cauchy-Binet identity (see e.g.~\cite[ch.~3]{For10}). We now use the multilinearity of the determinant to finally obtain~\eqref{eq:flatLPP2}.
\end{proof}

\subsection{The point-to-half-line last passage percolation}

For the half-flat and restricted half-flat cases, we just state the final formulae and briefly outline their proofs.

Similarly to~\eqref{eq:SOWhittakerRescaling}, properly rescaling a $GL_n(\R)$-Whittaker function yields the continuum version of a classical Schur function, which, thanks to the Weyl character formula for $GL_n$, can also be written in a determinantal form:
\begin{equation}
\label{eq:contS}
\schur^{\rm cont}_{\bm{\beta}}(\bm{x})
= \frac{\det(e^{\beta_j x_i})_{1\leq i,j\leq n}}{\prod_{1\leq i<j\leq n} (\beta_i - \beta_j)} \, .
\end{equation}
Again, when $\beta_i = \beta_j$ for some $i,j$, the latter formula should be viewed in the limit as $\beta_i-\beta_j \to 0$.

Following the same steps as in the flat case, one can express the law of $\hTau_{2n}$, with exponentially distributed waiting times, in terms of an integral of (the continuum version of) a symplectic Schur function and a classical Schur function.
Using the Cauchy-Binet identity in the same fashion as in Theorem~\ref{thm:flatLPP}, one finally obtains:
\begin{theorem}
\label{thm:hFlatLPP}
Let $\hTau_{\,2n}$ be the point-to-half-line last passage percolation with exponentially distributed waiting times as in~\eqref{eq:expDistribution}. Then, for all $u>0$:
\begin{align}
\label{eq:flatLPP1_new}
\P(\hTau_{2n} \leq u)
= \frac{H^{\hFlat}_{\bm{\ga},\bm{\gb}}}{e^{u \sum_{k=1}^n( \alpha_k+\beta_k)}}
\int_{\{0\leq x_n\leq \cdots \leq x_1\leq u\}}  
\sp^{\rm cont}_{\bm{\alpha}}(\bm{x})  \schur^{\rm cont}_{\bm{\beta}}(\bm{x}) \prod_{i=1}^n \diff x_i \, ,
\end{align}
where the functions $\sp^{\rm cont}_{\bm{\alpha}}$ and $\schur^{\rm cont}_{\bm{\beta}}$ are defined in~\eqref{eq:contSp} and~\eqref{eq:contS}, and
\begin{align*}
H^{\hFlat}_{\bm{\ga},\bm{\gb}}
:=\prod_{1\leq i\leq j\leq n} \!\! (\alpha_i + \alpha_j)
\prod_{1\leq i,j\leq n} \! (\alpha_i + \beta_j)
\end{align*}
is a normalizing factor. 
We can further write \eqref{eq:flatLPP1_new} in a determinantal form as
\begin{equation}
\label{eq:hFlatLPP}
\P(\hTau_{2n}\leq u) 
= \frac{1}{C_{\bm{\alpha},\bm{\beta}}}
\det\left( e^{-u ( \alpha_i+\beta_j)}
\int_0^u
\big(e^{\alpha_i x} - e^{-\alpha_i x}\big)
\big(e^{\beta_j x}\big) \diff x
\right)_{1\leq i,j\leq n} \, ,
\end{equation}
where $C_{\bm{\alpha},\bm{\beta}}$ is the Cauchy's determinant defined in~\eqref{eq:CauchyDet}.
\end{theorem}

\subsection{The restricted last passage percolation}

Finally, let us consider the restricted half-flat case. Here, the waiting times are supposed to be independent and distributed as follows:
\begin{equation}
\label{eq:restrExpDistribution}
W_{i,j} \sim
\begin{cases}
{\rm Exp}(\alpha_i) &1\leq i=j \leq n \, , \\
{\rm Exp}(\alpha_i + \alpha_j) &1\leq i<j\leq n \, , \\
{\rm Exp}(\alpha_i + \alpha_{2n-j+1}) &1\leq i\leq n \, , \,\, n < j\leq 2n-i+1 \, .
\end{cases}
\end{equation}

\begin{theorem}
Let $\rTau_{\,2n}$ be the restricted point-to-half-line last passage percolation with exponentially distributed waiting times as in~\eqref{eq:restrExpDistribution}. Then, for all $u>0$:
\begin{align}
\label{eq:rFlatLPP_1}
\P(\rTau_{2n}\leq u)
= \frac{H^{\rFlat}_{\bm{\ga}}}{e^{u \sum_{k=1}^n \alpha_k}}
\int_{\{0\leq x_n\leq \cdots \leq x_1\leq u\}}  
\sp^{\rm cont}_{\bm{\ga}}(\bm{x})
\prod_{i=1}^n \diff x_i \, ,
\end{align}
where the function $\sp^{\rm cont}_{\bm{\ga}}$ is defined in \eqref{eq:contSp}, and
\begin{align*}
H^{\rFlat}_{\bm{\ga}}
:= \prod_{i=1}^n \alpha_i
\prod_{1\leq i < j\leq n} \!\! (\alpha_i + \alpha_j)
\prod_{1\leq i \leq j\leq n} \!\! (\alpha_i + \alpha_j)
\end{align*}
 is a normalizing factor. 
We can further write~\eqref{eq:rFlatLPP_1} in a Pfaffian form as
\begin{equation}
\label{eq:rFlatLPP_2}
\P(\rTau_{2n}\leq u)
= \frac{\Pf(\Phi^{(n)})}{\Pf(S^{(n)})} \, ,
\end{equation}
where matrices $\Phi^{(n)}$ and $S^{(n)}$ are skew-symmetric of order $n$ or $n+1$, according to whether $n$ is even or odd respectively, and are defined by
\[
\Phi^{(n)}_{i,j} :=
\begin{dcases}
\int_0^u \int_0^u \sign(y-x) \phi_i(x) \phi_j(y) \diff x \diff y
& \text{for } 1\leq i,j\leq n \, , \\
\int_0^u \phi_i(x) \diff x
& \text{for $1\leq i\leq n$, $j=n+1$; if $n$ is odd} \, ,
\end{dcases}
\]
having set $\phi_j(x) := \alpha_j e^{-u \alpha_j} (e^{\alpha_j x} - e^{-\alpha_j x})$ for $1\leq j\leq n$, and
\[
S^{(n)}_{i,j} :=
\begin{dcases}
\frac{\alpha_j - \alpha_i}{\alpha_j + \alpha_i}
& \text{for } 1\leq i,j\leq n \, , \\
1
&\text{for $1\leq i\leq n$, $j=n+1$; if $n$ is odd} \, .
\end{dcases}
\]
\end{theorem}
The denominator $\Pf(S^{(n)})$ in~\eqref{eq:rFlatLPP_2} is known as Schur Pfaffian and, no matter the parity of $n$, it satisfies:
\begin{equation}
\label{eq:SchurPf}
\Pf(S^{(n)})
= \prod_{1\leq i<j\leq n} \frac{\alpha_j -\alpha_i}{\alpha_j + \alpha_i} \, .
\end{equation}
\begin{proof}
The proof of~\eqref{eq:rFlatLPP_1} follows the same steps as in the flat case. For the proof of~\eqref{eq:rFlatLPP_2}, one first uses the determinantal form~\eqref{eq:contSp2} of $\sp^{\rm cont}_{\bm{\ga}}$ and formula~\eqref{eq:SchurPf} to show that
\[
\P(\rTau_{2n}\leq u)
= \frac{1}{\Pf(S^{(n)}_{\bm{\alpha}})} \int_{\{0\leq x_1\leq \dots\leq x_n \leq u\}} \det(\phi_j(x_i))_{1\leq i,j\leq n} \prod_{i=1}^n \diff x_i \, ,
\]
where $\phi_j(x) := \alpha_j e^{-u \alpha_j} (e^{\alpha_j x} - e^{-\alpha_j x})$ for $1\leq j\leq n$.
To see that the latter integral is indeed the desired Pfaffian, it suffices to use the following integral identity due to de Bruijn~\cite{Bru55}: if $\nu$ is a Borel measure on $\R$ and $\phi_1,\dots,\phi_n\in L^2(\R,\diff \nu)$, then
\[
\int_{\{x_1\leq\dots\leq x_n\}}
\det\left(\phi_j(x_i)\right)_{1\leq i,j\leq n} \prod_{i=1}^n \nu(\diff x_i)
= \Pf(\Phi^{(n)}) \, ,
\]
where $\Phi^{(n)}$ is a skew-symmetric matrix of order $n$ or $n+1$, according to whether $n$ is even or odd respectively, and is defined by
\[
\Phi^{(n)}_{i,j} :=
\begin{dcases}
\int_{\R^2} \sign(y-x) \phi_i(x) \phi_j(y) \nu(\diff x) \nu(\diff y)
& \text{for } 1\leq i,j\leq n \, , \\
\int_{\R} \phi_i(x) \nu(\diff x)
& \text{for $1\leq i\leq n$, $j=n+1$; if $n$ is odd} \, .
\qedhere
\end{dcases}
\]
\end{proof}

\appendix
\section{}\label{appendixA}

 The integral parametrization for Whittaker functions used in number theory and in by particular by \cite{IS13} is different than ours: in this subsection, we will explain their connection\footnote{In this setting, we strictly stick to the notation of~\cite{IS13}. Accordingly, we remark that the hat in $\hat{W}^A_{n,\bm{a}}$ and $\hat{W}^B_{n,\bm{b}}$ do \emph{not} denote any transform here.}, and then show the equivalence between~\eqref{eq:IshiiStade} and the corresponding integral formulae in~\cite{IS13}. As it will become clear, in both the $\mathfrak{gl}_n$ and the $\mathfrak{so}_{2n+1}$ cases the two parametrizations are linked via the change of variables
\begin{equation}
\label{eq:ishiiStadeChangeOfVars}
y_1 := \frac{1}{\pi} \sqrt{\frac{x_2}{x_1}} 
\, , \quad \dots , \quad
y_{n-1} := \frac{1}{\pi} \sqrt{\frac{x_n}{x_{n-1}}} \, , \quad
y_n := \frac{1}{\pi\sqrt{{x_n}}} \, .
\end{equation}

For $\bm{a}\in\C^n$, set $\abs{a}:= a_1+\dots +a_n$. The $\mathfrak{gl}_n$-Whittaker function $\hat{W}^A_{n,\bm{a}}$ indexed by $\bm{a}\in\C^n$, according to the integral representation~\cite[Prop.~1.2]{IS13}, is defined as follows. For $n=2$,
\begin{equation}
\label{eq:gl_2WhittakerFnIshiiStade}
\hat{W}^A_{2,(a_1,a_2)}(y_1,y_2)
:= 2 y_1^{\abs{a}/2} y_2^{\abs{a}} K_{\frac{a_1 - a_2}{2}}(2\pi y_1) \, ,
\end{equation}
where $K$ is the Macdonald function
\begin{align}\label{macdonald}
K_\nu(x):=\frac{1}{2}\int_{\R_+} z^{\nu} \exp\Big(-\frac{x}{2}\Big(z+\frac{1}{z}\Big)\Big) \,\frac{\dd z}{z}.
\end{align}
 Recursively, for all $n\geq 3$,
\begin{equation}
\label{eq:glWhittakerFnIshiiStade}
\begin{split}
\hat{W}^A_{n,\bm{a}}(\bm{y})
:=\, &\pi^{-\abs{a}/2} \int_{\R_{+}^{n-1}} \hat{W}^A_{n-1,\tilde{\bm{a}}} \bigg(y_2\sqrt{\frac{t_2}{t_1}},\dots, y_{n-1}\sqrt{\frac{t_{n-1}}{t_{n-2}}}, y_n \frac{1}{\sqrt{t_{n-1}}} \bigg) \\
&\times \prod_{j=1}^{n-1} \exp\Big(- (\pi y_j)^2 t_j - \frac{1}{t_j} \Big) (\pi y_j)^{\frac{(n-j)a_1}{n-1}} t_j^{\frac{n a_1}{2(n-1)}} \frac{\diff t_j}{t_j} \, ,
\end{split}
\end{equation}
where $\tilde{\bm{a}}=(\tilde{a}_1,\dots,\tilde{a}_{n-1})$ is defined by $\tilde{a}_i := a_{i+1} + \frac{a_1}{n-1}$.

\begin{proposition}
\label{prop:glWhittakerFnIshiiStade}
If $a_i = 2 \alpha_{n-i+1}$ for $1\leq i\leq n$, and $\bm{x}$ and $\bm{y}$ satisfy~\eqref{eq:ishiiStadeChangeOfVars}, then
\[
\hat{W}^A_{n,\bm{a}}(\bm{y})
= \pi^{-(n+1)\abs{\bm{\alpha}}}
\Psi^{\mathfrak{gl}_n}_{-\bm{\alpha}} (\bm{x}) \, .
\]
\end{proposition}
\begin{proof}
For $n=2$, eq.~\eqref{eq:gl_2WhittakerFnIshiiStade} and the relations defining $\bm{y}$ and $\bm{a}$ in terms of $\bm{x}$ and $\bm{\alpha}$ yield
\[
\hat{W}^A_{2,(a_1,a_2)}(y_1,y_2)
= 2 \pi^{-3\abs{\bm{\alpha}}}
(x_1 x_2)^{-\abs{\bm{\alpha}}/2}
K_{\alpha_2 - \alpha_1}\bigg(2\sqrt{\frac{x_2}{x_1}}\bigg) \, .
\]
On the other hand, \eqref{eq:gl_2WhittakerFn} and~\eqref{macdonald} yield
\[
\Psi^{\mathfrak{gl}_2}_{-(\alpha_1,\alpha_2)} (x_1,x_2)
= 2 (x_1 x_2)^{-\abs{\bm{\alpha}}/2}
K_{\alpha_2 - \alpha_1}\bigg(2\sqrt{\frac{x_2}{x_1}}\bigg) \, ,
\]
proving the desired identity for $n=2$. Assume now that the result holds for $n-1$. In light of~\eqref{eq:ishiiStadeChangeOfVars} and after the change of variables $u_j = t_j x_{j+1}$ for $1\leq j\leq n-1$ in the integral~\eqref{eq:glWhittakerFnIshiiStade}, we obtain
\[
\begin{split}
\hat{W}^A_{n,\bm{a}}(\bm{y})
= \,& \pi^{-\abs{\bm{a}}/2} \int_{\R_{+}^{n-1}} \hat{W}^A_{n-1,\tilde{\bm{a}}} \bigg(\frac{1}{\pi}\sqrt{\frac{u_2}{u_1}},\dots, \frac{1}{\pi}\sqrt{\frac{u_{n-1}}{u_{n-2}}}, \frac{1}{\pi\sqrt{{u_{n-1}}}} \bigg) \\
&\times \prod_{i=1}^n x_i^{-\frac{a_1}{2}}
\prod_{j=1}^{n-1}
\exp\Big( -\frac{u_j}{x_j} - \frac{x_{j+1}}{u_j}\Big)
u_j^{\frac{n a_1}{2(n-1)}}
\frac{\diff u_j}{u_j} \, .
\end{split}
\]
Using the induction hypothesis and the property stated in~\eqref{eq:glWhittakerFnTranslation}, it is easy to see that
\[
\hat{W}^A_{n-1,\tilde{\bm{a}}} \bigg(\frac{1}{\pi}\sqrt{\frac{u_2}{u_1}},\dots, \frac{1}{\pi}\sqrt{\frac{u_{n-1}}{u_{n-2}}}, \frac{1}{\pi\sqrt{{u_{n-1}}}} \bigg)
= \pi^{-n \abs{\bm{\alpha}}}
\bigg(\prod_{j=1}^{n-1} u_j\bigg)^{-\frac{\alpha_n}{n-1}}
\Psi^{\mathfrak{gl}_{n-1}}_{-(\alpha_1,\dots,\alpha_{n-1})}(\bm{u}) \, .
\]
It follows that
\[
\begin{split}
\hat{W}^A_{n,\bm{a}}(y_1,\dots,y_n)
= \,& \pi^{-(n+1)\abs{\bm{\alpha}}}
\int_{\R_{+}^{n-1}}
\Psi^{\mathfrak{gl}_{n-1}}_{-(\alpha_1,\dots,\alpha_{n-1})}(\bm{u}) \\
&\qquad\times\bigg(\frac{\prod_{i=1}^n x_i}{\prod_{i=1}^{n-1}u_i}\bigg)^{-\alpha_n}
\prod_{j=1}^{n-1}
\exp\Big( -\frac{u_j}{x_j} - \frac{x_{j+1}}{u_j}\Big)
\frac{\diff u_j}{u_j} \, .
\end{split}
\]
Using the recursive relation~\eqref{eq:glWhittakerRecurs}, we get the desired identity for $n$.
\end{proof}

The $\mathfrak{so}_{2n+1}$-Whittaker function $\hat{W}^B_{n,\bm{b}}$ indexed by $\bm{b}\in\C^n$, according to the integral representation~\cite[Prop.~1.3]{IS13}, is defined as follows. For $n=1$,
\begin{equation}
\label{eq:so_3WhittakerFnIshiiStade}
\hat{W}^B_{1,b_1}(y_1)
:= 2 K_{b_1}(2\pi y_1) \, .
\end{equation}
Recursively, for all $n\geq 2$,
\begin{equation}
\label{eq:soWhittakerFnIshiiStade}
\begin{split}
\hat{W}^B_{n,\bm{b}}(\bm{y})
:=\, &\int_{\R_{+}^n} \int_{\R_{+}^{n-1}}
\hat{W}^B_{n-1,\tilde{\bm{b}}} \bigg(y_2\sqrt{\frac{t_2 s_2}{t_3 s_1}},\dots, y_{n-1}\sqrt{\frac{t_{n-1} s_{n-1}}{t_n s_{n-2}}}, y_n \sqrt{\frac{t_n}{s_{n-1}}} \bigg) \\
&\times \prod_{j=1}^{n-1} \bigg[\exp\Big(- (\pi y_j)^2 \frac{t_j}{t_{j+1}} s_j - \frac{1}{s_j} \Big)
(t_{j+1} s_j)^{\frac{b_n}{2}} \frac{\diff s_j}{s_j} \bigg] \\
&\times t_1^{b_n} \prod_{j=1}^n \bigg[
\exp\Big(-(\pi y_j)^2 t_j - \frac{1}{t_j}\Big)
(\pi y_j)^{b_n} \frac{\diff t_j}{t_j} \bigg] \, ,
\end{split}
\end{equation}
where $\tilde{\bm{b}}=(b_1,\dots,b_{n-1})$.

\begin{proposition}
\label{prop:soWhittakerFnIshiiStade}
If $b_i = 2 \beta_i$ for $1\leq i\leq n$, and $\bm{x}, \bm{y}$ satisfy~\eqref{eq:ishiiStadeChangeOfVars}, then
\[
\hat{W}^B_{n,\bm{b}}(\bm{y})
= \Psi^{\mathfrak{so}_{2n+1}}_{\bm{\beta}} (\bm{x}) \, .
\]
\end{proposition}
\begin{proof}
For $n=1$, we have indeed
\[
\hat{W}^B_{1,b_1}(y_1)
= 2 K_{2\beta_1}\bigg(\frac{2}{\sqrt{x_1}}\bigg)
= \Psi^{\mathfrak{so}_3}_{\beta_1} (x_1) \, .
\]
Here, the first equality follows from~\eqref{eq:so_3WhittakerFnIshiiStade} and the relations defining $\bm{y}$ and $\bm{b}$ in terms of $\bm{x}$ and $\bm{\beta}$, whereas the second equality is deduced by combining \eqref{eq:so_3WhittakerFn} and~\eqref{macdonald}.

Assume now that the result holds for $n-1$. In light of~\eqref{eq:ishiiStadeChangeOfVars} and after the changes of variables $v_j = x_j / t_j$ for $1\leq j\leq n$ and $u_j = v_{j+1} s_j$ for $1\leq j\leq n-1$ in the integral~\eqref{eq:soWhittakerFnIshiiStade}, we obtain
\[
\begin{split}
\hat{W}^B_{n,\bm{b}}(\bm{y})
= \,& \int_{\R_{+}^n} \int_{\R_{+}^{n-1}}
\hat{W}^B_{n-1,\tilde{\bm{b}}} \bigg(\frac{1}{\pi}\sqrt{\frac{u_2}{u_1}},\dots, \frac{1}{\pi}\sqrt{\frac{u_{n-1}}{u_{n-2}}}, \frac{1}{\pi\sqrt{{u_{n-1}}}} \bigg)  \,\bigg( \frac{\prod_{j=1}^n v_j^2}{\prod_{j=1}^n x_j \prod_{j=1}^{n-1} u_j} \bigg)^{-\frac{b_n}{2}}\\
&\times
\prod_{j=1}^{n-1} \bigg[\exp\Big(- \frac{u_j}{v_j} -\frac{v_{j+1}}{u_j} \Big) \frac{\diff u_j}{u_j} \bigg]
\prod_{j=1}^n \bigg[ \exp\Big(-\frac{x_{j+1}}{v_j} - \frac{v_j}{x_j}\Big)
\frac{\diff v_j}{v_j} \bigg]
 \, .
\end{split}
\]
Using the induction hypothesis and the fact that $\bm{b}=2\bm{\beta}$, we see that the latter expression coincides with $\Psi^{\mathfrak{so}_{2n+1}}_{(\beta_1,\dots,\beta_{n-1},-\beta_n)}(\bm{x})$ (see recursive formula~\eqref{eq:soWhittakerRecurs}), which in turn equals $\Psi^{\mathfrak{so}_{2n+1}}_{\bm{\beta}} (\bm{x})$ due to the invariance of the Whittaker functions under the action of the Weyl group on the
parameters $(\gb_1,...,\gb_n)$; in this situation this amounts to invariance under permutations and multiplication by $\pm1$.
\end{proof}

Now, the integral formula we are interested in is stated in~\cite[Thm.~3.2]{IS13}:
\[
2^n \int_{\R_{+}^n} \bigg(\prod_{j=1}^n y_j^j\bigg)^s
\hat{W}^A_{n,\bm{a}}(\bm{y}) \hat{W}^B_{n,\bm{b}}(\bm{y})
\prod_{j=1}^n \frac{\diff y_j}{y_j}
= \frac{\prod_{1\leq i,j\leq n}
\Gamma_{\mathbf{R}}(s+a_i + b_j)
\Gamma_{\mathbf{R}}(s+a_i-b_j)}
{\prod_{1\leq i<j\leq n} \Gamma_{\mathbf{R}}(2s+a_i+a_j)} \, ,
\]
where $\Gamma_{\mathbf{R}}(z) := \pi^{-z/2} \Gamma(z/2)$. Using the change of variables~\eqref{eq:ishiiStadeChangeOfVars} and Propositions~\ref{prop:glWhittakerFnIshiiStade} and~\ref{prop:soWhittakerFnIshiiStade}, the above formula can be easily rewritten as
\[
\int_{\R_{+}^n}
\bigg(\prod_{i=1}^n x_i\bigg)^{-s/2}
\Psi_{-\bm{\alpha}}^{\mathfrak{gl}_n}(\bm{x})
\Psi_{\bm{\beta}}^{\mathfrak{so}_{2n+1}}(\bm{x})
\prod_{i=1}^n \frac{\diff x_i}{x_i}
= \frac{\prod_{1\leq i,j\leq n}
\Gamma(s/2+\alpha_i + \beta_j)
\Gamma(s/2 + \alpha_i - \beta_j)}
{\prod_{1\leq i<j\leq n} \Gamma(s + \alpha_i+\alpha_j)} \, .
\]
Theorem~\ref{thm:IshiiStade} now follows by taking $s=0$. Note that, in turn, the latter identity can be deduced by Theorem~\ref{thm:IshiiStade}: indeed, the term
$(\prod_{i=1}^n x_i)^{-s/2}
\Psi_{-\bm{\alpha}}^{\mathfrak{gl}_n}(\bm{x})$
is itself a $\mathfrak{gl}_n$-Whittaker function as a whole, because of~\eqref{eq:glWhittakerFnTranslation}.

\vskip 4mm
{\bf Acknowledgements}: We would like to express our gratitude to T.~Ishii and E.~Stade for the many communications and explanations
on their work and on Whittaker functions, including 
the relation between orthogonal Whittaker functions and symplectic Schur functions via the Casselman-Shalika formula,
 mentioned in the introduction. We also thank I. Nteka for useful discussions.
 NZ thanks Y.~Sakellaridis for several long and very useful discussions. 
The work of NZ was supported by EPSRC via grant EP/L012154/1.
The work of EB was supported by EPSRC via grant EP/M506679/1.

\end{document}